\documentclass{imsart}

\RequirePackage{amsthm,amsmath,amsfonts,amssymb}
\RequirePackage{amsthm,amsmath,amsfonts,amssymb}
\RequirePackage[numbers,sort&compress]{natbib}
\RequirePackage[colorlinks,citecolor=blue,urlcolor=blue]{hyperref}
\RequirePackage{graphicx}

\usepackage[singlelinecheck=false]{caption}

\theoremstyle{plain}

\newtheorem{theorem}{Theorem}[section]
\newtheorem{lemma}[theorem]{Lemma}
\newtheorem{corollary}[theorem]{Corollary}
\newtheorem{cor}[theorem]{Corollary}
\newtheorem{proposition}[theorem]{Proposition}

\newtheorem*{theorem*}{Theorem}

\theoremstyle{remark}
\newtheorem{definition}[theorem]{Definition}

\usepackage{epstopdf}
\newcommand{\1}{{\bf 1}}
\newcommand{\Lf}{{L_f}}
\newcommand{\X}{\text {PAP}}

\newcommand{\spath}{{ \omega}}
\newcommand{\sloc}{s_{\spath}}
\newcommand{\possible}{\scwpplus}
\newcommand{\walk}{{{\omega}}}

\newcommand{\petrov}{{\text{Petrov}}}
\newcommand{\cone}{{\text{Weyl}}}
\newcommand{\mat}{{\text{Mat}}}

\newcommand{\prob}{{\mathbb P}}

\newcommand{\N}{{\mathbb N}}
\newcommand{\Z}{{\mathbb Z}}

\newcommand{\spacen}{\Omega_n}

\renewcommand{\P}{\mathbb{P}}
\newcommand{\E}{\mathbb{E}}
\renewcommand{\H}{\mathcal{H}}
\newcommand{\C}{\mathcal{C}}
\newcommand{\bu}{\mathbf{u}}

\newcommand{\cL}{\mathcal{L}}
\newcommand{\coup}{{L^\dagger}}
\newcommand{\zwalk}{{s}}

\newcommand{\R}{\mathbb{R}}
\newcommand{\ip}[2]{\left\langle #1, #2\right\rangle}

\newcommand{\cw}{CW}
\newcommand{\cwplus}{CW^+}
\newcommand{\cwpplus}{CW^{++}}
\newcommand{\cwminus}{CW^-}
\newcommand{\scwminus}{SCW^-}
\newcommand{\scwplus}{SCW^+}
\newcommand{\scwpplus}{SCW^{++}}

\newcommand{\msum}{\tilde M}
\newcommand{\Kstar}{\cone_{n^{.5-2\delta}}}
\newcommand{\Kstarred}{\cone_{n^{.5-\delta}}}
\newcommand{\snd}{\avn(\rho_d)}
\newcommand{\avn}{A\!v_n}
\newcommand{\av}{A\!v}
\newcommand{\posx}{\text{pos}_a}
\newcommand{\posy}{\text{pos}_b}
\newcommand{\ctx}{\text{count}_a}
\newcommand{\cty}{\text{count}_b}
\newcommand{\diff}{\text{diff}}

\newcommand{\pos}{\text{pos}}
\newcommand{\dis}{\text{Dis}}

\endlocaldefs

\begin{document}

 \begin{frontmatter}
\title{Scaling limits of permutations avoiding long decreasing sequences}
%\title{A sample article title with some additional note\thanksref{t1}}
%\runtitle{Scaling limits of permutations avoiding long decreasing sequences}
%\thankstext{T1}{A sample additional note to the title.}

\begin{aug}
%%%%%%%%%%%%%%%%%%%%%%%%%%%%%%%%%%%%%%%%%%%%%%%
%% Only one address is permitted per author. %%
%% Only division, organization and e-mail is %%
%% included in the address.                  %%
%% Additional information can be included in %%
\author[A]{\fnms{Christopher}~\snm{Hoffman}\ead[label=e1]{hoffman@math.washington.edu}},
\author[B]{\fnms{Douglas}~\snm{Rizzolo}\ead[label=e2]{drizzolo@udel.edu}}
\author[C]{\fnms{Erik}~\snm{Slivken}\ead[label=e3]{slivkene@uncw.edu}}
%%%%%%%%%%%%%%%%%%%%%%%%%%%%%%%%%%%%%%%%%%%%%%
%% Addresses                                %%
%%%%%%%%%%%%%%%%%%%%%%%%%%%%%%%%%%%%%%%%%%%%%%
\address[A]{
Department of Mathematics, University of Washington, Seattle, WA, 98195\printead[presep={,\ }]{e1}}

\address[B]{
Department of Mathematical Sciences, University of Delaware, Newark, DE, 19716\printead[presep={,\ }]{e2}}

\address[C]{Department,
Department of Mathematics, University of North Carolina Wilmington, NC, 28403\printead[presep={,\ }]{e3}}
\end{aug}

\begin{abstract}
We determine the scaling limit for permutations conditioned to have longest decreasing subsequence of length at most $d$.  These permutations are also said to avoid the pattern $(d+1)d \cdots 2 1$ and they can be written as a union of $d$ increasing subsequences.  We show that these increasing subsequences can be chosen so that, after proper scaling, and centering, they converge in distribution.  As the size of the permutations tends to infinity, the distribution of functions generated by the permutations converges to the eigenvalue process of a traceless $d\times d$ Hermitian Brownian bridge. 
\end{abstract}



\end{frontmatter}

\tableofcontents

\section{Introduction}

In this paper, we consider random permutations without long decreasing subsequences as a model of non-intersecting paths.  It is a classical result that if the longest decreasing subsequence of a permutation $\pi$ has length $d$ then $\pi$ can be written as the union of $d$ increasing subsequences.  The origins of this result are hard to trace, but it goes back at least to \cite{GREENE1974254} where it is already noted as something that is not hard to see.  Our main result is that if $\sigma$ is a uniformly random permutation of $[n]=\{1,2,\dots, n\}$ conditioned on its longest decreasing subsequence having length at most $d$, then these decreasing subsequences can be chosen so that, after linearly interpolating, scaling, and centering, they converge in distribution, as $n$ tends to infinity, to the eigenvalue process of a traceless $d\times d$ Hermitian Brownian bridge.  Our results fall in the intersection of two lines of research that have received significant interest in the recent literature -- properties of random pattern-avoiding permutations and limit theorems for non-intersecting paths.

Let $\mathcal{S}_n$ denote the set of permutations of length $n$.  For $k \leq n$, $\rho\in \mathcal{S}_k$ and $\tau \in \mathcal{S}_n$ we say $\tau$ contains the pattern $\rho$ if there exists $1\leq i_1<i_2<\cdots i_k\leq n$ such that for $1\leq r< s\leq k$, $\tau(i_r) < \tau(i_s)$ if and only if $\rho(r) < \rho(s)$.  The permutation $\tau$ avoids $\rho$ if it does not contain the pattern $\rho.$  We denote the subset of $\mathcal{S}_n$ that avoids all permutations in a set $A \subset \mathcal{S}_k$ by $\avn(A).$  Taking $\rho_d = (d+1)d\cdots 21$, the decreasing pattern of length $d+1$, we have that $\avn(\rho_d)$ is the set of permutations of $[n]$ whose longest decreasing subsequence has length at most $d$.

Permutations whose longest decreasing subsequence has length at most $d$ (or whose longest increasing sequence has length at most $d$) have a long history both of being studied directly and of appearing in the study of other mathematical objects.  For example, permutations avoiding the pattern $123$ seem to have first been considered by MacMahon \cite{m}, who showed that they are counted by the Catalan numbers $C_n = \frac{1}{n+1}{2n\choose n}$.  Later, Knuth showed that permutations avoiding any fixed pattern of length three are also counted by the Catalan numbers, as are the number of rectangular standard Young tableaux with $2$ rows and $2n$ boxes.  Regev \cite{regev} used the RSK correspondence to give an asymptotic formula for the cardinality of $\avn(\rho_d)$ for $d\geq 2$.  Also using RSK, Novak \cite{Novak11} extended Knuth's result in an asymptotic sense to show that for any $d$, the cardinality of $\av_{dn}(\rho_d)$ is asymptotically equal to the rectangular standard Young tableaux with $d$ rows and $dn$ boxes.  Permutations whose longest decreasing subsequence has length at most $d$ have also appeared in random matrix theory and integrable probability.  For example, \cite{Rains98} shows that if $n>d$, then the $2n$'th moment of the trace of random $d$-dimensional unitary matrix equals the number of permutations of $[n]$ whose longest increasing subsequence is at most $d$.  Further connections to integrals over classical groups are established in, for example, \cite{BR01}.  In \cite{Forrester_2001} it was shown that the number of configurations in certain random-turns vicious walker models with $d$ walks and $n$ steps was equal to the number of permutations of $[n]$ whose longest increasing subsequence is at most $d$. These results generally rely on the RSK algorithm giving a bijection from these permutations to pairs of Young tableaux with the same shape with at most $d$ rows.

Recently there has been significant interest in understanding, for a fixed set of patterns $A$, the behavior of a uniformly random element of the set $\avn(A) \subset \mathcal{S}_n$ as $n\to \infty$, see e.g.\ \cite{MR2343720,bassino2017, bassino2016brownian, borga2018localsubclose, borga2019square,  Borga2019,notknuth, hoffman2017pattern,  HRS1, HRS2,  Ja14,  Ja321, Ja_multiple,   maazoun, madras2010random, madras_monotone, ML, madras_yildirim,   mansour_yildirim,  mp, mrs} for a sample of the available results.  Most of this literature focuses on permutations avoiding short patterns and significant attention has been given to understanding uniformly random elements of $\avn(123)$ and $\avn(321)$, see e.g.\ \cite{notknuth,hoffman2017pattern, HRS1, HRS2, Ja321,mp}.  Early work in this area was motivated by studying the longest increasing subsequence problem for pattern avoiding permutations \cite{MR2343720,notknuth,Sniady}.  The longest increasing subsequence of a uniformly random element of $\avn(\sigma)$ when $\sigma$ is a permutation of length $3$ was studied in \cite{notknuth} using exact enumeration methods. The main result of \cite{Sniady} shows that with appropriate centering and scaling the limit of the shape of the Young tableaux obtained by applying RSK to a uniformly random element of $\avn(\rho_d)$ is given by the eigenvalues of a $d\times d$ GUE matrix conditioned to have trace $0$.  Similar results for the shape of the tableaux obtained by applying RSK to a random word were obtained in \cite{TracyWidom,JohanssonDiscreteOrthogonal}.  A common feature of all of these results is the combinatorial nature of the analysis, relying on generating functions, bijections with Dyck paths, or colored trees, or similar well understood combinatorial structures.  In contrast, our methods here rely on an approximate bijection (in a sense made clear in our analysis) that allows us to closely couple the graph of a uniformly random permutation whose longest decreasing subsequence has length at most $d$ with graph of a bridge of a random walk in $\R^d$ conditioned to remain in a certain cone.  Using this coupling, we are able to leverage recent results on scaling limits of random walks in cones \cite{denisov2015random} to establish our results. 

Non-intersecting paths and their connections to random matrices, which go back to Dyson \cite{Dyson62}, have also featured prominently in the physics, combinatorics, and probability literature, see e.g. \cite{Baik2000, Fisher1984, Gorin2015, Guttmann_1998, Johansson2002, oconnell2002}.  Non-intersecting Brownian bridges specifically have arisen in several contexts \cite{Corwin2014,MR2849479, Johansson2013, nguyen2017}.  Random matrices conditioned to have trace equal to $0$ have also previously appeared in the literature \cite{Biane09, JohanssonDiscreteOrthogonal, MR2363394, Sniady}.  In our case, because we are working with Gaussian processes, conditioning to have trace equal to $0$ can be easily thought of as projection, which allows for the transfer of many results.  In these applications, the models of non-intersecting paths that have been studied have an integrable structure and the ability to analyze the exact formulas that come out of this plays a central role.  In contrast, the model of non-intersecting paths that comes from random permutations without long decreasing subsequences is not known to be integrable.  The non-intersecting paths derived from a permutation $\sigma$ are essentially an emergent phenomenon.  It is easy to see them in simulations for large $n$, but the increasing subsequences that $\sigma$ divides into have disjoint domains, so it is not obvious what it means for them not to intersect.

\section{Main Results}
\label{function} 
%{\bf The $k$ in this section is $d$ in the later sections. It is probably easier to change it all to $d$s.}

\subsection{From permutations to functions}
Every $\sigma \in \snd$  defines a $d$-tuple of non-intersecting functions functions on $C([0,1])$ as follows. First we note that each element in $\sigma \in \snd$ gives a natural partition of $[n]$ into $d$ sets, $\{A^i(\sigma)\}_{i\in [d]}$ as follows. 
Define 
$$A^1(\sigma)=\{i:\  \not \exists \ i'<i \ \text{ such that } \ \sigma(i')>\sigma(i)\}$$
and for $1<i \leq d$
$$A^i(\sigma)=\{i:\  \not \exists \ i'<i \ \text{ and $i' \not \in \bigcup_{j<i}A^j(\sigma)$ such that } \ \sigma(i')>\sigma(i)\}.$$
Thus the sequence $A^1(\sigma)$ consists of all $i$ which are the left right maxima of $\sigma$.
The sequence $A^2(\sigma)$ consists of all $i$ which are the left right maxima of $\sigma$ after removing the elements $(i,\sigma(i))$ with $i \in A^1(\sigma)$, etc.

Next define $d$ sequences 
$$\alpha^l(\sigma)=\{(i,\sigma(i))\}_{i \in A^l(\sigma)}$$
for $l\in [d]$.  These sequences give a unique way to construct pairs of words $\omega_\sigma \in [d]^n\times [d]^n$ where $\omega_\sigma(i) = (l_1,l_2)$ if $i \in A^{l_1}$ and $\sigma^{-1}(i) \in A^{l_2}$.  The pair of words $\omega_\sigma$ can be seen by projecting the labels of the points $\alpha^l(\sigma)$ either horizontally or vertically (see Figure \ref{proj3}). 

\begin{figure}
\includegraphics[scale=.4]{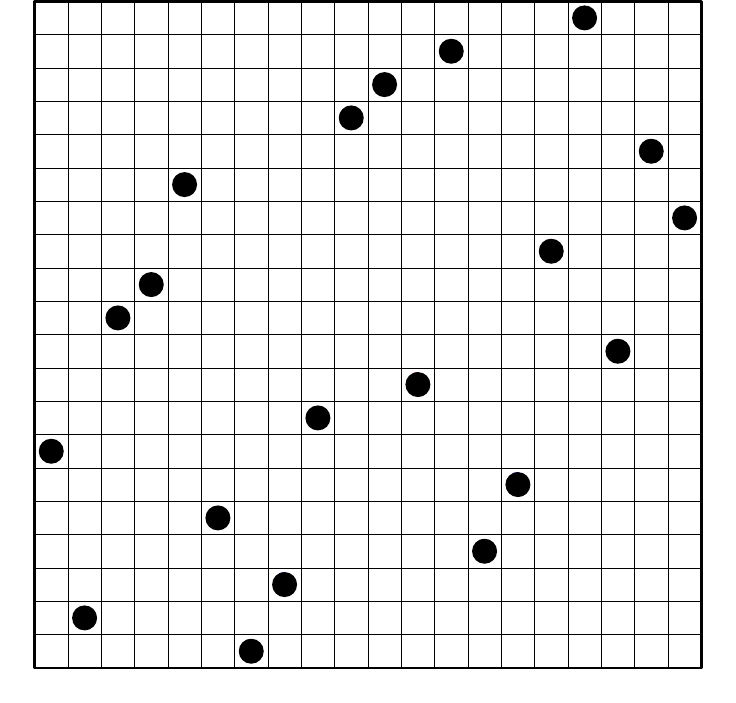}
\hspace{1cm}
\includegraphics[scale=.4]{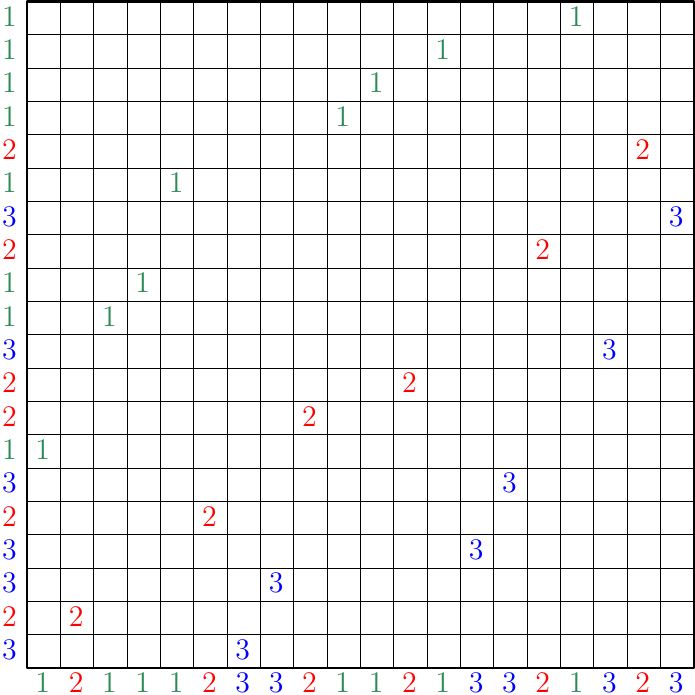}

\caption{On the left is a  permutation $\sigma\in A\!v_{20}(\rho_3)$. On the right we partition the points into three sets. The sequence on the $x$-axis is $\{a(i)\}_{i=1}^n$ and the sequence on the $y$-axis is $\{b(i)\}_{i=1}^n$.}
\label{proj3}
\end{figure}

%\begin{figure}
%\includegraphics[scale=.35]{321big10000.png}
%\hspace{1cm}
%\includegraphics[scale=.35]{321bigpath10000.png}
%
%\caption{On the left is a permutation $\sigma\in A\!v_{100000}(\rho_2)$. On the right are the paths with corresponding functions
%$f(\alpha^1(\sigma))$ and $f(\alpha^2(\sigma))$.}
%\label{proj2path}
%\end{figure}

For any sequence $\alpha=\{(a(i),b(i))\}_{i=1}^m$ and $n$ with $1 \leq a(1)<a(2)< \dots <a(m) \leq n$ we can form a continuous function $f(\alpha)$ on $[0,1]$ by linearly interpolating at constant speed between the points $(0,0)$, $(1,0)$ and
\[\left\{ \left(\frac{a(i)}{n+1},\frac{b(i)-a(i)}{\sqrt{2dn}} \right)\right\}_{i=1}^m.\]
For $\sigma \in \snd$ we take the $d$ sequences $\{\alpha^l(\omega)\}_{l \in [d]}$, and form
\[P_\sigma=\bigg(f(\alpha^1(\sigma)),\cdots,f(\alpha^d(\sigma))\bigg).\]
If follows from our definition of $A^i$ that $f(\alpha^1(\sigma)) \geq f(\alpha^2(\sigma)) \geq \cdots \geq f(\alpha^d(\sigma))$, so that $P_\sigma$ is a family of non-intersecting paths (See Figure \ref{perm and path}).  Our main result is an invariance principle for these paths.  

\begin{figure}
\includegraphics[scale=.4]{label_proj3}
\hspace{.8cm}
\includegraphics[scale=.4]{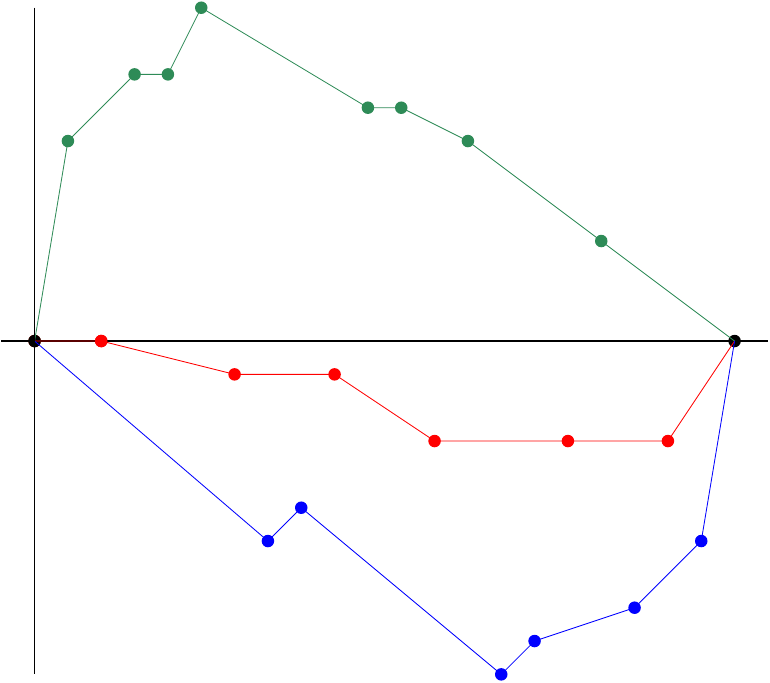}
\caption{On the left is a labeled plot of a permutation $\sigma\in A\!v_{20}(\rho_3)$. On the right is the corresponding collection of paths 
$$P_\sigma=(f(\alpha^1(\sigma)), f(\alpha^2(\sigma)),f(\alpha^3(\sigma)))$$
with $f(\alpha^1(\sigma))$ in green, $f(\alpha^2(\sigma))$ in red 
and $f(\alpha^3(\sigma))$ in blue.}
\label{perm and path}	
\end{figure}

\subsection{Traceless Dyson Brownian bridge}
In order to state our main result formally we need to introduce the limiting object, which is the process ranked eigenvalues of a $d$ by  $d$ Hermitian Brownian bridge conditioned to have trace equal to $0$ for all time.  

%\begin{figure}
%\includegraphics[scale=.4]{av_100000_d6boxed.png} \hspace{2cm} \includegraphics[scale=.25]{ex_100000_d6axes.png}
%\caption{A permutation in $\sigma \in \avn(\rho_6)$ and the corresponding scaled process $P_\sigma.$}	
%\label{perm and path}
%\end{figure}

Let $\{Z_{ii}(t)\}_{i=1}^{d}$ be standard Brownian bridges on the time interval $[0,1]$ conditioned so that $\sum_{i=1}^{d}Z_{ii}=0$.
Let $\{Z_{ij}\}_{1\leq i <j \leq d}$ be independent standard complex Brownian bridges (i.e.\ $\sqrt{2}Re(Z_{ij})$ and $\sqrt{2}Im(Z_{ij})$ are independent standard Brownian bridges).  Finally let $Z_{ji} = \bar{Z}_{ij}.$
We use these random variables to define the following Hermitian matrix valued process:
\[
Z(t)= (Z_{ij}(t))_{1\leq i,j \leq d}.
\]
For a Hermitian matrix $M$, we let $\lambda_1(M) \geq \lambda_2(M) \geq\cdots \geq \lambda_d(M)$ be the eigenvalues of $M$ ranked in non-increasing order. Furthermore, we define
\[ \Lambda(M)=(\lambda_1(M),\lambda_2(M),\cdots,\lambda_d(M)).\]
Notice that 
$$\sum_{i=1}^{d}\lambda_i(Z_t)=\sum_{i=1}^{d}Z_{ii}(t)=0.$$
for all $t \in [0,1]$.    

The process $\Lambda(Z) = (\Lambda(Z(t)), 0\leq t\leq 1)$ is what appears as the limiting object in our main result. We sometimes refer to it as the traceless Dyson Brownian bridge.

\begin{theorem} \label{overwrite} If $\sigma$ is a uniformly random elements of $\snd$ then, as $n\to\infty$, the following convergence holds in distribution with respect to the supremum norm topology on $C([0,1],\R^d)$: 
$$P_\sigma \ \xrightarrow{\text{dist}} \  \Lambda(Z).$$
\end{theorem}
%{\bf Why do we have the $(\cdot)$ here? We don't seem to have it anywhere else.}

This is proved in Section \ref{sec main proof}. Note that the map from $\sigma$ to $P_\sigma$ is not canonical. (For instance
we could have defined the functions by the orthogonal distance to the diagonal instead of the vertical distance.) It is easy to modify our result to prove an appropriate limit theorem for the modified function.  Figure \ref{exampski} gives an example of a large permutation $\sigma \in \avn(654321)$ and the corresponding $P_\sigma.$

\begin{figure}
\centering
\includegraphics[scale=.4]{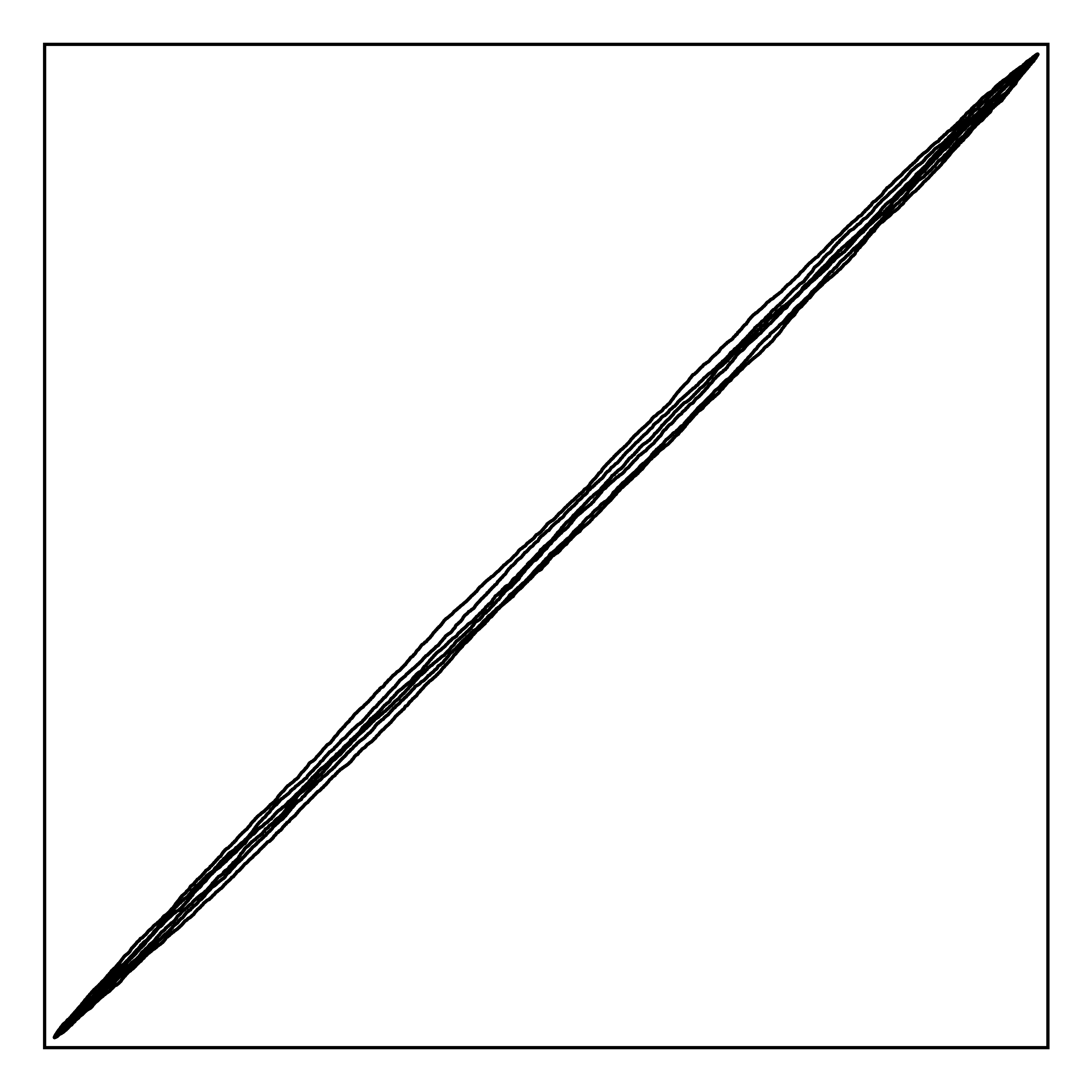}
\hspace{1cm}
\includegraphics[scale=.25]{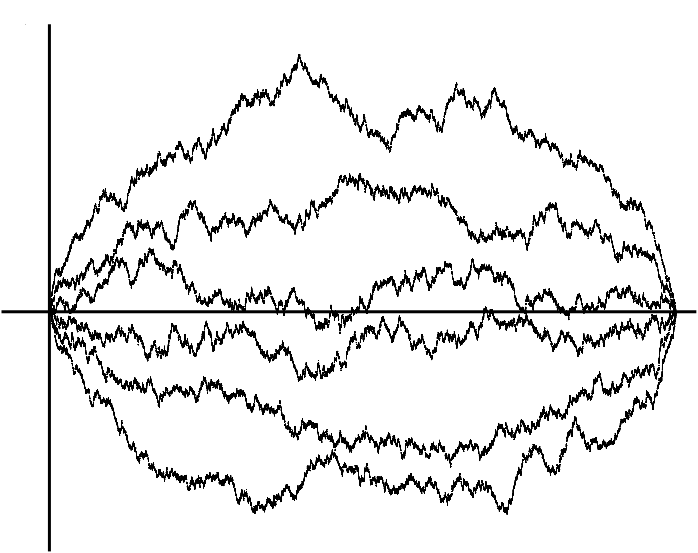}
\caption{A permutation $\sigma\in A\!v_{100000}$ generated uniformly at random and the corresponding collection of scaled paths $P_\sigma$.}
\label{exampski}	
\end{figure}

\subsection{Connection to previous results}
Theorem \ref{overwrite} can be viewed as an extension of Theorem 1.2 of \cite{HRS1} as follows. That theorem considered the case of $d=2$.  Let $(\mathbf{e}_t)_{t\in [0,1]}$ be standard Brownian excursion.
It proves that
\begin{itemize}
\item $\sqrt{2}f(\alpha^1(\sigma)) \ \xrightarrow{\text{dist}}\ (\mathbf{e}_t)$  and
\item $f(\alpha^1(\sigma)) +f(\alpha^2(\sigma))   \ \xrightarrow{\text{dist}}\ 0.$ 
\end{itemize}
The convergence is in the supremum norm topology on $C([0,1],\R)$.

To derive this result from Theorem \ref{overwrite} we let $(B_j(t))_{t\in [0,1]}$ be independent standard Brownian bridges.  Then
\[ Z =_d \frac{1}{\sqrt{2}}  \begin{pmatrix} B_1 & B_2+iB_3  \\ B_2-iB_3& -B_1 \end{pmatrix}.\]
A direct computation shows that
\[ \lambda_1(Z) \stackrel{d}{=} \frac{1}{\sqrt{2}}\sqrt{B_1^2+B_2^2+B_3^2} \quad \textrm{and} \quad  \lambda_2(Z) \stackrel{d}{=} -\frac{1}{\sqrt{2}}\sqrt{B_1^2+B_2^2+B_3^2}.\]
Using the identity in law between  Brownian excursion and the $3$-dimensional Bessel bridge \cite[Chapter XII]{RevuzYor} shows that $(\lambda_1(Z), \lambda_2(Z)) \stackrel{d}{=} (2^{-1/2}\mathbf{e}_t,  -2^{-1/2} \mathbf{e}_t)_{t\in [0,1]} $, where $ (\mathbf{e}_t)_{t\in [0,1]}$ is a standard Brownian excursion.  This is the conclusion of  \cite[Theorem 1.2]{HRS1} (with a slightly different normalization).  See Figure \ref{excursion} for an example.

\begin{figure}
\includegraphics[scale=.25]{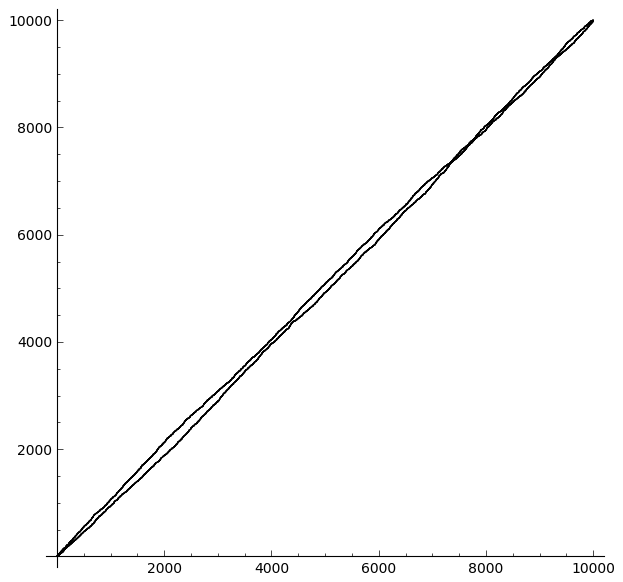}
\hspace{1cm}
\includegraphics[scale=.25]{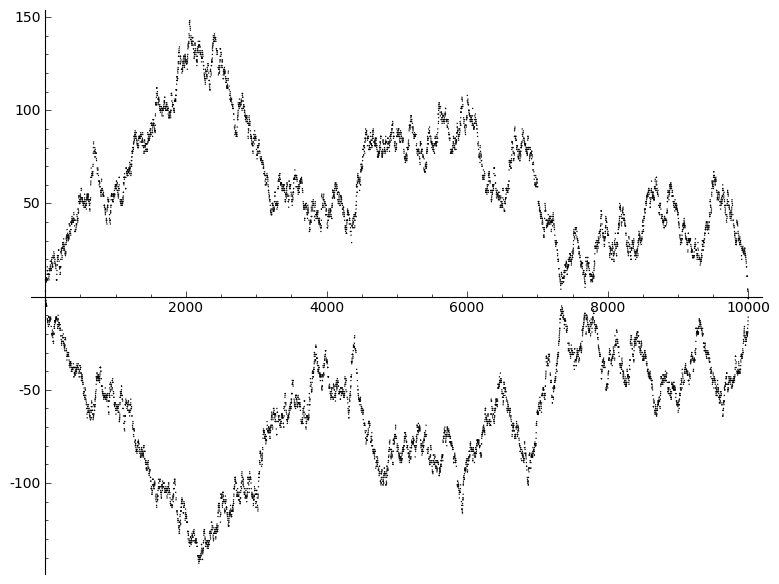}
\caption{On the left is a permutation $\sigma$ in $A\!v_{10000}(321)$.On the right are the functions $f(\alpha^1(\sigma))$ and $f(\alpha^2(\sigma)).$ They are approximated by a Brownian excursion $\mathbf{e}_t$ and $-\mathbf{e}_t$, respectively. }
\label{excursion}	
\end{figure}

\section{Outline}

%\begin{figure}
%\includegraphics[scale=.55]{unlabel_proj3}
%\hspace{1cm}
%\includegraphics[scale=.55]{label_proj3}
%
%\vspace{1.5cm}
%
%\includegraphics[scale=.54]{fluc_walk.pdf}
%\hspace{1cm}
%\includegraphics[scale=.54]{walk_walk.pdf}
%
%\caption{In the upper left is a  permutation $\sigma\in A\!v_{20}(\rho_3)$. In the upper right we partition the permutation into three increasing subsequences, which we call $A^i(\sigma)$. The sequence on the $x$-axis is $\{a(i)\}_{i=1}^n$, the sequence on the $y$-axis is $\{b(i)\}_{i=1}^n$. We combine these two sequences to form 
%$\walk_\sigma =(a,b)$. In the lower left we have the coordinate projections of the function $P_\sigma$. We show that the scaling limit of this function is the traceless Dyson Brownian bridge. In the lower right are the coordinate projections of the function $s_{\walk_\sigma}$. This is the function that we compare with the random walk $S(t)$ on the Weyl Chamber.}
%\label{proj3_redux}
%\end{figure}

In Section \ref{function} we showed one way to take a permutation $\sigma \in \snd$ to get an $\R^d$ valued function $P_\sigma$ on $[0,1]$. In 
Section \ref{divest} we will show a different way to map $\snd$ to $\R^d$ valued functions on $[0,1]$.  In particular, we first take $\sigma \in \snd$ and map it to a path $s_{\omega_\sigma}$ on $\Z^d$ and then space time scaling pushes this forward to an $\R^d$ valued function $\hat s_{\omega_\sigma}$ on $[0,1]$.  If $\sigma \in \snd$ is uniformly random, we will analyze $s_{\omega_\sigma}$ using techniques developed to study random walks in cones.  We are interested in a cone in $\Z^d$ that, by a slight abuse of terminology, we call the Weyl Chamber
$$\cone = \{(x_1,\cdots,x_d) \in \Z^d: x_1 \geq \cdots \geq x_d\}.$$
Ideally we would have liked our paper to have consisted of the following steps.  

\begin{itemize}
\item Find a random walk $S$ such that if $\sigma \in \snd$ is uniformly random then $s_{\omega_\sigma}$ is distributed like $(S_i)_{i=1}^n$ conditioned to stay in the Weyl Chamber and start and end at the origin.
\item Prove that for every $\sigma \in \snd$ the functions $P_\sigma$ and  $\hat s_{\omega_\sigma}$ are close in the supremum norm.
\item Prove that the scaling limit of a bridge of the random walk $S$ conditioned to remain in the Weyl Chamber is distributed like the eigenvalues of a traceless Dyson Brownian bridge. 
\item Combine these three statements to prove our main theorem. 
\end{itemize}

Unfortunately the claims in the first two bullet points above cannot be accomplished.  Fortunately some appropriate modification of each of these steps is achievable. 
And these modifications are sufficiently strong to allow us to prove Theorem \ref{overwrite}.

In Section \ref{divest} we show how to take $\sigma \in \snd$ and map it to a path $s_{\omega_\sigma}$ on $\Z^d$.  We also introduce a random walk $S(t)$ such that $s_{\omega_\sigma}$ is in the range of $S(t)$ and discuss some typical properties of the paths of $s_{\omega_\sigma}$.  %We refer to paths in the range of $S(t)$ as random walk paths or paths. {\bf Do we ever refer to these as random walk paths?}
 
 In Section \ref{wcc} we do much of the work connecting pattern avoiding permutations and the paths of random walks.
First in Lemma \ref{permsRgreat} we define a subset of $\sigma \in \snd$ such that $s_{\omega_\sigma}$ spends most of the time in the Weyl Chamber.
Next in Lemma \ref{middleman} we define a subset of paths in the Weyl Chamber that start and end at the origin where for each $s_\omega$ in this subset, there is a $\sigma \in \snd$ such that $s_{\omega_\sigma}(m)=s_\omega(m)$ for most $m$. Finally in Lemma \ref{tahini} we define a set of $\sigma \in \snd$ such that the functions $P_\sigma$ and  $\hat s_{\omega_\sigma}$ are close in the supremum norm.

In Section \ref{peaches} 
we show that the size of the above subsets is $1-\epsilon$ times the size of the respective spaces.
Thus we have the adaptations of the first three bullet points. We also show how to combine these results with
the scaling limit of our random walk in the Weyl Chamber to prove Theorem \ref{overwrite}.

Sections  \ref{minuseleven}, \ref{walksincones} and \ref{birs} are auxiliary sections.
In Section  \ref{minuseleven} we define some technical lemmas about random walks close to the Weyl Chamber that will be used in Section \ref{peaches}. These lemmas are proven in Section \ref{birs}.
In Section \ref{walksincones} we adapt the previous literature to show that
the scaling limit of our random walk in the Weyl Chamber is distributed like the eigenvalues of a traceless Hermitian Brownian bridge. The proofs in Section \ref{walksincones} are independent of the results in the rest of the paper.

\section{Notation} \label{divest}

\begin{figure}
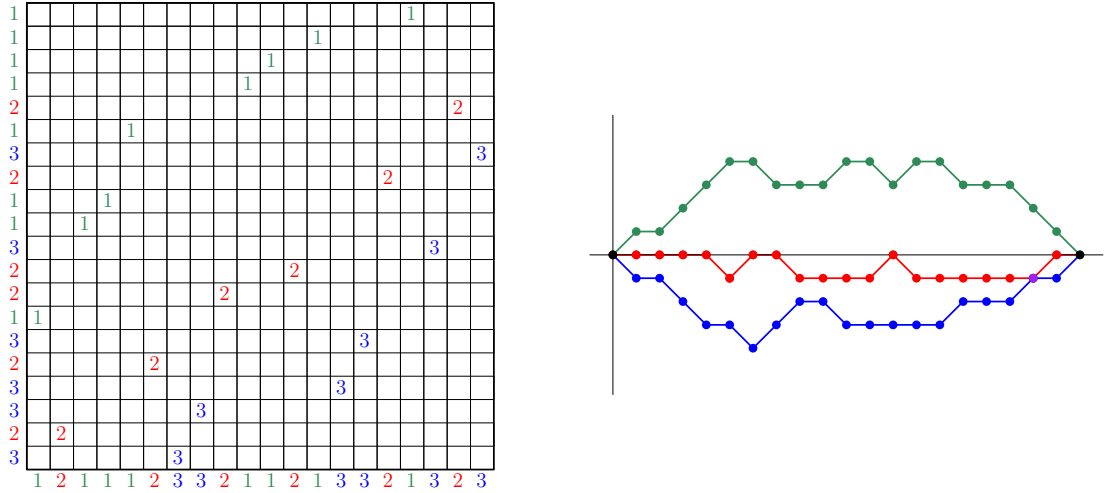

\includegraphics[scale=.45]{walk_perm}
\hspace{1cm}
\includegraphics[scale=.45]{walk_walk}
   \caption{On the left is permutation \(\sigma\in A\!v_{20}(\rho_3)\) with corresponding projections.  
On the right are the coordinate projections of 
\[s_{\walk_\sigma}(i)=\sum_{j=1}^i e_{a(j)}-e_{b(j)}.\]}
\label{proj3_redux_redux}
\end{figure}

\subsection{Paths on $\Z^d$}
Let $e_l$ be the $d$-dimensional vector with a one in the $l$th coordinate and zero everywhere else.  For the norm on $\Z^d$, we use the $L^1$ norm.  

\begin{definition}
Let $I$ be a connected subset of $\N$. A {\it path} is a function $s: I \to \Z^d$ where $$s(t+1)-s(t) \in \{e_i-e_j\}_{1\leq i,j \leq d}$$ for all $t$ such that $t, t+1 \in I$ and
$$s(t) \in \left\{(z_1,\cdots , z_d ) \in \Z^d \ : \ \sum_{i=1}^d  z_i=0\right\}$$
for all $t \in I$.
\end{definition}
%Unless otherwise specified, we will take $k=0$. {\bf Is this necessary?}
As $s(t+1)-s(t) \in  \{e_i-e_j\}_{1 \leq i,j \leq d}$, the sum of the coordinates of a path $s$ is constant as a function of $t$.  We will be primarily interested in paths taking values in the codimension $1$ subpace  $\left\{(z_1,\cdots , z_d ) \in \Z^d \ : \ \sum_{i=1}^d  z_i=0\right\}$ of $\R^n$.

We consider a path $s$ to be a lattice path on a lattice whose vertices are the points of $\Z^d$ and whose edges are given by the relation $x\sim y$ if $x-y \in \{e_i-e_j\}_{1 \leq i,j \leq d}$.  In particular, we define the boundaries of sets relative to this lattice.

\begin{definition}
For $A\subseteq \Z^d$, the boundary of $A$ is 
\[\partial A = \left\{x\in A : x+e_i-e_j\in A^c \textrm{ for some } 1\leq i,j\leq d\right\}.\]
\end{definition}

%A path can take $d^2$ possible steps $s(t+1)-s(t)$ (counted with multiplicity) of the form, $$\{e_i-e_j\}_{1\leq i,j \leq d},$$ 
%so it is restricted to the co-dimension one subspace $$\left\{(x_1,\cdots , x_d ) \in \Z^d \ : \ \sum_{i=1}^d 
%x_i=0\right\}.$$

\subsection{Defining our probability space}

Let $\Omega_{\N}=[d]^{\N} \times [d]^{\N} $ and let $\P = \mu^\N\times \mu^\N$ be the product measure on $\Omega_{\N}$, where $\mu$ is the uniform distribution on $[d]$. This will be our probability space for much of the paper. (We will also consider uniform distribution on a number of subsets of permutations.) For $\omega \in \Omega_{\N}$ we write $\omega =(a,b)$ where $a=(a(t))_{t=1}^\infty$ is the projection onto the first sequence and $b=(b(t))_{t=1}^\infty$ is the projection onto the second sequence.  When it will not cause confusion, we will use $\omega$ both for an arbitrary element in $\Omega_{\N}$ and the canonical random variable given by the identity map on $(\Omega_{\N}, \P)$.  We also let $\omega(t) = (a(t),b(t))$.  Observe that if $\omega$ is distributed like $\P$ then the random variables $\{\omega(t)\}_{t \in \N}$ are i.i.d.\ with 
$$\prob(\omega(t)=(i,j))=1/d^2$$ for all $i,j \in [d].$

For $\omega=(a,b) \in \Omega_\N$ we define the path $s_\omega$ by 
\[ s_\omega(i) =\sum_{j=1}^i e_{a(j)}-e_{b(j)}.\]

Note that if $\omega$ has distribution $\P$, then $\sloc$ is a lazy random walk such that $\sloc(t+1)-\sloc(t)=0$ with probability $1/d$ and $\sloc(t+1)-\sloc(t)=e_j-e_k$ with probability $1/d^2$ for each $ i \neq j$ with $1 \leq i,j \leq d$.

For a given starting time $t \in \N$ and $x \in \Z^d$ we use the notation $\P_{(t,x)}$ for the conditional probability on the set $\{\omega \in \Omega_{\N} \ : \ s_{\omega}(t)=x\}$ and $\E_{(t,x)}$ for the expectation with respect to $\P_{(t,x)}$. This can also be used when $t$ is a stopping time.  For the remainder of the paper, we will implicitly restrict to $(t,x)$ such that $\P(\sloc(t)=x)>0$.

%First we describe the map from $\snd$ to a path in $\Z^d$. We have already described how to map $\sigma \in \snd$ to an element of $[d]^n \times [d]^n$.  

Given a finite set $A$ a function $g$ on $A$ we use the notation
$(g(a) \ :\ a \in A)$ to denote the empirical distribution of the collection.  That is,  
\begin{equation}\label{eq empi} (g(a) \ :\ a \in A)=\frac{1}{|A|}\sum_{a \in A} \delta_{{g(a)}}\end{equation}
where  $\delta_{g(a)}$ is a point mass at $g(a)$ and $|A|$ is the cardinality of $A$.  This is used because we will frequently need to couple empirical distributions on parts of paths defined in terms of permutations with those from parts of paths defined in terms of random walks (see, e.g, for Corollary \ref{Elsa}).%This will be used in Sections \ref{} and \ref{}, see e.g. Corollary \ref{Elsa}, where it will be convenient to treat paths 
\subsection{$\cw((0,0),(n,0))$}
%{\bf I think that all the $i$ and $j$ subscripts here should be parenthesis. Fuck.}
For $\omega \in \Omega_\N$ with $\omega(i)=(a(i),b(i))$  and $l \in [d]$
define $$\ctx^l(m)=\sum_{j=1}^{m} \1_{a(j)=l}, \ \ \ \ \text{ and } \ \ \ \ 
\cty^l(m)=\sum_{j=1}^{m} \1_{b(j)=l}.  $$
Observe that the $l$th coordinate of $s_\omega(m)$ is 
\begin{equation} \label{tunnel mountain}
\diff^l(m)=\ctx^l(m)-\cty^l(m).
\end{equation}

Let
\begin{equation}\label{byeweek}
\spacen =\big\{(a,b) \in  [d]^n \times [d]^n: \diff^l(n)=0\  \forall \ l \in [d]\big\}. 
\end{equation}

For any $n$ we can, with a slight abuse of notation, think of $\Omega_n$ as the subset of $\Omega_{\N}$ given by all infinite sequences that extend a finite sequence in $\Omega_n$. Similarly we can have a set of $\Omega_{\N}$ that is defined by only finitely many coordinates and think of this as a finite object.

\begin{definition} \label{63yards}
We define $\cw((i,v),(j,w))$ to be the set of $\omega \in \Omega_{\N}$ such that $s_{\omega}(i)=v$, $s_{\omega}(j)=w$ and $s_{\omega}(\ell) \in \cone$ for $i\leq \ell \leq j$. 
\end{definition}

%If $W\in [d]^n \times [d]^n$ is chosen uniformly at random then $s_W$ is the random walk that remains in place with probability $1/d$ and takes each of the other $d(d-1)$ possible steps with probability $1/d^2$, all independently of its previous steps.
%Note that this map from $[d]^n\times[d]^n$ to paths in $\Z^d$ is not one to one.
%(For instance if $d=n=2$ then there are four choices of $\walk$ such that $s_\walk$ is identically zero.)

%Based on this we define a random walk $S(t)$ on $\Z^d$. We let $S(t)$ be the random walk that remains in place with probability $1/d$ and takes each of the other $d(d-1)$ possible steps with probability $1/d^2$, all independently of its previous steps. For a given starting time $T \in \Z$ and $x \in \Z^d$ the state space is all infinite paths with
%$S(T)=x$. We use the notation $\P_{(T,x)}$ to define the probability measure on this space and $\E_{(T,x)}$ to define integration against this measure.

\subsection{The function $P_\omega$}
For a sequence $a$, and an interval $I\subset [n]$ with $I = [i',j')$ we let $\ctx^l(I)= \sum_{i\in I} \1_{a(i) = l}$ and define $\cty^l(I)$ similarly for the sequence $b$.    
For $i>0$ define the function $\posx^l(i) := \inf\{t: \ctx^l(t) = i\}$ if $i \leq \ctx^l(n)$ and $\posx^l(i):=n$ otherwise.   Similarly define $\posy^l(i)$.  

Note that any $\omega \in \spacen$ defines $d$ increasing sequences $\alpha^1, \cdots \alpha^d$, in the following way. Let 
$$\alpha^l(m)=(\posx^l(m),\posy^l(m)).$$

The condition that all the $\diff^l(n)=0$ ensures that the sum of the lengths of the sequences is $n$. 
Using the function $f(\alpha)$ defined at the beginning of section 2 we convert these $d$ sequences into piecewise linear functions.  In this way any element $\omega \in \spacen$ generates a $\R^d$ valued function $P_\omega$.

\subsection{Petrov conditions} \label{pc}

We now describe a family of events that are moderate deviation conditions on $\Omega_{\N}$.
We refer to these as the Petrov conditions and they play a critical role in the paper.

%$[d]^n\times [d]^n$ which is indexed by 1 through $n$. We refer to these as the Petrov conditions.
%We are typically interested in elements of $[d]^n\times [d]^n$ which satisfy all but a small number of the Petrov Conditions. These conditions play a critical role in the paper. For instance 
%these conditions allow us to identify a set of
%$\sigma \in \snd$ such that the two $\R^d$-valued functions $P_\sigma$ and  $\hat s_{\omega_\sigma}$ are close in the supremum norm.

%\newpage
\begin{definition}[Petrov Conditions] \label{seven eleven}
Fix $m \in \N$.  For $\omega \in \Omega_{\N}$, we say $\omega$ has property $\petrov(m)$ if the following properties are satisfied for each $l\in [d]$:\\

\noindent For all intervals $[i,j] \subseteq [0,m]$ with $|j-i| > m^{.1}$
\begin{enumerate}
	\item $|\ctx^l(j) - \ctx^l(i) - \frac{1}{d}(j-i) | < (2d)^{-2}(j-i)^{.6},$
	\item $|\cty^l(j) - \cty^l(i) - \frac{1}{d}(j-i)| < (2d)^{-2}(j-i)^{.6}.$
\end{enumerate}
For all intervals $[i,j] \subseteq [0,m]$ with $|j-i| < m^{.4}$
\begin{enumerate}
	\item $|\ctx^l(j) - \ctx^l(i) - \frac{1}{d}(j-i) | < (2d)^{-2}m^{.25},$
	\item $|\cty^l(j) - \cty^l(i) - \frac{1}{d}(j-i) | < (2d)^{-2}m^{.25}.$
\end{enumerate}

\end{definition}

Note that any interval of length less than $\kappa$ alone cannot cause $\petrov(m)$ to not be satisfied.  The following lemma extends the Petrov conditions to the functions $\posx^l$ and $\posy^l$.

\begin{lemma}\label{extendo}
Suppose $\omega \in \Omega_{\N}$ has property $\petrov(m)$.  Let $[i,j]\subseteq[0,m]$ with\\ $ \max\{\posx^l(j), \posy^l(j)\} \leq m$.  If $|j-i|>m^{.1}$, then
\begin{enumerate}
	\item $\left|\posx^l(j) - \posx^l(i) - d(j-i)\right| < (2d)^{-1}(j-i)^{.6},$
	\item $|\posy^l(j) - \posy^l(i) - d(j-i)| < (2d)^{-1}(j-i)^{.6}.$
\end{enumerate}
If $|j-i|<m^{.4}$, then 
\begin{enumerate}
	\item $\left|\posx^l(j) - \posx^l(i) - d(j-i)\right| < (2d)^{-1}m^{.25},$
	\item $|\posy^l(j) - \posy^l(i) - d(j-i)| < (2d)^{-1}m^{.25}.$
\end{enumerate}	
\end{lemma}

\begin{proof} 
Let $0 \leq i<j \leq m$ satisfy $s = \posx^l(i)$ and $t = \posx^l(j) \leq m$.  Then $\ctx^l(s) = i$ and $\ctx^l(t) =j$.  If $j-i>m^{.1}$ then $t-s>m^{.1}$ and by the Petrov conditions,
\begin{equation}\label{brain}
	|\ctx^l(t) - \ctx^l(s) - d^{-1}(t-s)| < (2d)^{-2}(t-s)^{.6}.
\end{equation} 
By the Triangle Inequality
\begin{equation*}
	d^{-1}(t-s) - (\ctx^l(t) - \ctx^l(s)) < (2d)^{-2}(t-s)
\end{equation*}
or
\begin{equation*}
(t-s)(1- (2d)^{-1}) < d(\ctx^l(t) - \ctx^l(s))	
\end{equation*}
giving, in terms of $i$ and $j$,
\begin{equation}\label{pinky}
\posx^l(j) - \posx^l(i) <  2d(j-i).	
\end{equation}
Rewriting inequality \eqref{brain} in terms of $i$ and $j$ gives
\begin{equation}\label{ralph}
	|(j-i) - d^{-1}(\posx^l(j) - \posx^l(i))| < (2d)^{-2}(\pos^l(j) - \posx^l(i))^{.6}.
\end{equation}
Then, by \eqref{pinky}, we have 
$$|(j-i) - d^{-1}(\posx^l(j)-\posx^l(i))| < (2d)^{-2}|2d(j-i)|^{.6} < (2d)^{-1}|j-i|^{.6}.$$
The exact same argument works for $\posy^l$, and a similar argument works when $j-i < m^{.4}$, finishing the proof.
\end{proof}
Throughout this paper we assume the results of Lemma \ref{extendo} when we cite the Petrov conditions. 

Fix $n \in \N$. For a pair of sequences $\omega = (a,b)$, let $\omega^*=(a^*,b^*)$ denote the pair given by the reverse of the first $n$ elements of the two sequences.  That is $a^*(i)=a(n+1-i)$ and  $b^*(i)=b(n+1-i)$.  Similarly for a path $s = \{s(i)\}_{i=0}^n$ on $\Z^d$ of length $n$ let $s^*$ denote the reverse of the path $\{s(n-i)\}_{i=0}^n.$ Although it does not matter we can set $\omega^*(m)=\omega(n)$ for all $m>n$.

\begin{definition} \label{eight ten}
We say $\omega$ has property $\petrov^*(m)$ if $\omega^*$ has property $\petrov(m).$  
\end{definition}

\begin{lemma}\label{conditionedpetrov}  \label{pathpetrov}
There exists $\gamma> 0$ and $C$  such that for all $t<m$ and $x \in \Z^d$
$$\prob_{(t,x)}(\petrov(m)^C) \leq C e^{-m^{\gamma}}.$$
$$\prob_{(t,x)}(\petrov(m)^C \ | \ \Omega_n) \leq C e^{-m^{\gamma}}.$$
\end{lemma}

\begin{proof}
Standard moderate deviation estimates imply that there exists $\gamma'$ such that
the set of $(a,b) \in [d]^n \times [d]^n$ in $\petrov^C(m)$ is bounded by $C e^{-m^{\gamma}}.$ See Lemma 2.4 of \cite{hoffman2017pattern} for more details.

We get a similar bound when we condition on $(a,b) \in \Omega_n$ as follows. As $\prob_{(0,0)}(\Omega_n)\geq n^{-k}$ for some $k$ if $m>n^{.1}$ then we only need to lower $\gamma$. If  $m\leq n^{.1}$ then the local central limit theorem implies
that conditioning on any sequence $\{(a(i),b(i))\}_{i=1}^m$ the probability of being in $\Omega_n$ differs by at most a factor of 2. Thus 
\[\prob_{(t,x)}( \petrov(m)^C \ | \ \Omega_n ) \leq C' e^{-m^{\gamma}}. \qedhere\]
\end{proof}

%\begin{lemma}\label{pathpetrov}
%	
%	Let $S$ be a random walk that is the image of a pair of sequences chosen uniformly in $[d]^n\times[d]^n$.  There exists $C,\gamma>0$, such that the probability $\petrov(m)$ fails is at most $Ce^{-m^{\gamma}}.$	
%\end{lemma}
%
%By definition, if $\petrov(m)$ occurs such that $s_\omega =S$ then $\petrov(m)$ occurs for $S$ itself.  Thus the probability that $\petrov(m)$ does not occur for a walk $S$ is bounded above by the probability that $\petrov(m)$ does not occur for every one of the pairs $\omega$ such that $s_\omega = S,$ and thus by Lemma \ref{realpetrov} is at most $Ce^{-m^{\gamma}}.$

\begin{lemma} \label{considiff}

Let $\omega \in \Omega_{\N}$ satisfy $\petrov(m)$.  The following hold:

\begin{align}
	&|\ctx^l(m) - \cty^l(m) | < d^{-2}m^{.6},\label{doobie}\\
	&|s_{\omega}(m)| < d^{-1}m^{.6},  \label{maximus} 
\end{align}
If $[j,j'] \subset [0,m]$ and $|j-j'|> m^{.3}$ then 
\begin{equation}\label{cardiff}
|s_{\omega}(j')-s_{\omega}(j)|< d^{-1}|j'-j|^{.6}.	
\end{equation}

\end{lemma}

\begin{proof}

If $\omega$ satisfies $\petrov(m)$ then $|\ctx^l(m) - d^{-1}m| < (2d)^{-2}m^{.6}$.  Similarly $|\cty^l(m) - d^{-1}m| < (2d)^{-2}m^{.6}$ and thus by the triangle inequality $$|\ctx^l(m) - \cty^l(m)| \leq |\ctx^l(m) - d^{-1}m|+|\cty^l(m) - d^{-1}m| < d^{-2}m^{.6},$$ proving \eqref{doobie}.  This provides a uniform bound for each of the $d$ coordinates of $s_\omega(m)$ and thus also proves \eqref{maximus}.  

For Equation \eqref{cardiff} we have
\begin{align*}
|s_{\omega}(j') - s_\omega(j) | &\leq d \max_{l\in [d]} | \ctx^l(j')-\cty^l(j') - ( \ctx^l(j) - \cty^l(j)) |	\\
& \leq d \max_{l\in [d]} |  \ctx^l(j') - \ctx^l(j) - ( \cty^l(j') - \cty^l(j)) |\\
&\leq d \left( d^{-1}(j'-j) + (2d)^{-2}|j'-j|^{.6} - d^{-1}(j'-j) + (2d)^{-2}|j'-j|^{.6} \right)\\
& \leq d^{-1}|j'-j|^{.6}.
\end{align*}

\end{proof}

\begin{lemma}\label{fiasco}

Let $\omega \in \Omega_{\N}$ satisfy $\petrov(m)$.  For $l\in [d]$, let $i\leq m$ be such that $\ctx^l(m) = \cty^l(i)$.  Then $m-i < m^{.6}$
	A similar statement holds if $\ctx^l(i) = \cty^l(m)$.
\end{lemma}
\begin{proof}

By Lemma \ref{considiff} $|\ctx^l(m) - \cty^l(m)| < d^{-2}m^{.6}.$ Under our assumptions we may replace $\ctx^l(m)$ with 
   $\cty^l(i)$, giving 
$$|\cty^l(i) - \cty^l(m)| < d^{-2}m^{.6}.$$  
   By $\petrov(m)$ we have 
$$|\cty^l(m) - \cty^l(i) - d^{-1}(m-i)|<(2d)^{-2}m^{.6}.$$  Thus combined we have 
$$|d^{-1}(m-i)| < (2d)^{-2}m^{.6} + d^{-2}m^{.6} < d^{-1}m^{.6}.$$  Multiplying by $d$ finishes the proof.     
 \end{proof}

\begin{lemma}\label{messy messi}

Let $\omega \in \Omega_{\N}$ satisfy $\petrov(m)$.  Then for all $l$
$$
\posx^l(\ctx^l(m)) \in ( m - m^{.19} , m]
$$
and 
$$
\posy^l(\cty^l(m)) \in ( m - m^{.19} , m].
$$

\end{lemma}

\begin{proof}
	
	Consider the intervals $I_1= (m-m^{.3}-m^{.19}, m-m^{.19}]$ and $I_2 = (m-m^{.3}-m^{.19},m]$.  Both intervals have size at least $m^{.3}$ and therefore by the Petrov conditions
\begin{equation*}
|\ctx^l(I_1) - d^{-1}m^{.3}| < (2d)^{-2}m^{.18}	
\end{equation*}
while
\begin{equation*}
|\ctx^l(I_2) - d^{-1}(m^{.3} + m^{.19})| < (2d)^{-2}|m^{.3} + m^{.19}|^{.6} < d^{-2}m^{.18}.
\end{equation*}
These together imply that the interval $I_3 = (m-m^{.19}, m]$ satisfies
\begin{equation} \label{lasty}
	\ctx^l(I_3)  = \ctx^l(I_2) - \ctx^l(I_1) \geq d^{-1}m^{.19} - 2(d)^{-2}m^{.18} > 0.
\end{equation}

The value $\posx^l(\ctx^l(m))$ is the position of the last occurrence of $l$ at or before position $m$.  Inequality \eqref{lasty} shows that there is at least one $l$ in $I_3$ and thus this last  occurrence must occur somewhere in $I_3.$  
The same argument holds for $pos_y^l(\cty^l(m)).$
\end{proof}

\section{From permutations to paths close to the Weyl Chamber}

\label{wcc}

\subsection{Definitions}
Remember that we have defined 
$$\cone = \{(x_1,\cdots,x_d) \in \Z^d: x_1 \geq \cdots \geq x_d\}.$$
Heuristically we think of the path $s_{\walk_\sigma}$ associated with a permutation in $\sigma \in \avn(\rho_d)$ as being a path in $\cone$. But the reality is that it may not necessarily stay within $\cone$. (See Figure \ref{outside_cone}.) And not every 
$\omega$ such that 
 $s_\omega(t) \in \cone$ for all $t$ is the image $s_{\walk_\sigma}$ for some $\sigma \in \snd$.  Because of this we need to consider other family of paths.

We define $\cone_a$ to be $\cone$ shifted so that its apex is at $(dk,(d-1)k,\dots,k,0)$ where $k=[ a ]$ is the integer part of $a$. That is
\begin{equation} \label{divertida}
\begin{split} \cone_a & =(dk,(d-1)k,\dots,k,0)+\cone\\
& = \{(x_1,\cdots,x_d) \in \Z^d: x_i \geq x_{i+1}+k ,\ \forall i=\{1,\dots,d-1\} \}.
\end{split}
\end{equation}
Recall the definition of $\cw((i,v),(j,v'))$ from Definition \ref{63yards}. Based on this we define 
\begin{itemize}
\item $\cwminus((i,v),(j,v'))$ to be all $\omega \in \Omega_{\N}$ such that 
\begin{enumerate}
\item $s_{\omega}(i)=v$, $s_{\omega}(j)=v'$, 
\item $s_{\omega}(m) \in \cone_{m^{.4}}$ for all $m \in [i,j]$ and  
\item for every $m \in [i,j]$, every $l \in [d]$ and every interval $[i',j'] \subset [i+1,m]$ the conditions from the definition of $\petrov(m)$ (Definition \ref{seven eleven}) are satisfied.
\end{enumerate}
\end{itemize}
Note that the definition can be checked by knowing $s_{\omega}(i)$ and $\omega(k)$ for $k \in [i+1,j]$.
This definition is very useful because in Lemma \ref{middleman} we show that if  $s_{\omega}\in \cwminus((t,x),(n-t^*,x^*))$ (for certain choices of $t, t^*, x$ and $x^*$) then it can be extended so that it equals $s_{\walk_\sigma}$ for some $\sigma \in \snd$.
We also define
\begin{itemize}
\item $\cwplus((i,v),(j,v'))$ to be all $\omega \in \Omega_{\N}$ such that 
\begin{enumerate}
\item $s_{\omega}(i)=v$, $s_{\omega}(j)=v'$, and 
\item $s_{\omega}(k)\in \cone_{-m^{.4}}$ for all $k \in [i,j]$,
\end{enumerate}
\item $\cwpplus((i,v),(j,v'))$ to be all $\omega \in \Omega_{\N}$ such that  
\begin{enumerate}
\item $s_{\omega}(i)=v$, $s_{\omega}(j)=v'$, and 
\item for all $m \in [i,j]$ either 
$s_{\omega}(m)\in \cone_{-m^{.4}}$  or 
there exists $l \in [d]$ and an interval $[i',j'] \subset [i,m]$ such that the conditions from the definition of $\petrov(m)$ (Definition \ref{seven eleven}) are not satisfied.
\end{enumerate}
%$\petrov(m)$ fails (by some interval in $[i,j]$).
\end{itemize}
In Lemma \ref{permsRgreat} we will show that if $\sigma \in \snd$ then the associated path $s_{\walk_\sigma} \in \cwminus((0,0),(n,0))$.
To understand $\cone_{\pm i^{.4}}$ we note that if we let $\dis$ be the $L^1$ distance on $\R^{d}$
then we see that 
\begin{enumerate}
\item if $\dis(x,\cone) < i^{.4}$
then $x\in \cone_{-i^{.4}}$ and 
\item  if $\dis(x,\cone) >d^2 i^{.4},$
then $x\not \in \cone_{-i^{.4}}$.
\end{enumerate}
Also we see that 
\begin{enumerate}
\item if $\dis(x,\partial \cone) < i^{.4}$
then $x\in \cone_{i^{.4}}$ and 
\item  if $\dis(x,\partial \cone) >d^2 i^{.4},$
then $x\not \in \cone_{i^{.4}}$.
\end{enumerate}

In all the preceeding notation we can replace $v$ and $v'$ by $\cdot$ which represents taking a union. 
For example $$\cwminus((j,\cdot),(n,\cdot))=\bigcup_{v,v'}\cwminus((j,v),(n,v'))$$ and
$$\cwplus((j,\cdot),(n,v'))=\bigcup_{v}\cwplus((j,v),(n,v')).$$

Fix $n$. We also want to consider symmetric version of these sets. 
By the definition of $\cw$ it is already symmetric. By this we mean that if $\omega \in \cw((0,0),(n,0))$  then so is 
$\omega^*$. 
The corresponding statements are not true for $\omega \in \cwminus((0,0),(n,0))$ or $\omega \in \cwpplus((0,0),(n,0))$.

We define
$$\scwpplus((i,v),(j,v'))$$ to be all paths $\omega$ such that $s_{\omega}(i)=v$, $s_{\omega}(j)=v'$, and for all $m \in [i,j] \cap [0,n/2]$ either  
$s_{\omega}(i)\in \cone_{-m^{.4}}$  or there exists $l \in [d]$ and an interval $[i',j'] \subset [i,m]$ such that the conditions from the definition of $\petrov(m)$ (Definition \ref{seven eleven}) are not satisfied.
Also for all $m \in [i,j] \cap [n/2,n]$ either 
$s_{\omega}(m)\in \cone_{-m^{.4}}$ or there exists $l \in [d]$ and an interval $[i',j'] \subset [i,m]$ such that the conditions from the definition of $\petrov^*(n-m)$ (Definition \ref{seven eleven}) are not satisfied.

%and for all $i \in [n/2,k]$ 
%either $\petrov^*(i)$ fails or
%$$\dis(s_{\omega}(i),\cone) \leq (n-i)^{.4}.$$ 

\begin{figure}
\includegraphics[scale=.75]{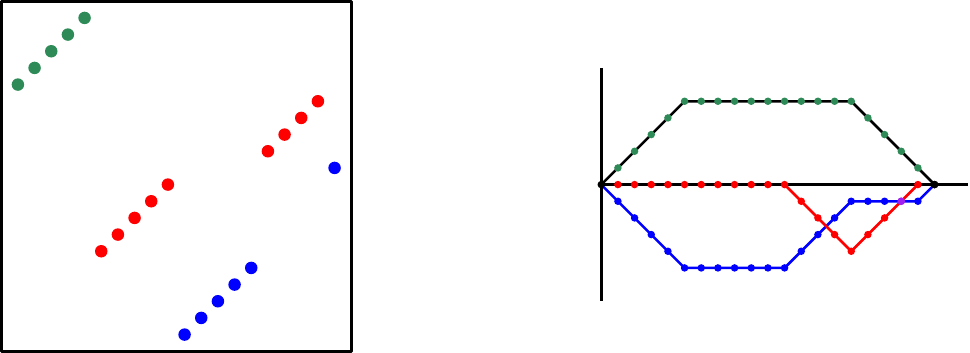}
\caption{A permutation $\sigma \in A\!v_{20}(\rho_3)$ whose associated path $s_{\omega_\sigma}$ is not in $\cone$.}
\label{outside_cone}
\end{figure}

The connection between $\avn(\rho_d)$ and $\scwpplus((0,0),(n,0))$ is summarized in the following lemma.  

\begin{lemma}\label{permsRgreat}
For any $\sigma\in \avn(\rho_d)$, ${\walk_\sigma} \in \scwpplus((0,0),(n,0)).$  	
\end{lemma}

\begin{proof}
 
%{\bf Why can we assume that $i$ is large? The statement is for every $i$.}
%
%{\it Is any of this italic section necessary?
%
%Suppose $\petrov(i)$ occurs and $[j,j']\subseteq [0,i]$ is an interval of length at least $i^{.3}$, so that 
%	$$|\ctx^l(j')-\ctx^l(j) - d^{-1}(j'-j)| < d^{-1}|j'-j|^{.6}$$
%	and
%	$$|\cty^l(j')-\cty^l(j) - d^{-1}(j'-j)| < d^{-1}|j'-j|^{.6}.$$
%
%Lemma \ref{considiff} gives
%\begin{align}
%	&|\ctx^l(i) - \cty^l(i) | < d^{-1}i^{.6},\label{doobie brothers}\\
%	&|s_{\walk_\sigma}(i)| < i^{.6}, \label{stoobie} \\
%	&|s_{\walk_\sigma}(j')-s_{\walk_\sigma}(j)|< |j'-j|^{.6},\label{roobie}
%\end{align}
%while Lemma \ref{messy messi} implies
%\begin{equation*}
%\posx^l(\ctx^l(i)) \in ( i - i^{.19} , i] \text{ and } \posy^l(\cty^l(i)) \in ( i - i^{.19} , i]	.
%\end{equation*}
%
%Let us assume that $\walk = \walk_\sigma$ for some $\sigma\in \avn(\rho_d)$.  For $l< l'$, let $s'>0$ be such that $\posx^{l'}(s') \leq n$ and let $s= \ctx^l(\posx^{l'}(s'))$.  The point $(i,\sigma(i)) = (\posx^l(s),\posy^l(s))$ must lie above and to the left of $(\posx^{l'}(s'),\posy^{l'}(s')).$  In this case $$\posx^l(s) < \posx^{l'}(s')$$ and $$\posy^l(s) > \posy^{l'}(s').$$  
%
%}

We will show that if 
\begin{itemize}
\item $\sigma\in \avn(\rho_d)$, 
\item $\dis(s_{\walk_\sigma}(i),\cone) > i^{.4}$ and 
\item $\petrov(i)$ occurs
\end{itemize}
then we have a contradiction.  For $\walk_\sigma$ to be distance greater than $i^{.4}$ from $\cone$, there must be some $1\leq l < d$ such that,
\begin{equation}\label{farfromcone}	
\ctx^l(i)-\cty^l(i)-\ctx^{l+1}(i) + \cty^{l+1}(i) < -i^{.4}. 
\end{equation}

Let $$j = \posx^{l+1}(\ctx^{l+1}(i)),$$ the closest position of an $l+1$ in $a$ that occurs at or before $i$.  Let $$k = \posx^l(\ctx^l(j)),$$ the closest position of an $l$ that occurs strictly before position $j$ (which is the position of an $l+1$) in $a$.  With these definitions we have 
$$\sigma(k)=\posy^{l}(\ctx^l(k)) > \posy^{l+1}(\ctx^{l+1}(j))=\sigma(j).$$

It is possible that $k$ is the position of the closest $l$ that occurs at or before $i$ in $A$.  However it is also possible that some $l$ occurs between $j$ and $i$.  In either case, $\posx^l(\ctx^l(i))$ is the position of the closest $l$ to $i$ and $\posy^l(\ctx^l(i)) \geq \sigma(k) > \sigma(j)$.  

We have three cases to consider.  Either 
\begin{enumerate}
 \item $\sigma(j) < \sigma(k) \leq \posy^l(\ctx^l(i)) \leq i$ \label{case1},
 \item $\sigma(j)  < i < \posy^l(\ctx^l(i)$, or
 \item $i \leq \sigma(j) < \sigma(k) \leq \posy^{l}(\ctx^l(i))$.
\end{enumerate}
 
 The arguments for the first and last case are exactly the same with a slight change in the definition of $j$ and $k$.  The argument for the middle case requires only a slightly modified approach.  We will proceed by considering the first case (see Figure \ref{ijk}), leaving the other two cases to the reader.    
 %{\bf Where do we use the assumption that we are in case 1. It would be could to point that out.}
\begin{figure}
	\includegraphics[scale=.4]{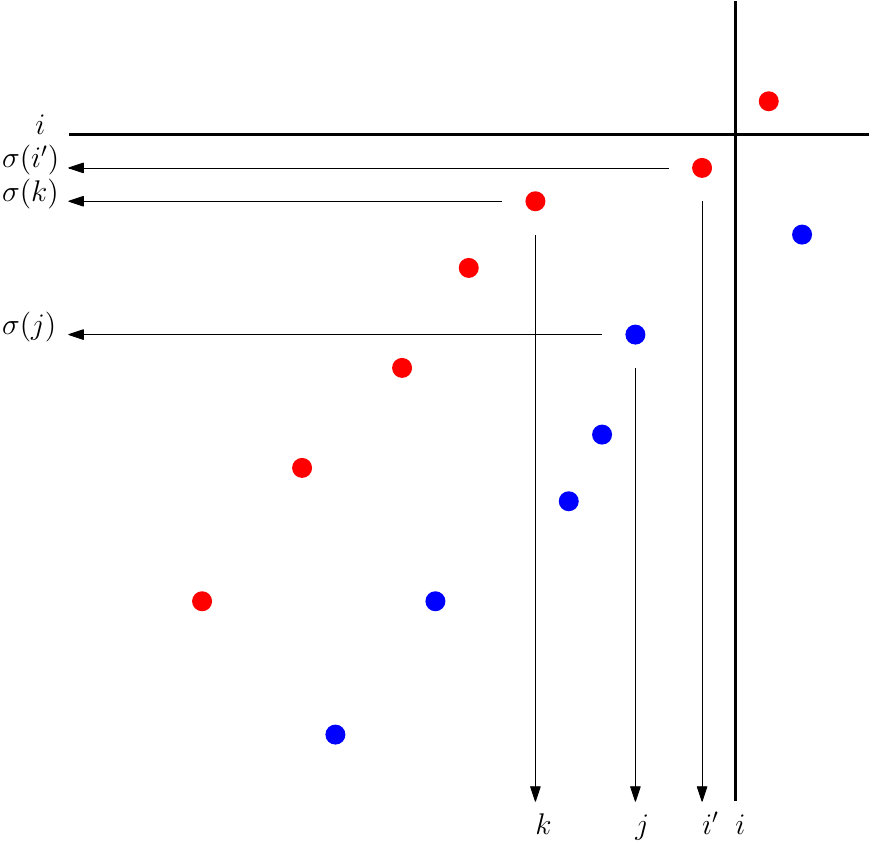}
	\caption{Points from increasing sequences.  Red points correspond to points that project to $l$ and blue points correspond to points that project to $l+1$.  The the last point labeled $l$ that occurs before the vertical line $x=i$ also lies below the horizontal line $y=i$.}
	\label{ijk}
\end{figure} 
 
By Lemma \ref{messy messi} we have 
\begin{equation}\label{close call}
	i-k < 2i^{.19}.
\end{equation}

%Furthermore because our path comes from a permutation in $\avn(\rho_d)$, the first point in $A^l(\sigma)$ must 
%\begin{align}
%\sigma(k) = \posy^{l}(\ctx^{l}(k)) &> \posy^{l+1}( \ctx^{l+1}(i)) = \sigma(j) \label{order2}, \\%{\bf where does this come from?}\\
%\text{and} & \nonumber \\
%\posy^l(\ctx^l(k))&\leq \posy^l(\ctx^l(i)) \label{order3}.% {\bf where does this come from?}
%\end{align}

Lemma \ref{considiff} implies $\ctx^l(i)$ and $\cty^l(i)$ differ by at most ${d}^{-2}i^{.6}$.  We are assuming that $\posy^l(\ctx^l(i)) \leq i$.  Thus by Lemma \ref{extendo} we may use $\petrov(i)$ to claim 
$$%\begin{equation*}
|\posy^l(\ctx^l(i)) - \posy^l(\cty^l(i)) - d(\ctx^l(i)-\cty^l(i)) | $$ 
$$< |\ctx^l(i) - \cty^l(i)|^{.6}\leq \frac{1}{2}i^{.36},$$
%\end{equation*}
and therefore 
\begin{equation} \label{mls cup}
\posy^l(\ctx^l(i) - \posy^l(\cty^l(i)) = d(\ctx^l(i) - \cty^l(i)) + \epsilon^l_i
\end{equation}
where 
$|\epsilon_i^l|  \leq \frac{1}{2}i^{.36}.$
%\begin{equation}
%|\epsilon_i^l|  \leq \frac{1}{2}i^{.36}.
%\end{equation}
Similarly, 
\begin{equation}\label{quartz2}
\posy^{l+1}(\ctx^{l+1}(i)) - \posy^{l+1}(\cty^{l+1}(i)) = d(\ctx^{l+1}(i)-\cty^{l+1}(i)) + \epsilon_i^{l+1},	
\end{equation}
with $|\epsilon_i^{l+1}| \leq \frac{1}{2}i^{.36}$.

Combining \eqref{mls cup} and \eqref{quartz2} with \eqref{farfromcone} we have
\begin{align} 
&\left[\posy^l(\ctx^l(i)) - \posy^l(\cty^l(i))\right] - \left [\posy^{l+1}(\ctx^{l+1}(i)) - \posy^{l+1}(\cty^{l+1}(i))\right] \nonumber\\
&\qquad\qquad= d\left(\ctx^l(i)-\cty^l(i)-\ctx^{l+1}(i) + \cty^{l+1}(i)\right) + \epsilon_i^l + \epsilon_i^{l+1} \nonumber \\
&\qquad\qquad< -di^{.4} + i^{.36} \nonumber \\
&\qquad\qquad< -i^{.4}. \label{horseface}	
\end{align}
On the other hand by Lemma \ref{messy messi}, 
\begin{equation}\label{pete} \text{
$0<i - \posy^l(\cty^l(i)) < i^{.19}$ \ \ and \ \ $0< i - \posy^{l+1}(\cty^{l+1}(i)) < i^{.19}$. } 
\end{equation}
 By construction $\posy^l(\ctx^l(i)) \geq \posy^l(\ctx^l(k)) > \posy^{l+1}(\ctx^{l+1}(i)),$ so along with \eqref{pete} we have
\begin{align} 
	&\left[\posy^l(\ctx^l(i)) - \posy^l(\cty^l(i))\right] - \left[\posy^{l+1}(\ctx^{l+1}(i)) - \posy^{l+1}(\cty^{l+1}(i))\right] \nonumber \\
	& \qquad\qquad\geq \left[\posy^l(\ctx^l(k)) - \posy^{l+1}(\ctx^{l+1}(i))\right] \nonumber\\
	&\qquad\qquad\qquad\qquad + \left[\posy^{l+1}(\cty^{l+1}(i)) - \posy^l(\cty^l(i)) \right] \nonumber\\
	&\qquad \qquad> 0 + \left[ - 2i^{.19} \right].\label{donkey}
\end{align}

Both inequalities \eqref{horseface} and \eqref{donkey} cannot simultaneously be true.  Therefore we can conclude that for $\sigma \in \avn(\rho_d)$ both $\dis(s_{\walk_\sigma}(i),\cone) > i^{.4}$ and $\petrov(i)$ cannot both be true.  A symmetric argument will work using $\petrov^*(i)$ for $i > \lfloor n/2 \rfloor.$  	
\end{proof}

%%%%%%%%%%%%%%%%%%%%%%%%%%%%%%%%%%%%%%%%%%%%%%%%%%%
%%%%%%%%%%%%%%%%%%%%%%%%%%%%%%%%%%%%%%%%%%%%%%%%%%%
%%%%%%%%%%%%%%%%%%%%%%%%%%%%%%%%%%%%%%%%%%%%%%%%%%%
%%%%%%%%%%%%%%%%%%%%%%%%%%%%%%%%%%%%%%%%%%%%%%%%%%%
%%%%%%%%%%%%%%%%%%%%%%%%%%%%%%%%%%%%%%%%%%%%%%%%%%%
%

For any $\omega \in \spacen$ we have the matrix given by 
$$\text{Mat}(\omega)(i,j)=\begin{cases} l &\mbox{if } (i,j) =(\posx^l(m),\posy^l(m)) \text{ for some $m$ and $l$}\\
0 & \mbox{else}. \end{cases}$$
This map $\text{Mat}$ is clearly one-to-one. For notational convenience we often refer to $\spacen$ when we consider the set of matrices which are the image of $\spacen$ under the map $\text{Mat}$.

%For $\walk \in \Omega_n$ recall the matrix $\mat(\walk)$ as a matrix with entries from $\{0,\cdots d\}$, where the $ij$th entry in $\mat(\walk) = 0$ unless for some $l>0$ there exists $t$ such that $\posx^l(t)=i$ and $\posy^l(t) = j$.  

We say the $ij$th entry of $\mat(\walk)$ is \emph{proper} if $\mat(\walk)_{ij} = l$ and 
\begin{itemize}
\item for every $0< l' < l$, there exists $i'<i$ and $j'>j$ such that $\mat(\walk)_{i'j'} = l'$, and 
\item for every $l' \geq l > 0$ and $i'<i$ and $j'>j$, $\mat(\walk)_{i'j'} \neq l'$.
\end{itemize}

We say $\mat(\walk)$ is \emph{proper} if every nonzero $ij$th entry is proper.  See Figure \ref{penguins} for an example of a proper labeling.

\begin{figure}
\includegraphics[scale=.4]{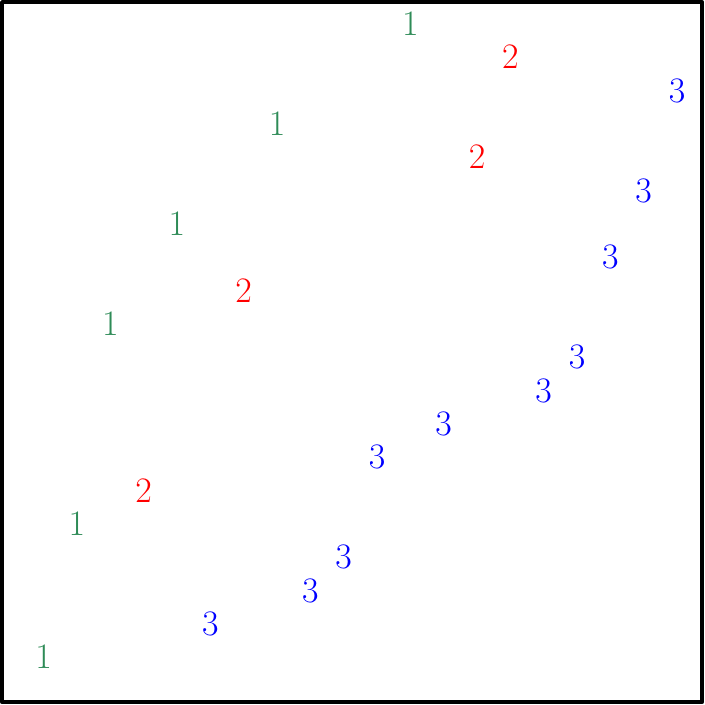}
\hspace{1cm}
\includegraphics[scale=.4]{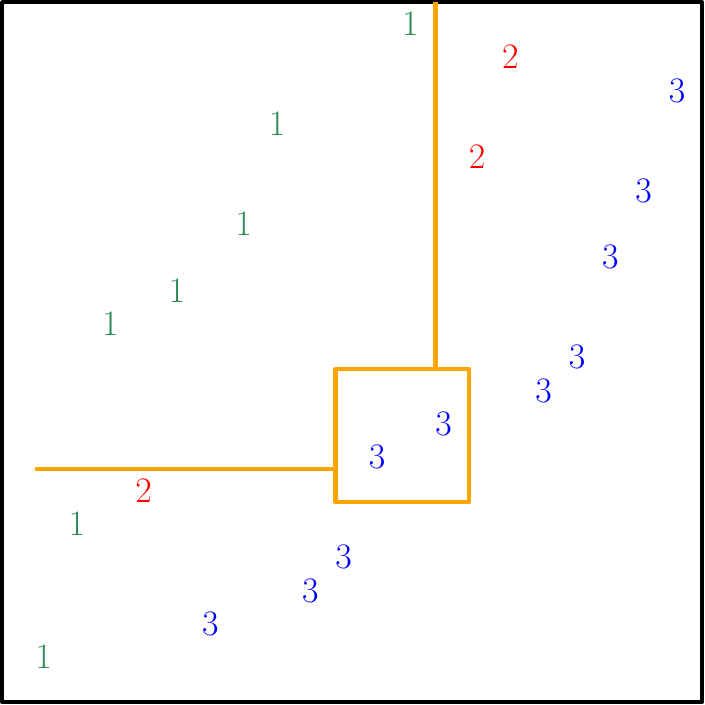}	
\caption{On the left, a properly labeled permutation in $A\!v_{20}(\rho_3)$.  On the right, an improperly labeled permutation in in $A\!v_{20}(\rho_3)$.  If both $3$s in the smaller box were relabeled with $2$s, then the permutation on the right would be properly labeled.}
\label{penguins}
\end{figure}

For $\sigma \in \avn(\rho_d)$, we let $\walk_\sigma \in \Omega_n$ be the pair of sequences given by projection of the non-zero entries of the matrix $\mat(\sigma)$ onto the $x$ and $y$ axis.  Conversely for $\walk \in \Omega_n$ let $\sigma_\walk$ denote the permutation in $\mathcal{S}_n$ where $(i,\sigma_\walk(i))$ is constructed by finding unique values $t$ and $l$ such that $\posx^l(t) = i$ and $\posy^l(t) = \sigma_\walk(i).$  We say $\walk\in \Omega_n$ is \emph{minimal} if and only if there exists a $\sigma\in \avn(\rho_d)$ such that $\walk = \walk_{\sigma}.$

\begin{lemma} \label{iowa}
$\walk\in \Omega_n$ is minimal if and only if $\mat(\walk)$ is proper.  
\end{lemma}

\begin{proof}
This follows from the definition of minimal and the construction of sets $A^i$ in Section \ref{function}.
\end{proof}

For $\walk\in \Omega_n$ and times $t,t^* < \lfloor n/2\rfloor$ we define the decomposition of $\walk$ by 
$$\walk = \walk^1\oplus\walk^2\oplus\walk^3$$ where $\walk^1$ denotes the beginning of $\walk$ until time $t$, $\walk^2$ the portion of $\walk$ from $t$ to $n-t^*$, and $\walk^3$ the portion of $\walk$ from $n-t^*$ to $n$. 
 
%Let $s(t) = x$ and $s_\walk(n-t^*) = y$.  In terms of paths in cones we have for $i \leq t$, $s_\walk(i) = s_{\walk^1}(i)$, for $i \leq t^*$, $s_{\walk^*}(i) = s_{(\walk^3)^*}(i)$ and for $i \in [t,n-t^*]$, $s_{\walk}(i) = x + s_{\walk^2}(i-t).$  

For any $j,k<n/2$ and $x,y\in \Z^d$ define 
 $\scwminus((j,x),(n-k,y))$ to be all paths $\spath$ such that
\begin{enumerate}
\item $\sloc(i) \in \cwminus((j,x),(\lfloor n/2 \rfloor,\cdot)$ and
\item $\sloc(n-i) \in \cwminus((k,y),(\lfloor n/2 \rfloor,\cdot)$
\end{enumerate}
%$\scwplus((j,x),(n-k,y))$ and $\scwpplus((j,x),(n-k,y))$ are defined in an analogous way. 

\begin{lemma}\label{middleman} 

%{\bf Conditions on $L, t, t^*, x$ and $y$}.  

Fix $n$, $L>0$ and let $x,y \in \cone_{2^L}$.  Let $t$ and $t^*$ be bounded by both $4.1^L$ and $n/2$.  Let $\walk \in \scwminus((t,x),(n-t^*,y)).$ 
Let $\sigma \in \avn(\rho_d)$ be a permutation such that $\walk_\sigma$ satisfies $\petrov(t)$ and $\petrov^*(t^*)$, $s_{\walk_\sigma}(t) = x$ and $s_{\walk_{\sigma}}(n-t^*) = y$.  Finally, let $\omega' \in \Omega_n$ be given by  
$$\omega'=\walk^1_\sigma\oplus\walk^2\oplus\walk^3_\sigma.$$  
That is, the pair of sequences whose initial component until position $t$ and final component from position $n-t^*$ until $n$ are both obtained from $\walk_\sigma$ and whose middle component is obtained from $\walk$.  Then there exists $\sigma \in \snd$ such that $\omega'=\walk_\sigma$.   
\end{lemma}

\begin{proof}
We will first show that $\mat(\omega')$ is proper. By Lemma \ref{iowa} this is sufficient to prove the lemma. We work by contradiction and
suppose $M=\mat(\omega')$ is not proper.  Then there exists $1 < l \leq d$ and $(i,j)$ such that $M_{ij} = l$ and all points labeled $l-1$ that occur to the left of $i$ occur below $j$.  

If $(i,j) \in [0,t]^2$ then $M_{ij} = l$ if and only if $\mat(\walk_\sigma)_{ij} = l.$  Thus properness of the entries of $\mat(\walk_\sigma)$ ensures properness of $M_{ij}$ in this range.

Now consider points $(i,j)$ not in $[0,t]^2.$  First we will assume that $t\leq  j\leq n/2$ and $i\leq j$.  If $M_{ij} = l$ then 
$$\ctx^l(i) = \cty^l(j).$$  
If $\sigma_{\omega'} \in \scwminus((t,x),(n-t^*,y))$ then 
\begin{equation}\label{frank ocean}
j^{.4} < \dis(\sigma_{\omega'}(j), \partial \cone) \leq \min_{l'} \{\diff^{l'}(j) - \diff^{l'+1}(j)\}.  
\end{equation}
%By Lemma \ref{fiasco} $|j-i| < 2j^{.6}$

Suppose that no point above and to the left of $(i,j)$ is labeled $l-1$.  This implies that \begin{equation} \label{new hampshire} \ctx^{l-1}(i) = \cty^{l-1}(j)\end{equation} 
since $(i,j)$ is labeled $l$.  Similarly $\ctx^l(i) = \cty^l(j)$, as the point $(i,j)$ is labeled $l$.  By Lemma \ref{fiasco}, if $\petrov(j)$ occurs, then $|j-i| < j^{.6}$ and for all $1\leq l'\leq d$, 
\begin{equation}\label{almost_there}
d^{-1}(j-i) - (2d)^{-2}j^{.36} \leq \ctx^{l'}(j) - \ctx^{l'}(i)  \leq d^{-1}(j-i)+ (2d)^{-2}j^{.36}.
\end{equation}
As $\dis(s_{\omega'}(j) , \cone) > j^{.4}$ 
\begin{align*}
j^{.4} &< \diff^{l-1}(j) - \diff^l(j) \\
& = \ctx^{l-1}(j) - \cty^{l-1}(j) - \left ( 	\ctx^l(j) - \cty^l(j) \right) \\
& = \ctx^{l-1}(j) - \ctx^{l-1}(i) - \left ( \ctx^l(j) - \ctx^l(i)\right) \\
& \leq d^{-1}(j-i) + (2d)^{-2}j^{.36} - (d^{-1}(j-i) - (2d)^{-2}j^{.36}) \\
& \leq j^{.36}.
\end{align*}
This is a contradiction to \eqref{new hampshire} as it implies that if $\petrov(j)$ occurs and $\dis(s_{\omega'}(j),\cone) > j^{.4}$, then $\ctx^{l-1}(i) \neq \cty^{l-1}(j).$  

%Now suppose that there exists some $i'< i$ and $j'>j$ such that $M_{i'j'} = l+1$.  Since the points labeled $l+1$ are increasing we may assume that 
%\begin{equation*}
%\text{$j'= \posy^{l+1}( \cty^{l+1}(j) + 1)$ \ \ \  and \ \ \ $i' = \posx^{l+1}(\cty^{l+1}(j)).$}
%\end{equation*}  
%That is $(i',j')$ that is above line $j$ and by assumption $\posx^{l+1}(\cty^{l+1}(j))<i$.
%
%On the one hand we have from Lemma \ref{messy messi}
%\begin{align}\label{tartelette}
%	&B=\posx^l(\ctx^l(j)) - \posx^l(\cty^l(j)) - \left (\posx^{l+1}(\ctx^{l+1}(j)) - \posx^{l+1}(\cty^{l+1}(j))\right)\\
%	&< j - i - (j-j^{.19} -i) \nonumber \\
%	&= j^{.19}. \nonumber
%\end{align}
%On the other hand, $\petrov(j)$ occurs, so
%\begin{align}
%	&B=\posx^l(\ctx^l(j) - \posx^l(\cty^l(j)) - \left (\posx^{l+1}(\ctx^{l+1}(j)) - \posx^{l+1}(\cty^{l+1}(j))\right)\label{donut} \\
%	& \geq d\left( \ctx^l(j) - \cty^l(j)  - (\ctx^{l+1}(j) - \cty^{l+1}(j)) \right) - 4dj^{.36}\nonumber\\
%	& > dj^{.4} - 4dj^{.36}.\nonumber
%\end{align}
%Combining inequalities \eqref{tartelette} and \eqref{donut} gives the contradiction $j^{.19}> dj^{.4}-4dj^{.36}$.  Thus $M_{ij} = l$ is proper 
%and no points in this region {\bf what region?} can cause $\mat({\omega'})$ to be non-proper.   

A similar argument works if we assume $t\leq j \leq i \leq n/2$.  We may also apply the same argument for $(i,j)$ such that $t^* \leq n-j \leq n/2$ or $t^* \leq n-i \leq n/2$.  
We also must consider the cases when $i\leq n/2 \leq j$ or $j\leq n/2\leq i$.  
Those two cases are very similar to the argument above and we leave the details to the reader.

%%%%Suppose the latter case.  We cannot necessarily use Lemma \ref{fiasco} directly but it can be seen that in this case both $i$ and $j$ must be close enough to $n/2$ that their difference satisfies $|j-i| < Cj^{.6}$ for some constant $C$.  This slight modification does not effect the rest of the proof.
The only points not covered are those in $[0,t]\times [n-t^*,n]$ or $[n-t^*,n]\times [0,t]$.  The conditions $\petrov(t)$ and $\petrov^*(t^*)$ guarantee that all points in these regions are labeled $0$ and thus cannot make $M({\omega'})$ non-proper.  Then we may conclude that $\mat({\omega'})$ is proper and therefore ${\omega'}$ is minimal. Thus Lemma \ref{iowa} implies there exists $\sigma \in \snd$ such that ${\omega'}=\omega_\sigma$. 
\end{proof}

For any function $s:[n] \to \Z^d$ let $\hat s$ be the scaled and linearly interpolated function on $[0,1]$ given by $\hat s(t) = \frac{s(\lfloor nt \rfloor)}{\sqrt{2n/d}}$.
 The following lemma gives conditions such that $\hat{s}_{\walk_\sigma}(t)$ and $P_\sigma(t)$ are close in the sup norm. (See Figure \ref{compare} for a comparison).
Recall that $\dis$ is the $L^1$ distance on $\R^{d}$.
\begin{figure}
\includegraphics[scale=.45]{fluc_walk.pdf} \hspace{.25cm}\includegraphics[scale=.45]{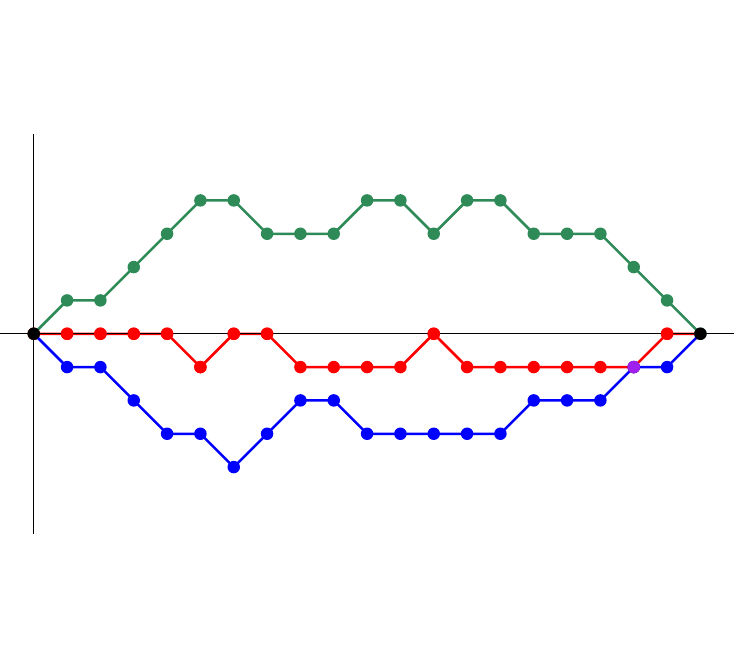}
\caption{A comparison of $P_\sigma(t)$ (left) and ${s}_{\walk_\sigma}(t)$ (right) for the permutation $\sigma \in A\!v_{20}(\rho_3)$ from the permutation from Figure \ref{proj3}. }
\label{compare2}	
\end{figure}

\begin{lemma} \label{tahini}
Fix $T,T'$ and let $n$ be sufficiently large.  Let $\sigma\in \avn(\rho_d)$ with ${\walk_\sigma} \in \scwminus((T,\cdot),(n-T',\cdot))$.
  Then
%\begin{itemize}
%\item $\petrov(m)$ holds for all $m \in [T,n-T]$  and 
%\item 
%\end{itemize}
$$\sup_{t \in [0,1]} \dis(P_\sigma(t), \hat s_{\omega_\sigma}(t)) \leq n^{-.1}. $$
%\textbf{Does this need to be }${n^{-.6}}$?
\end{lemma}

\begin{figure}
\includegraphics[scale=.30]{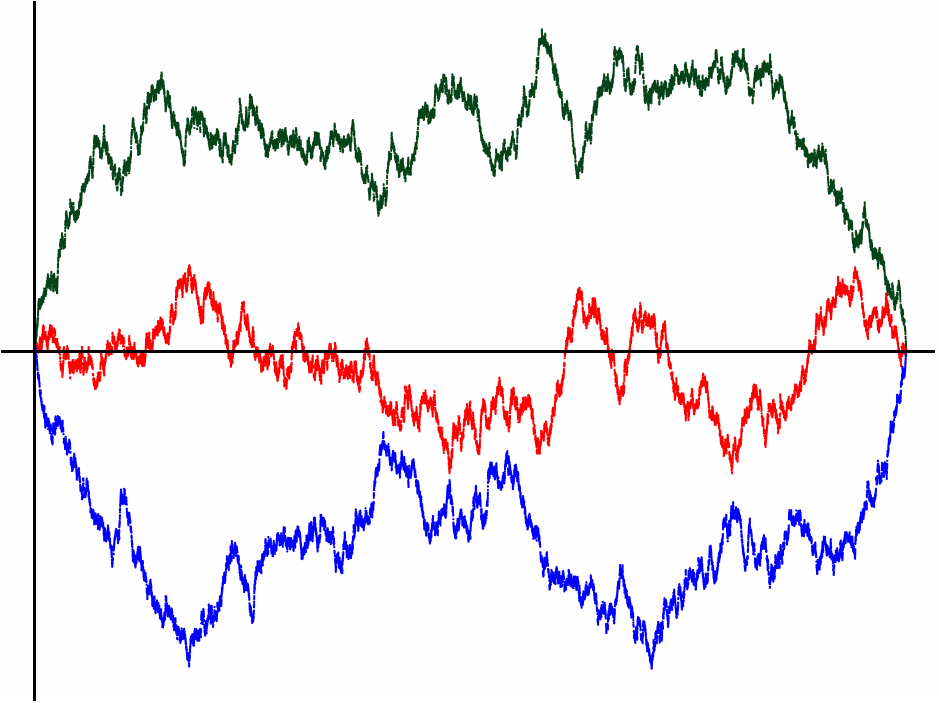} \hspace{1cm}\includegraphics[scale=.30]{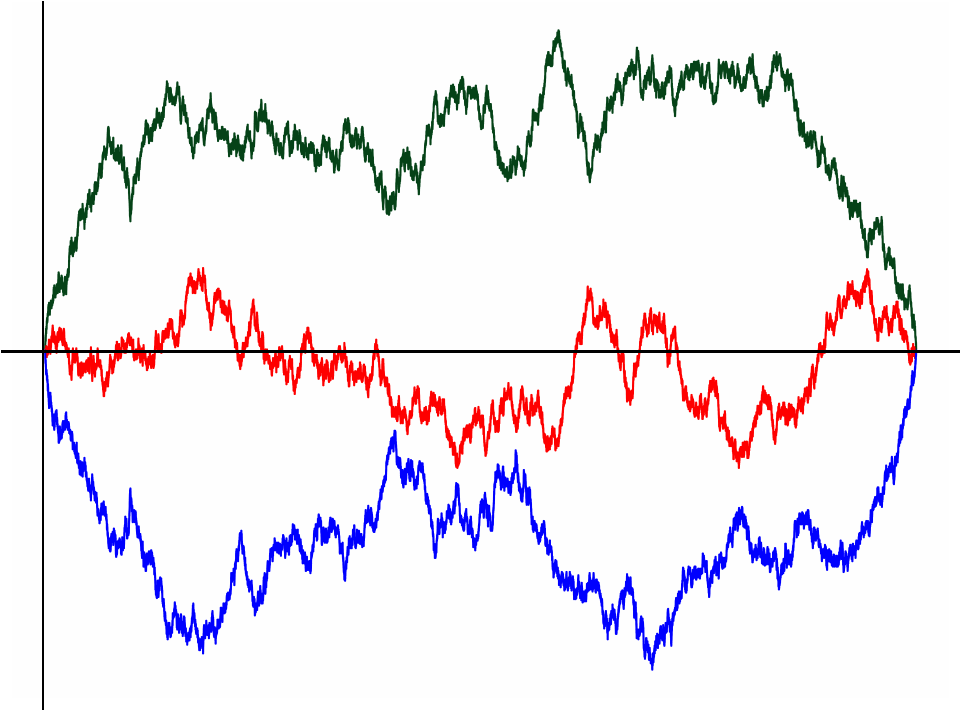}
\caption{A comparison of $P_\sigma(t)$ (left) and $\hat{s}_{\walk_\sigma}(t)$ (right) for a permutation $\sigma \in A\!v_{100000}(\rho_3)$. Lemma \ref{tahini} gives conditions on $\walk \in \Omega_n$ which insures that if $n$ is large then the two sets of  functions are close.}
\label{compare}	
\end{figure}

\begin{proof}
Let $n^{.39}>\max(T,T')$. As ${\walk_\sigma} \in \scwminus((T,\cdot),(n-T',\cdot))$ the Petrov conditions in Definition \ref{seven eleven} hold.
For $t \leq T/n$ (using $i=0$, $j=\min(tn,T^{.1})$ and $m=T$ in the first two conditions of Definition \ref{seven eleven})
the maximum over $l \in [d]$ of the component $f(\alpha^l(\sigma))(t)$ is at most $4T/\sqrt{n}$ and therefore $|P_\sigma(t)| < 4dT/\sqrt{n}$.  As $n^{.39}>\max(T,T')$ we have 
$|P_\sigma(t)| < 4dn^{-.11} < \frac{1}{2}n^{-.1}$.  By Lemma \ref{considiff}, $\petrov(T)$ guarantees that $|s_{\sigma}(nt)| < |(nt)|^{.6}$, thus for $t \leq T/n < n^{-.61}$, $|s_{\sigma}(nt)| < n^{.25}$.  Scaling by $\sqrt{2n/d}$ to obtain $\hat{s}_{\omega_\sigma}(t)$ we see that $\hat{s}_{\omega_\sigma}(t) < \frac{1}{2}n^{-.1}$. Combining these bounds shows that for small $t$, 
$$\dis(P_\sigma(t), \hat s_{\omega_\sigma}(t)) \leq n^{-.1}.$$  
%A similar argument is true for $nt > n-T'$ by using the same argument but with the reverse of $\hat{s}_{\omega_\sigma}$ and $P_\sigma$.  

For $t \in [T/n,1/2]$, the $l$th component of $P_\sigma(t)$ is  obtained by linear interpolation between the points
%the nearest points in each of the $d$ subsequences to the left of the vertical line $x=i$ where $i = \lfloor nt \rfloor $.  That is
$$\left(\frac{1}{(2dn)^{1/2}}\left(\posy^l(\ctx^l(i)) - \posx^l(\ctx^l(i))\right)\right).$$
The $l$th component of  $\hat{s}_{\omega_\sigma}(t)$ is given by the  values 
$\frac{1}{(2n/d)^{1/2}}(\ctx^l(i) - \cty^l(i)).$

By the Petrov conditions 
\begin{multline*}
 \frac{1}{(2dn)^{1/2}}(\posy^l(\ctx^l(i))-\posx^l(\ctx^l(i))) \\
 =  \frac{1}{(2dn)^{1/2}}(\posy^l(\ctx^l(i))-\posy^l(\cty^l(i)) + \posy^l(\cty^l(i))-\posx^l(\ctx^l(i))) \\
% \textbf {I don't believe the line above. It implies the middle two terms are equal. Why are they?}\\
  = \frac{1}{(2dn)^{1/2}}(d(\ctx^l(i)-\cty^l(i)) + \epsilon^l(i))
\end{multline*}
where by Lemmas \ref{considiff} and \ref{messy messi}
$$|\epsilon^l(i)| \leq \left(|\ctx^l(i)-\cty^l(i)|^{.6} + d|2i^{.3}|\right)< n^{.4}.$$ 
Thus for large enough $n$ and for any $t \in [0,1/2]$
$$\dis(P_\sigma(t),\hat{s}_{\omega_\sigma}(t))  \leq n^{-.1}.$$   For $t\in [1/2,1]$ a similar argument shows that for both $n/2 < nt < n-T'$ and $n-T' \leq nt \leq n$, $\dis(P_\sigma(t),\hat{s}_{\omega_\sigma}(t))\leq n^{-.1}$ for large enough $n$, and therefore the bound holds for all $t\in[0,1]$.
\end{proof}

  \section{Random walks close to the Weyl Chamber}

\label{minuseleven}

In this section we state some results about this walk conditioned to remain close to a cone that we will need. 
The proofs of these results are intricate and are delayed until  Sections \ref{walksincones} and \ref{birs}.

The following proposition allows us to import lemmas from \cite{denisov2015random} even though our random walk $\sloc(t)$ does not 
satisfy all of the hypothesis in \cite{denisov2015random}.

\begin{proposition} \label{grievance}
Suppose that $\omega$ is distributed like $\P$.  There is a linear transformation $H:\R^d\to \R^{d-1}$, invertible on the span of $\{e_i-e_j\}_{1\leq i,j \leq d}$, such that the random walk $H(\sloc)$ and the interior of the cone generated by $H(\mathrm{Weyl})$ satisfy the hypotheses of \cite{denisov2015random}.  In particular, 
\begin{enumerate}
\item The coordinates of each step of $H(s_\omega)$ are uncorrelated, normalized to have mean $0$ and variance $1$, and have finite moments of all order \cite[p. 995]{denisov2015random}. 
\item  $H(s_\omega)$ takes its values in a lattice in $\R^{d-1}$ that is a non-degenerate linear transformation of $\Z^{d-1}$ and that is a communicating class for $H(s_\omega)$ \cite[p. 999]{denisov2015random}. 
\item The cone generated by $H(\mathrm{Weyl})$ is convex and there is a harmonic function $u$ on $\R^{d-1}$, positive on the interior of the cone and vanishing on its boundary, as well as a constant $\kappa>0$, such that for all $x$ in the cone
\[ \P(\tau^B_x >t) \sim \kappa \frac{u(x)}{t^{d(d-1)/4}},\quad t\to\infty,\]
where $\tau^B_x$ is the first time a standard Brownian motion started from $x$ exits the cone \cite[p. 994]{denisov2015random}.
\end{enumerate}
\end{proposition}

See Proposition \ref{grievance two} for a detailed statement with a specific choice for $H$.  For our purposes, the existence of $H$ is typically more important than any particular choice for it.

We define the sequence $$T_l=\min\{\inf \{t: \sloc(t) \in \cone_{2^l}\},\lfloor 4.1^l\rfloor\}.$$
For $z \in \Z^d$ we define 
$$U(z)=\prod_{i<j}(z_i-z_j). $$
The function $U$ is harmonic for $\sloc$, see \cite{konig2001non}.

\begin{lemma} \label{chaplain_one}
For any $l>l_0>L$, $ x \in \partial \cone_{2^{l_0}}$, $T <4.1^{l_0}$, and $|x|\leq 2T$,

\begin{equation} \label{city lights one}
\sum_{x'}U(x')
\P_{(T,x)}\big( \spath \in \cwplus((T,x),(T_l,x'))  \big) \leq K U(x)
\end{equation}
where $K= 1+\prod_{j=0}^{\infty}(1+.01\cdot (.95)^{l_0+j})$
\end{lemma}

\begin{proof}
This is proved in Lemma \ref{chaplain} in Section \ref{birs}.
\end{proof}

 For any $l\geq l'>L$, $T<4.1^{L}$, $|x| \leq T$ and  $x \in \partial \cone_{2^{L}}$ 
let $\hat E_1(l',T,x)$ be the event that 
\begin{enumerate}
\item $\spath \in \cwplus((T,x),(T_l,\cdot)$
\item there exists $t \in (T_{l'-1},T_{l'})$ such that $\dis(\sloc(t), \partial \cone) \leq t^{.4}$ and
\item $\dis(\sloc(t), \partial \cone) > t^{.4}$
for all $t \in(T_{l'},T_{l})$. 
\end{enumerate}

\begin{lemma} \label{magnus_one}
For any $l>l_0>L$, $T\geq \lfloor 4.1^{l_0}\rfloor$ and $x$ with $|x|\leq 2T$ and $x \notin  \cone_{2^{l_0}}$ 

\begin{equation} \label{carlsen_one}
\sum_{x'}U(x')
\P_{(T,x)}\big( \spath \in \cwplus((T,x),(T_l,x'))  \big) \leq K( 2T)^{d(d-1)/2}
\end{equation}
where $K= \prod_{j=0}^{\infty}(1+.02\cdot (.95)^{l_0+j})$.
\end{lemma}

\begin{proof}
This is proved in Lemma \ref{magnus} in Section \ref{birs}.
\end{proof}

\begin{lemma} \label{benalla_one}
There exists a function $H(L)=o(1)$ such that  for any $L, l>L$, $T<4.1^L$, $|x| \leq 2T$ and  $x \in \partial \cone_{2^L}$ 
\begin{equation} \label{les enfants one}
\sum_y U(y)\prob_{(T,x)}(\spath \in \cwpplus((T,x),(T_l,y))\setminus \cwminus((T,x),(T_l,y)) ) \leq H(L)U(x).
\end{equation}
\end{lemma}

\begin{proof}
This is proved in Lemma \ref{benalla} in Section \ref{birs}.
\end{proof}

Choose $\Lf$ to be the smallest integer such that 
\begin{equation} \label{back flip one}
4^\Lf>n^{1-.08/d(d-1)}.
\end{equation} 
Also let $\delta = .02/d(d-1) < .01.$  Then $2^\Lf > n^{.5-2\delta} > n^{.49}.$
Thus

Remember that for any $j,k<n/2$ and $x,y\in \Z^d$ define 
 $\scwminus((j,x),(n-k,y))$ to be all paths $\spath$ such that
\begin{enumerate}
\item $\sloc(i) \in \cwminus((j,x),(\lfloor n/2 \rfloor,\cdot)$ and
\item $\sloc(n-i) \in \cwminus((k,y),(\lfloor n/2 \rfloor,\cdot)$
\end{enumerate}
$\scwplus((j,x),(n-k,y))$ and $\scwpplus((j,x),(n-k,y))$ are defined in an analogous way.

\begin{lemma} \label{spontini3}
For any $\epsilon>0$ there exists $l$ such that if $x,y \in \cone_{2^l}$ and
$T,T^* \leq (4.1)^l$ then for any $n$ sufficiently large
$$\prob_{(T,x)}\bigg( \scwminus((T,x),(n-T^*,y)) \ \bigg| \ \scwpplus((T,x),(n-T^*,y))  \bigg)>1-\epsilon$$
which implies %(do we need the following?)
$$\prob_{(T,x)}\bigg( \scwminus((T,x),(n-T^*,y)) \ \bigg| \ \cw((T,x),(n-T^*,y))  \bigg)>1-\epsilon.$$
\end{lemma}

\begin{proof}
This is proved in Lemma \ref{spontini2} (and Corollary \ref{Anna}) in Section \ref{birs}.
\end{proof}

\begin{lemma} \label{current_one}
There exist $C''$ such that for all $n$, $R \in [n/2,n]$, $T \leq n/4$ and 
for all $x,y \in \Kstar$ such that $|x|,|y| \leq n^{.5-\delta}$,

$$
\prob_{(T,x)}(\spath \in \cw((T,x),(R,y)) )
\geq C''U(x)U(y)n^{-d(d-1)/2}  \cdot n^{-(d-1)/2}. $$
There also exists $C'''$ such that for all $n$, all $R \in [n/2,n]$ and 
for all $x,y \in \Kstar$ such that $|x|,|y| \leq n^{.5-\delta}$
$$
\prob_{(T,x)}(\spath \in \cw((T,x),(R,y)) )
\leq C'''U(x)U(y)n^{-d(d-1)/2}  \cdot n^{-(d-1)/2}. $$
\end{lemma}

\begin{proof}
This is proved in Lemma \ref{current} in Section \ref{birs}.
\end{proof}

We define a pseudo-metric $D_K$ between paths.  Fix $K$ and $n>2K$ and two paths $s$ and $s'$ in $\Z^d$. We say $D_{K}(s,s')=0$ if $s(t)=s'(t)$ for all $i \in [K,n-K]$. Otherwise we say $D_{K}(s,s')=1$. We can extend $D_K$ to distributions on paths. For two measures on paths $\mu$ and $\mu'$ we set 
$D_K(\mu,\mu')$ to be the infimum over all couplings $\nu$ of $\mu$ and $\mu'$ of 
$\E_{\nu} ( D_{K}(s,s'))$. We also recall from Equation \eqref{eq empi} the notation 
\[ (g(a) \ :\ a \in A)=\frac{1}{|A|}\sum_{a \in A} \delta_{{g(a)}}.\]

\begin{cor} \label{Elsa}
Fix $\epsilon>0$ and $K$. Let $M$ be a measure on quadruples 
$(s,x)$ and $(t,y)$ such that with probability one have $0 \leq s,t \leq K$, $x,y \in \cone$ and $|x|,|y|\leq K$. 
For $n>2K$ define $\hat M_n$ to be the measure generated by picking 
$(s,x)$ and $(t,y)$ according to $M$ and then sampling $\walk$ from
$\cw((s,x),(n-t,y))$.
There is a $K'$ such that for any $M$ 
%there is a coupling of $\hat M$ with $\cw ((0,0),(n,0))$ such that
$$D_{K'}\bigg(\hat M_n,(\sloc : \omega \in \cw ((0,0),(n,0)))\bigg)<\epsilon.$$
\end{cor}

\begin{proof}
This is proved in Corollary \ref{red hen} in Section \ref{walksincones}.
\end{proof}

\begin{theorem}\label{cone scaling limit two}
Suppose that $x\in \C$.  Then for all bounded, continuous functions $f: D([0,1], \R^{d}) \to \R$ we have
\[\E\left( f\left( \frac{x+\sloc(n\cdot)}{\sqrt{2n/d}}\right) \middle| \tau_x >n, \sloc(n)=0 \right)  \rightarrow \E[ f(\Lambda(Z))].\]
\end{theorem}

\begin{proof}
This is proved in Theorem \ref{cone scaling limit} in Section \ref{walksincones}.
\end{proof}

\section{The scaling limit}
\label{peaches}

In this section we connect the set of $\snd$ with paths on $\Z^d$. We fix $n$ and $d$ and (as we will be adding subscripts and superscripts) we  write $\snd$ as $\X$. Then we calculate the scaling limit.
We will frequently use the set we defined in \eqref{byeweek}
$$\spacen =\bigg\{(a,b) \in  [d]^n \times [d]^n: \#\{i:a_i=l\}=\#\{i':b_{i'}=l\}\  \forall \ l \in [d]\bigg\}. $$
Consider 
$\walk=\{(a,b)\} \in\spacen.$
In Section \ref{divest} we showed how a pair of sequences $\walk$ maps to a path $s_\walk$ in $\Z^d$.

We can compose the embedding of $\X$ into $\spacen$ and the map from $\spacen$ to paths on $\Z^d$ to 
get a map from a $\X$ to paths on $\Z^d$.  For $\sigma\in \X$ we let $s_{\walk_\sigma}$ denote this path.  If we linearly interpolate and scale  
$s_{\walk_\sigma}$ properly  we then get an ordered collection of $d$ functions on $[0,1]$ that start and end at 0. We call this scaled function $\hat s_{\walk_\sigma}$.
We also have another map $P_\sigma$ from $\X$ to a collection of $d$ functions on $[0,1]$ 
that start and end at 0. 
We will show that the scaled version of
$s_{\walk_\sigma}$ is usually very close to $P_\sigma$. 
In fact we will show that $P_\sigma$ and $\hat s_{\walk_\sigma}$ are sufficiently close so that when $\sigma$ is chosen uniformly from $\snd$ they have the same scaling limit. 

Fix some large integer $L$.  Define
$$R_L=\inf \{t:\ \petrov(t) \cap\{ \spath(t) \in \cone_{2^L}\}\}$$ and 
$$R^*_L= \inf \{t:\ \petrov^*(t) \cap \{ \spath^*(t) \in \cone_{2^L}\}\}.$$
Note that $R_L$ is a stopping time and $R^*_L$ is a stopping time for the reverse walk. 

For $n>100\cdot 4.1^L$ we divide the set of $\X$ into three disjoint subsets.
Define 
$$\X_1=\bigg\{\sigma: \max(R_L({\walk_\sigma}),R^*_L({\walk_\sigma}))<\lfloor 4.1^L\rfloor, \walk_\sigma \in \scwminus((R_L,\cdot)(n-R_L,\cdot))\bigg\},$$

$$\X_2=\bigg \{\sigma: \max(R_L({\walk_\sigma}),R^*_L({\walk_\sigma}))<\lfloor 4.1^L\rfloor, \walk_\sigma \not \in \scwminus((R_L,\cdot)(n-R^*_L,\cdot))\bigg \}$$
and
 
$$\X_3=\left\{\sigma: \max(R_L({\walk_\sigma}),R^*_L({\walk_\sigma}))\geq\lfloor 4.1^L\rfloor\right\}.$$

Our strategy is (roughly) as follows. We will show that if $L$ is large then the scaling limit of $s_{\walk_\sigma}$ for 
$\sigma \in \X_1$ is close to the traceless Dyson Brownian Bridge (TDBB). From Lemma \ref{tahini} we showed that the scaled version, $\hat s_{\omega_\sigma}$, of  
$s_{\walk_\sigma}$ is close to $P_\sigma$ for $\sigma \in \X_1$.
Thus we can determine the scaling limit of $P_\sigma$ for $\sigma \in \X_1$.
Then we will show that for any $\epsilon>0$ there exists an $L$ such that
$|\X_1|>(1-\epsilon)|\X|$.
Thus the scaling of $P_\sigma$ for $\sigma \in \X_1$
is the same as the scaling of $P_\sigma$ for $\sigma \in \X$. We use the connection with random walks in a cone to show that they 
both are given by TDBB.

\subsection{$|\X_3|=o(|\X|)$}

\begin{lemma} \label{perisic}
There exists $c>0$ such that for all $n$
$$|\X| \geq c d^{2n} n^{-(d^2-1)/2}.$$
\end{lemma}

\begin{proof}
This follows  from \cite[Theorem 2.10]{regev} where an exact enumeration is given. 
\end{proof}

Let $g(m) = \sum_{k>4.1^m}k^{2d^2} e^{-ck^\beta}.$  Given $\epsilon$ and $d$ choose $L$ such that 
\begin{equation} \label{roger}
(4.1^L)^{d^2}g(L)<\epsilon
\end{equation} 
and
\begin{equation} \label{rafa}
g(L)^2<\epsilon.
\end{equation}

Let $\beta=.15.$ 
Consider the following events:  

\begin{itemize}
	\item $E^0:= \{R_L \leq \lfloor 4.1^L\rfloor$ and $\spath \in \cwpplus((0,0),(R_L,\cdot))\}.$
	\item $E^1_k:=$ $\{R_L=k$ and $\spath \in \cwpplus((0,0),(R_L,\cdot))\}.$
	\item $F:=$ $\{R_L\geq n^{1-\beta}$ and $\spath \in \cwpplus((0,0),(n^{1-\beta},\cdot))\}$.
\end{itemize}

We can also similarly define $E^{0,*}$, $E^{1,*}_k$ and $F^*$ based on the path $S^*$. 

\begin{lemma}\label{obsidian}
If  $\sigma \in \X$ and $R_L,R^*_L<n/2$  then
	%and for $\spath_\sigma \in E^0 \cap E^{0,*}  \cap  \possible((0,0),(n,0))$ then 
         $$\spath_\sigma \in   \cwpplus((0,0),(R_L,\cdot)),$$  and
	 $$\spath^*_\sigma \in   \cwpplus((0,0),(R^*_L,\cdot))$$ 
	 and for any $0\leq s\leq t \leq n$ we have  
	 	 $$\spath_\sigma \in   \possible((s,\cdot),(t,\cdot)).$$

	 \end{lemma}

\begin{proof}
By Lemma \ref{permsRgreat} we have that 
         \begin{equation}\spath_\sigma \in   \possible((0,0),(n,0)).\label{vandermonde}\end{equation} 
This combined with the definition of $\possible$ implies that for any $0 \leq s<t \leq n$ 
 $$\spath_\sigma \in   \possible((s,\cdot),(t,\cdot)).$$ 
 which proves the third claim.
 Line \eqref{vandermonde}, the definition of $\cwpplus((0,0),(t,\cdot))$ and the fact that  $R_L \leq n/2$ implies that 
         $$\spath_\sigma \in   \cwpplus((0,0),(R_L,\cdot)).$$ 
As $R_L<n/2$ this proves the first claim. The second claim is proven in an identical manner.
%The proof of the second claim is very similar to the proof of Lemma \ref{permsRgreat}. We leave the proof to the reader.
\end{proof}

%\begin{lemma}
%	For any  $\sigma $ and $i,j<n/2$ if 
%         $$\spath_\sigma \in   \cwpplus((0,0),(i,\cdot)),$$ 
%	 $$\spath_\sigma \in   \possible((i,\cdot),(n-j,\cdot))$$ and 
%	 $$\spath^*_\sigma \in   \cwpplus((0,0),(j,\cdot))$$ then for any $i'<i$ and $j'<j$
%	 $$\spath_\sigma \in   \possible((i',\cdot),(n-j,\cdot))$$ and 
%	 $$\spath_\sigma \in   \possible((i,\cdot),(n-j',\cdot)).$$  
%\end{lemma}

%\begin{proof}
%\end{proof}
We define $T^*_L$ based on the definition of $T_L$ (before Lemma \ref{chaplain_one}) using the walk $\spath^*_\sigma$.

\begin{lemma}\label{breakdown}
	If  $\sigma \in \X$    has
	$\max(R_L,R^*_L)\geq \lfloor 4.1^L \rfloor$ for $\spath_\sigma$, then $\spath_\sigma$ must be in one of the following events:
	\begin{enumerate}
\item $F$ 
\item $F^*$ 
\item $$\bigcup_{k \in (\lfloor 4.1^L \rfloor,n^{1-\beta})} E^1_k \cap E^{0,*}  \cap   \possible((k,\cdot),(n-T^*_{L},\cdot))$$
\item $$\bigcup_{k^* \in (\lfloor 4.1^L \rfloor,n^{1-\beta})} E^0 \cap E^{1,*}_{k^*}  \cap   \possible((T_L,\cdot),(n-k^*,\cdot))$$
\item $$\bigcup_{k,k^* \in (\lfloor 4.1^L \rfloor,n^{1-\beta})} E^1_k \cap E^{1,*}_{k^*}  \cap   \possible((k,\cdot),(n-k^*,\cdot))$$
%\item there exists $k^* \in (\lfloor 4.1^L \rfloor,n^{1-\beta})$ such that $E^0 \cap E^{1,*}_{k^*} \cap  \{\spath_\sigma \in \possible((T_L,\cdot),(n-k^*,\cdot))\}$ or 
%\item there exists $k,k^* \in (\lfloor 4.1^L \rfloor,n^{1-\beta})$ such that $E^{1}_k \cap E^{1,*}_{k^*} \cap  \{\spath_\sigma \in \possible((k,0),(n-k^*,0))\}$.
%\item there exists $k \in (\lfloor 4.1^L \rfloor,n^{1-\beta})$ such that $E^1_k \cap E^{0,*}  \cap  \{ \spath_\sigma \in \possible((k,0),(n,0))\}$
%\item there exists $k^* \in (\lfloor 4.1^L \rfloor,n^{1-\beta})$ such that $E^0 \cap E^{1,*}_{k^*} \cap  \{\spath_\sigma \in \possible((T_L,\cdot),(n-k^*,\cdot))\}$ or 
%\item there exists $k,k^* \in (\lfloor 4.1^L \rfloor,n^{1-\beta})$ such that $E^{1}_k \cap E^{1,*}_{k^*} \cap  \{\spath_\sigma \in \possible((k,0),(n-k^*,0))\}$.
\end{enumerate}
\end{lemma}

\begin{proof}
If  $\max(R_L,R^*_L) \geq n^{1-\beta}$  then $\spath_\sigma \in F \cup F^*$. 
If  $\max(R_L,R^*_L) \leq n^{1-\beta}$  then the hypotheses of Lemma \ref{obsidian} are satisfied and we have that 
         $$\spath_\sigma \in   \cwpplus((0,0),(R_L,\cdot)),$$ 
	 $$\spath_\sigma \in   \possible((R_L,\cdot),(n-R^*_L,\cdot))$$ and 
	 $$\spath^*_\sigma \in   \cwpplus((0,0),(R^*_L,\cdot)).$$ 
Using this plus the hypothesis that $\max(R_L,R^*_L)$ is bounded below by $ \lfloor 4.1^L \rfloor$ we get that at least one of $E^1_k\cap E^{0,*}$, $E^0\cap E^{1,*}_{k^*}$, or $E^1_k \cap E^{1,*}_{k^*}$ must occur for some $\lfloor 4.1^L \rfloor \leq k,k^* \leq n^{1-\beta}$. 
As in the previous lemma $\sigma \in \X$ by Lemma \ref{permsRgreat} $\spath_\sigma \in \possible((0,0),(n,0))$. Thus by the definition of $\possible((0,0),(n,0))$ we have that  $\spath_\sigma \in \possible((i,*),(j,*))$ for any $0 \leq i<n/2$ and $n/2<j\leq n$. 
%If $\max(R_L,R^*_L) \geq n^{1-\beta}$ then at least one of $F$ or $F^*$ occurs.   
\end{proof}

\begin{lemma} \label{umtiti}
There exist $C_1$, $c_2$ and $\delta_1 > 0$ such that
 $$\prob_{(0,0)}(F), \prob_{(0,0)}(F^*)  \leq C_1e^{-c_2n^{\delta_1}}.$$
\end{lemma}

\begin{proof}
Let $m$ be the last time before $R_L$ such that $\petrov(m)$ fails. If $m^5>n$ then by Lemma \ref{pathpetrov} this occurs with probability at most $c_3e^{-m^{\gamma}} < c_3e^{-n^{\gamma'}}$. If $m^5 \leq n$, then from $n^{.2}$ until $n^{.85}=n^{1-\beta}$, $s_{\spath_\sigma}$ must be within $n^{.4}$ of the boundary of $\cone$.  By standard arguments about random walk, the probability the fluctuation over the time scale $n^{.85}=n^{.8}n^{.05}$ is bounded by $e^{-c_4n^{.05}}$ for some positive constant $c_4$.  Then choose $\delta_1, C_1$ and $c_2$ so that $\min(c_3e^{-n^{\gamma'}},e^{-c_4n^{.05}}) < C_1 e^{-c_2n^{\delta_1}}.$  By symmetry the second inequality also holds.
\end{proof}

Let $\coup$ be the largest $l$ such that $4.1^l<n/10$.

\begin{lemma} \label{les bleus}
There exists $C<\infty$ such that for all $x,y \in \cone_{2^{\coup}}$, $T_{\coup}$and $T^*_{\coup}$
$$\P_{(T_{\coup},x)}(\possible((T_{\coup},x),(n-T^*_{\coup},y))) \leq C U(x)U(y) n^{-(d^2-1)/2}.$$
\end{lemma}

%{\bf Should we get rid of the conditioning with $\P_{(T_{\coup},x)}$? We never said anything about the subscript being random. We do this in Section 9.}

\begin{proof}
If a path is in $\scwpplus((T_{\coup},x),(n-T^*_{\coup},y))$ then either it is in $\cone_{-n^{.4}}$ from $T_\coup$ to $T^*_{\coup}$
 or $\petrov(m)$ fails for some $m>2^{\coup}>n^{.2}$. We will show that the probability of the former is bounded by 
$C' U(x)U(y) n^{-(d^2-1)/2}.$
By Proposition \ref{grievance} 
 and Lemma 28 in \cite{denisov2015random} this probability is bounded by 
 $C(x,y) n^{-(d^2-1)/2}.$
The function 
$C(x,y)$ 
is bounded by
$C U(x+(n^{.4},\dots,n^{.4})+x_0) U(y+(n^{.4},\dots,n^{.4})+x_0)$ for some fixed $x_0$. (See the note at the bottom of page 10 in \cite{denisov2015random} 
or Theorem 4 of \cite{varopoulos1999potential}.)
As every coordinate of $x$ and $y$ are at least $2^{\coup}>n^{.4}$ then 
$$U(x+(n^{.4},\dots,n^{.4})+x_0) U(y+(n^{.4},\dots,n^{.4})+x_0)\
    \leq 2^{d(d-1)}U(x)U(y).$$
Thus the probability of this event is at most $C U(x)U(y) n^{-(d^2-1)/2}$.

By Lemma \ref{conditionedpetrov} the probability that $\petrov(m)$ fails for some
$m$ in this region is at most $e^{-cn^{\gamma'}}$ for some $\gamma'>0$. 
The lemma follows by putting these two estimates together with the union bound.
\end{proof}

\begin{lemma} \label{belgium}
There exists $r$, $C$ and $\alpha> 0$ such that for all $k \in [\lfloor 4.1^L\rfloor,T_{\coup})$

$$\E_{(0,0)} \left( \1_{E^1_k \cap \{ k \leq T_{\coup}\}} 
 U(\sloc(T_{\coup}))\right)\leq C k^r e^{-ck^{\alpha}}.$$ 	
\end{lemma}
\begin{proof}

Fix $k >4.1^L.$  
First we show that $\P_{(0,0)}(E^1_k) \leq e^{-ck^{\alpha}}$. This works in exactly the same way as Lemma \ref{umtiti}.
Let $m$ be the largest integer less than or equal to $k$ such that $\petrov(m)$ does not occur. If $m \geq k^{1/5}$ we use the bound on the probability of $\petrov(m)$ failing. If $m \leq k^{1/5}$ then the path stays close to the boundary of $\cone$ without entering $\cone_{2^L}$ between $k^{1/5}$ and $k-1$. This also has low probability, which can be bounded by a similar argument to that used in Lemma \ref{umtiti}

If $E^{1}_{k}$ occurs and $k<T_{\coup}$ then we let $x=\sloc(k)$. Then we have $|x| \leq 2k$ and 
$x \not \in INT(\cone_{2^L})$. Thus Lemma \ref{magnus_one} applies.
The portion of the path after $k$ is independent of the portion before $k$. So

\begin{eqnarray*}
\lefteqn{\E_{(0,0)} \left( \1_{E^1_k } 
 U(\sloc(T_{\coup}))\right)}&&\\
 &=&\sum_{x \not \in INT(\cone_{2^L})} \prob_{(0,0)}(E^1_k \cap \sloc(k)=x)\cdot \\
 & &\qquad\sum_{y} \P_{(0,0)}\big( \spath \in \cwplus((T,x),(T_{\coup},y)) \ | \ \sloc(T)=x \big) U(\sloc(y)) \\
 &\leq & \P_{(0,0)}(E^1_k ) k^r \\
& \leq & C k^r e^{-ck^{\alpha}}. 
\end{eqnarray*}
The equality is because the portion of the path after $k$ is independent of the portion before $k$. 
For the next line we have $|x| \leq 2k$ and 
$x \not \in INT(\cone_{2^L})$. Thus Corollary \ref{magnus_one} justifies the first inequality.
 \end{proof}

\begin{lemma} \label{red devils}
There exists $c>0$ such that for all  $n$ and $L$ (and thus all $\coup$)
$$\E_{(0,0)} \left( \1_{E^0  }%\cap\cwpplus((T_L,\cdot),(T_{\coup},\cdot))} 
U(\sloc(T_{\coup}))\right)\leq c (2\cdot 4.1^L)^{d(d-1)/2}$$ 	
and
$$\E_{(0,0)} \left( \1_{E^{0,*}  }%\cap\cwpplus((T_L,\cdot),(T_{\coup},\cdot))} 
U(\sloc^*(T^*_{\coup}))\right)\leq c (2\cdot 4.1^L)^{d(d-1)/2}.$$ 	

\end{lemma}

\begin{proof}
Every point $z$ such that $\sloc(T_{R_L})=z$ has $U(z) \leq (2\cdot 4.1^L)^{d(d-1)/2}$.
Then we apply Lemmas \ref{chaplain_one} and \ref{benalla_one}.	
The second statement follows in the same way.
\end{proof}

\begin{lemma} \label{recall}
There exists $C''$, $c$, $r'$ and $\alpha$ such that for all $L$
$$\sum_{ k>4.1^L}\P_{(0,0)} \bigg( E^1_k \cap E^{0,*}  \cap  \possible((T_{\coup},\cdot),(n-T^*_{\coup},\cdot))\bigg)
 \leq n^{-(d^2-1)/2} \sum_{k \geq 4.1^L} C' k^r e^{-ck^\alpha}  .$$
\end{lemma}

\begin{proof}
Fix $k$. We sample $\spath$ as follows. First we sample $\spath$ from 0 to $T_{\coup }$ 
then we sample $\spath^*$
from $0$ to $T^*_{\coup }$ then we sample $\spath$ from $T_{\coup }$ to $n-T^*_{\coup }$. We only get a valid path in the desired set if 
$\spath \in E^1_k$, $\spath \in E^{0,*}$ and 
$$\sloc(n-T^*_{\coup })=\sloc^*(T^*{(\coup )}).$$
 The first two of these events are independent and the third is independent conditioned on $\sloc(T_{\coup })=x$.
Thus using Lemma \ref{current_one} and Lemmas \ref{belgium} and \ref{red devils} we get
\begin{eqnarray*}
\lefteqn{\P_{(0,0)} \left( E^1_k \cap E^{0,*}  \cap  \possible((T_{\coup },\cdot),(n-T^*_{\coup },\cdot)\right)}
&&\\
&= &
\sum_{x,y} \P_{(0,0)}\bigg(
   E^1_k  \cap \{ k \leq T_{\coup }\} \cap 
   \sloc(T_{\coup })=x\bigg) 
\P_{(0,0)} \bigg (E^{*,0}  \cap \sloc^*(T^*_L,\cdot)=y\bigg) \\
&& \cdot \P_{(0,0)} \bigg( \spath \in \scwpplus((T_{\coup },x),(n-T^*_{\coup },y) \ |  \ \sloc(T_{\coup })=x)\bigg)\\
&\leq &
\sum_{x,y} \P_{(0,0)}\bigg(
   E^1_k  \cap \{ k \leq T_{\coup }\} \cap \sloc(T_{\coup })=x \bigg) 
   \\
&& \hspace{.5in}  \cdot	
\P_{(0,0)} \bigg (E^0  \cap \sloc(T_{\coup })=y\bigg) 
\cdot CU(x)U(y)n^{-(d^2-1)/2}\\
&\leq & n^{-(d^2-1)/2} \left(\sum_{x}  \1_{E^1_k \cap \{ k \leq T_{\coup }\} \cap \sloc(T_{\coup })=x}U(x)\right)%\\
\left( \sum_{y}\1_{E^{0,*}  \cap \sloc^*(T^*_{\coup })=y} U(y) \right)\\
& \leq & (n^{-(d^2-1)/2})(C' k^r e^{-ck^\alpha})  (2k)^{d(d-1)/2} \\
& \leq & C'' k^{r'} e^{-ck^\alpha}  n^{-(d^2-1)/2}.
\end{eqnarray*}
Summing up over $k$ gives us 
\begin{eqnarray*}
\sum_{ k>4.1^L}\P_{(0,0)} ( E^1_k \cap E^{0,*}  \cap  \possible((T_{\coup },\cdot),(n-T^*_{\coup },\cdot))
&\leq &  n^{-(d^2-1)/2} \sum_{k \geq 4.1^L} C'' k^{r'} e^{-ck^\alpha}. 
\end{eqnarray*}
\end{proof}

\begin{lemma} \label{economia}
\begin{multline}
\sum_{ k,k'>4.1^L}\P_{(0,0)} ( E^1_k \cap E^{1,*}_{k'}  \cap  \possible((T_{\coup },\cdot),(n-T^*_{\coup },\cdot)) \\
\leq
n^{-(d^2-1)/2} \sum_{k,k' \geq 4.1} C^2 k^r e^{-ck^\alpha}(k')^{r'} e^{-ck'^\alpha}.
\end{multline}
\end{lemma}

\begin{proof}
The proof of this goes in exactly the same way as the proof of Lemma \ref{recall}.
\end{proof}

We combine the preceding results to conclude

\begin{lemma} \label{anthony}

$|\X_3| =o(|\X|)$.
\end{lemma}

\begin{proof}
By Lemma \ref{perisic}  $$|\X| \geq c d^{2n}n^{-(d^2-1)/2}.$$
By Lemma \ref{breakdown} we have that $$\X_3 \subset \bigcup_{i=1}^5\X_{3,i}$$ where the sets 
$\X_{3,i}$	are the sets defined in the statement of Lemma \ref{breakdown}.
By Lemma \ref{umtiti} we get that for any $\epsilon>0$ we can find a large $L$ such that for all $n$ sufficiently large
$$|\X_{3,1}|, |\X_{3,2}|  \leq \epsilon d^{2n}n^{-(d^2-1)/2}.$$

If $\sigma \in \X$ but not in the union of sets in Lemma \ref{recall} then $\petrov(m)^C$ or $\petrov^*(m)^C$ 
must occur for some $m>n^{.4}$. This has probability at most $Ce^{-n^\alpha}$ for some positive $C$ and $\alpha$.
Thus by Lemma \ref{recall} we get that for any $\epsilon>0$ we can find a large $L$ such that for all $n$ sufficiently large
$$|\X_{3,3}|, |\X_{3,4}|  \leq \epsilon d^{2n}n^{-(d^2-1)/2}.$$
Similarly by Lemma \ref{economia}
we get that for any $\epsilon>0$ we can find a large $L$ such that for all $n$ sufficiently large
$$|\X_{3,5}| \leq \epsilon d^{2n}n^{-(d^2-1)/2}.$$
Thus the lemma follows.
\end{proof}

By Lemma \ref{permsRgreat} the path that is the image of a permutation in $\X$ is in $\scwpplus$.  We will show that the cardinality of $$\scwpplus ((T_L,\cdot),(n-T^*_L,\cdot))\cap \{\max(T_L,T^*_L)< \lfloor 4.1^L \rfloor\}$$ is $1+o(1)$ times the cardinality of $\scwminus((T_L,\cdot),(n-T^*_L,\cdot))\cap \{\max(T_L,T^*_L)< \lfloor 4.1^L \rfloor\}$.

The image of the set $\X_1$ is not exactly the set of paths that we can use our previous results to calculate the scaling limit. But the size of the 
symmetric difference between the set we will describe and $\X_1$ will be small
in comparison with $|\X_1|$. Thus the two sets will have the same scaling limits.
We first describe which paths we want to exclude from $\X_1$. We want paths that satisfy the Petrov conditions at both $T_L$ and $T^*_L$. Let 
$$\X'_1=\bigg\{\sigma: \sigma \in \X_1, \spath \text{ satisfies $\petrov(T_L)$}, \spath^* \text{ satisfies $\petrov(T^*_L)$} )\bigg\}$$

%{\bf In the following lemma we are using estimates based on random walks. Shouldn't they be based on $\omega$. We get that using Lemma \ref{I'm gonna}. Which mean we need to move Lemma \ref{I'm gonna} up before this.}

\begin{lemma} \label{may day}
$|\X \setminus \X_1'| =o(|\X|).$
\end{lemma}

\begin{proof}
By Lemma \ref{anthony} and our partition 
$$\X=\X_1 \cup \X_2 \cup \X_3$$
 it is enough to bound $|(\X_1\cup \X_2) \setminus \X_1'|$.
If $$\omega_{\sigma} \in (\X_1\cup \X_2) \setminus \X_1'$$ then $$s_{\omega_\sigma} \in \scwpplus((T_L,\cdot),(n-T^*_L,\cdot))$$ with $T_L,T^*_L < \lfloor (4.1)^L\rfloor$ and either 
\begin{enumerate}
\item $\petrov(T_L)^C$,
\item $\petrov^*(n-T^*_L)^C$ or 
\item $\omega_\sigma \not\in \scwminus((T_L,\cdot),(n-T^*_L,\cdot))$.
\end{enumerate}
There are at most  $4.1^L$ choices for each of $T_L$ and $T^*_L$ and $4.1^{Ld}$ choices for $s_{\omega_\sigma}(T_L)$ and $s^*_{\omega_\sigma}(T^*_L)$.  By Lemma \ref{pathpetrov} for each of these $k$ and $k'$ the probability that $\petrov(k)^C$ or $\petrov^*(k')^C$ occurs is
$2 e^{-2^{cL}}$ for some  $c>0$. Each of the possible $s_{\omega_\sigma}(T_L)$ and $s_{\omega_\sigma}^*(T^*_L)$ have $U(s_{\omega_\sigma}(T_L))$ and $U(s_{\omega_\sigma}^*(T^*_L))$ and most $4.1^{Ld(d-1)}.$
For each of these sets we apply Lemma \ref{les bleus}.
Thus the number of permutations satisfying one of the first two conditions is at most
$$(4.1^{Ld})^2(4.1^{L})^2(2e^{-c2^L})d^{2n}n^{-(d^2-1)/2} \leq \epsilon d^{2n}n^{-(d^2-1)/2}.$$
This implies that
\begin{equation} \label{billie} |\X_1\setminus \X_1'| \leq |\X_1|.\end{equation}

Now we count the permutations satisfying the first two conditions but not the third.
Call this set $Y$. Partition $Y$
by the time and position at times $T_L$ and $n-T^*_L,$ that is
 $$Y=\bigcup_{t,t^*,v,w}\{\sigma:\ T_L=t, T^*_L=t^*, s_{\omega_\sigma}(T_L) = v, s_{\omega_\sigma}(n-T^*_L) = w\} \cap Y.$$
 
Call a set on the right hand side to be $H(t,t^*,v,w)$. Further partition a set on the right hand side $H(t,t^*,v,w)$ into $H_1(t,t^*,v,w)$ and $H_2(t,t^*,v,w)$ be the collection of pairs of sequences that are obtained from permutations in $\X_1$ (and thus $\X_1'$) or $\X_2$ respectively.  By Lemma \ref{spontini3} $(1-\epsilon)|H(t,t^*,v,w)|$ of pairs of sequences in $H(t,t^*,v,w)$ will be in $\scwminus((T_L,v),(n-T^*_L,w))$ and therefore in $H_1(t,t^*,v,w)$, hence $|H_2(t,t^*,v,w)| \leq \epsilon |H(t,t^*,v,w)|.$  By Lemma \ref{permsRgreat} the pair of sequences associated to a permutation in $\X$ is always in $\scwpplus,$ thus 
$$|H_1(t,t^*,v,w)| \geq (1-\epsilon) |H(t,t^*,v,w)|.$$  Combined this shows that $|H_2(t,t^*,v,w)| \leq \frac{\epsilon}{1-\epsilon}|H_1(t,t^*,v,w)|.$  This bound is uniform over all $v,w,T_L,$ and $T^*_L$ such that $\petrov(T_L)$ and $\petrov^*(T^*_L)$ hold and $v,w \in \cone_{2^L},$ thus 
\begin{equation} \label{jean} |\X_2| \leq \frac{\epsilon}{1-\epsilon} |\X_1'|. \end{equation}

Combining \eqref{billie} and \eqref{jean} with Lemma \ref{anthony} proves the lemma.
\end{proof}

\begin{lemma} \label{maybe she's born with it}
For every $\sigma \in \X_1'$ there exists $t,t^* \leq \lfloor 4.1^L \rfloor$,  $x,x^* \in \cone_{2^L}$ and $\tilde \walk \in [d]^{\{t+1,\dots,n-t^*\}}$ such that
\begin{enumerate}
\item $T_L(s_{\walk_\sigma})=t$,
\item $T^*_L(s_{\walk_\sigma})=t^*$,
\item $s_{\walk_\sigma}(t)=x$,
\item $s^*_{\walk_\sigma}(t^*)=x^*$,
\item $\tilde \walk^1(j)=\walk_\sigma(j)$ for all $j \in \{1,\dots, t\}$ which satisfies $\petrov(t)$,
\item $\tilde \walk^2(j)=\walk_\sigma(j)$ for all $j \in \{t+1,\dots, n-t^*\}$, and
\item $\tilde \walk^3(j)=\walk_\sigma(n+1-j)$ for all $j \in \{1,\dots, t^*\}$ which satisfies $\petrov(t^*)$.
\end{enumerate}
For $\sigma \neq \sigma'$ this collection of these seven objects is different.
\end{lemma}

\begin{proof}
The existence of these objects follow from the definition of $\X'_1$. If $\sigma \neq \sigma'$ then $\walk_\sigma \neq \walk_{\sigma'}$.  Thus one of the last three must be different as well.
\end{proof}

%\begin{lemma} \label{maybe its maybelline}
%For all $\epsilon>0$ there exists $L$ such that there exists $K'$ such that 
%for all $x,x^* \in \cone_{2^L}$, $t,n-t^* \leq \lfloor 4.1^L \rfloor$ and
%$\omega' \in \cwminus((t,x),(n-t^*,x^*)) $
%
%$$\frac{|\{ \walk \in [d]^{\{t+1,\dots,n-t^*\}} \times [d]^{\{t+1,\dots,n-t^*\}} :\ s_{\walk}=s_{\omega'} \text{ and } \walk \in \cap_{m=t}^{n-t^*}\overline{\petrov(m)}\}|}
%{|\{ \walk \in [d]^{\{t+1,\dots,n-t^*\}} \times [d]^{\{t+1,\dots,n-t^*\}} :\ s_{\walk}=s_{\omega'}\}|}>1-\epsilon.$$
%\end{lemma}
%
%\begin{proof}
%From the definition of $\petrov(m)$  (Definition \ref{seven eleven}) for each $s$ and $m$ the chance that  $\walk$
%with $s_\walk=s$ does not satisfy $\overline{\petrov(m)}$ is at most 
%$2e^{-m^{.01}}$. 
%The lemma then follows from the union bound.
%\end{proof}

%\begin{lemma}  \label{I'm gonna}
%For every $t, k \in \N$, $x \in \Z^d$ and path $s$ on $\Z^d$ of length $k$
% starting at $(t,x)$ we have 
%\begin{equation}\label{rogers}
%\P_{(t,x)}(s)=\P\left(
%\omega  \in [d]^{\{t+1,t+2,\dots,t+k\}} \times [d]^{\{t+1,t+2,\dots,t+k\}}:\ s_{\omega}=s \right).
%\end{equation}
%\end{lemma}
%
%\begin{proof}
%Let $N=|\{i \in \{t+1,t+2,\dots, t+k\}\ :\ s(i)=s(i-1)\}|$. It is easy to check that both quantities in \eqref{rogers} are
%$d^{-2k+N}$.
%\end{proof}

\begin{lemma} \label{wash that man}
For any $\epsilon>0$ there exist $L$ and $K$ satisfying the following. Fix any $t$, $t^*$, $x$, $x^*$, $\tilde \walk^1$ and $\tilde \walk^3$. 
Let $A=A(t, t^*, x, x^*, \tilde \walk^1, \tilde \walk^3)$ be the set of $\sigma \in \X_1'$ associated with these six objects. 
The fraction of $\sigma \in A$ with
%which satisfy the hypothesis of Lemma \ref{tahini}
$$\walk_{\sigma} \in \scwminus((t,x),(n-t^*,x^*))$$
is at least $1-\epsilon.$
\end{lemma}
 
\begin{proof}
%For any $\tilde \walk^2 \in [d]^{\{t+1,t+2,\dots,n-t^*\}} \times [d]^{\{t+1,t+2,\dots,n-t^*\}}$ we can map it to a path in $\Z^d$ which starts at $(t,x)$.
%Define $\overline{\cwminus}((t,x),(n-t^*,x^*))$ and $\overline{\cwpplus}((t,x),(n-t^*,x^*))$
%to be the subsets of $[d]^{\{t+1,t+2,\dots,n-t^*\}} \times [d]^{\{t+1,t+2,\dots,n-t^*\}}$ which map to paths in the corresponding sets.
By Lemma \ref{spontini3} 
%and \ref{I'm gonna}
%$$\left|{\cwminus}((t,x),(n-t^*,x^*))\right|>(1-\epsilon)\left|{\cwpplus}((t,x),(n-t^*,x^*))\right|$$
%
%%{\bf The conditions in the bigcap need to be symmetric.}
%%By Lemma \ref{maybe its maybelline} 
%%$$
%%\left|{\cwminus}((t,x),(n-t^*,x^*))\cap \bigcap_{m=t}^{n-t^*}\overline{\petrov(m)}\right|>(1-\epsilon)
%%\left|{\cwminus}((t,x),(n-t^*,x^*))\right|
%%$$
%Thus
\begin{equation} \label{no deal}
\left|{\scwminus}((t,x),(n-t^*,x^*))\right|>(1-\epsilon)
\left|{\scwpplus}((t,x),(n-t^*,x^*))\right|
\end{equation}

For every 
$\sigma \in A$ there exists $\tilde \omega^2 \in [d]^{\{t+1,t+2,\dots,n-t^*\}} \times [d]^{\{t+1,t+2,\dots,n-t^*\}}$. 
By Lemma \ref{permsRgreat} we have 
\begin{equation} \label{south pacific}
{\tilde \walk^1 \oplus \tilde \walk^2 \oplus \tilde \walk^3}  \in \scwpplus ((0,0),(n,0)).
\end{equation}
By \eqref{south pacific} we have that $\tilde \walk^2 \in \scwpplus((t,x),(n-t^*,x^*))$.
Thus every $\sigma \in A$ corresponds with a unique element in the set on the right hand side of \eqref{no deal}.

By Lemma \ref{middleman} we have that if 
$$\tilde \walk^2 \in \scwminus((t,x),(n-t^*,x^*))$$ then  
$\tilde \walk^1 \oplus \tilde \walk^2 \oplus \tilde \walk^3$ corresponds to a permutation in 
$\X$. It is easy to check that this permutation is also in $\X'_1$. 
Thus every element in the set on the left hand side of \eqref{no deal} corresponds with a unique $\sigma \in A$.

Combining the conclusions of these two paragraphs with \eqref{no deal} completes the proof.
\end{proof}

\begin{lemma} \label{right outa my hair}
For any $\epsilon>0$ there exist $L$ and $K$ such that the following is true. Fix any $t$, $t^*$, $x$, $x^*$, $\tilde \walk^1$ and $\tilde \walk^3$. 
Let $A=A(t, t^*, x, x^*, \tilde \walk^1, \tilde \walk^3)$ be the set of $\omega \in \X_1'$ associated with these six objects. 
Then 
there exists $K$ 
%and a coupling $\nu$ of $(s_{\walk_\sigma}: \ \sigma \in A)$ and $\cw((0,0),(n,0))$ 
such that 
$$D_K\bigg((s_{\walk_\sigma}: \ \sigma \in A),(s_{\walk}:\ \walk \in \cw((0,0),(n,0)))\bigg) <\epsilon$$
where $D_K$ is defined prior to Corollary \ref{Elsa}
\end{lemma}

\begin{proof}
Let 
$$A_1 =\left\{\sigma \in A: \ \tilde \walk^2_{\sigma} \in \scwminus((t,x),(n-t^*,x^*))  \right\}.$$ 
%As every element on the left hand side of \eqref{no deal} corresponds with a unique element of $A_1$ and every element of $A$ corresponds with a unique element on the right hand side of \eqref{no deal} 
By Lemma \ref{wash that man} we get $|A_1|>(1-\epsilon)|A|$.
As $A_1 \subset A$ the previous statement implies that for any $K$
\begin{equation} \label{customs union}
D_K((s_{\walk_\sigma}: \ \sigma \in A_1),(s_{\walk_\sigma}: \ \sigma \in A)) < \epsilon.\end{equation}
By Lemma \ref{spontini3} we get
\begin{equation} \label{general election}
D_K\bigg((s_{\walk_\sigma}: \ \sigma \in A_1),(s_{\walk}:\ \walk \in {\cw}((t,x),(n-t^*,x^*)))\bigg) < \epsilon.\end{equation}
%By Lemma \ref{I'm gonna} for any $\epsilon>0$ we get $K$ such that
%\begin{equation} \label{extension}
%D_K((s_{\walk}: \ \walk \in {\scwminus}((t,x),(n-t^*,x^*))), (s_{\walk}: \ \walk \in \scwminus((t,x),(n-t^*,x^*))))<\epsilon\end{equation}
By Corollary \ref{Elsa} there exists $K$ such that 
\begin{equation} \label{amendments}
D_K\bigg((s_{\walk}:\ \walk \in \cw((0,0),(n,0))), (s_{\walk}:\ \walk \in\cw((t,x),(n-t^*,x^*)))\bigg) < \epsilon.\end{equation}
Combining \eqref{customs union}, \eqref{general election}
%, \eqref{extension} 
and \eqref{amendments} proves the lemma. 
\end{proof}

\subsection{Proof of our main theorem}\label{sec main proof}
We are now prepared to prove our main result, which we restate for convenience.

\begin{theorem*}  \label{overwrite2} If $\sigma$ is a uniformly random elements of $\snd$ then, as $n\to\infty$, the following convergence holds in distribution with respect to the supremum norm topology on $C([0,1],\R^d)$: 
$$P_\sigma \ \xrightarrow{\text{dist}} \  \Lambda(Z).$$
\end{theorem*}

\begin{proof}[Proof of Theorem \ref{overwrite}]
By Theorem \ref{cone scaling limit two} random walk in a cone has a scaling limit of traceless Dyson Brownian bridge.  By Lemma \ref{may day} it remains to show that the distribution $(P_\sigma : \ \sigma \in \X'_1)$ has the same scaling limit as a random walk in a cone.

By Lemma \ref{maybe she's born with it}
we can write $\X_1'$ as a disjoint union of sets parameterized by 
$t$, $x$, $t^*$, $x^*$, $\tilde \walk^1$ and $\tilde \walk^3$.
Thus, using the notation from Equation \eqref{eq empi}, we can write $(s_{\walk_\sigma}: \ \sigma \in \X'_1)$ as a linear combination of pieces of the form
$$(s_{\walk_\sigma}: \ \sigma \in \X'_1, t, x, t^*, x^*,\tilde \walk^1, \tilde \walk^3).$$
By Lemma \ref{right outa my hair} each of those pieces 
can be coupled  with
$\cw((0,0),((n,0)))$ to show
\begin{multline*}D_K\bigg((s_{\walk_\sigma}: \ \sigma \in \X'_1, t, x, t^*, x^*,\tilde \walk^1, \tilde \walk^3),(s_\omega : \ \omega \in \cw((0,0),((n,0))))\bigg)=\\
% 
%such that the fraction of pairs of paths whose distance in the $D_K$ pseudo-metric is positive is less than $ \epsilon$. Thus we can find a similar coupling of any linear combination of them, including $(s_{\walk_\sigma}: \ \sigma \in \X'_1)$. So we have a coupling $\nu_1$ of 
D_K\bigg((s_{\walk_\sigma}: \ \sigma \in \X'_1),(s_\omega : \ \omega \in \cw((0,0),(n,0)))\bigg)<\epsilon.\end{multline*}
% such that with probability $1-10\epsilon$ the paired paths agree between $K$ and $n-K$. 
Scaling these paths we get a coupling $\nu$ of $(\hat s_\omega : \ \omega \in \cw((0,0),(n,0)))$  and 
$(\hat s_{\walk_\sigma}: \ \sigma \in \X'_1)$
such that with probability at least $1-\epsilon$ the paths are within $Cn^{-.5}$ in the supremum norm.

By Lemmas \ref{wash that man} at least $1-\epsilon$ fraction of the 
$\sigma \in \X'_1$ satisfy the hypothesis of 
Lemma \ref{tahini}.
By Lemma \ref{tahini} for those $\sigma$ we have that 
$$\sup_{t \in [0,1]} \dis(P_\sigma(t), \hat s_{\omega_\sigma}(t)) \leq n^{-.1}. $$
Thus the coupling $\nu$ gives a coupling of 
of $(\hat s_\omega : \ \omega \in \cw((0,0),(n,0)))$ and 
$(P_\sigma : \ \sigma \in \X'_1)$
such that with probability at least $1-2\epsilon$ the paired paths are within 
$Cn^{-.5}+n^{-.1}$ in the supremum norm. By the opening paragraph this completes the proof.
\end{proof}

%\newpage
\section{Walks in the Weyl Chamber}\label{walksincones}
Recall that if $\omega$ has distribution $\P$, then $\sloc$ is a lazy random walk such that $\sloc(t)-\sloc(t-1)=0$ with probability $1/d$ and $\sloc(t)-\sloc(t-1)=e_j-e_k$ with probability $1/d^2$ for each $ i \neq j$ with $1 \leq i,j \leq d$.  To simplify notation, in this section we let $S_t=S(t) =\sloc(t)$, $X_t = \sloc(t)-\sloc(t-1)$ and define $\P_x$ to be the law of $S+x$ for $x\in \R^d$.

For any $m$ recall the definition of $\cone_m$ in \eqref{divertida}.
%=\{(x_1,\cdots,x_d) \in \Z^d:\  d((x_1,\cdots,x_d),\ext(\cone))\geq m\}.$$

%{\bf This has already been defined in a slightly different way. They agree as long as $m \geq 0$.}

\subsection{Results from the literature}

In this section we recall some useful theorems from the literature.

\begin{theorem}[\cite{konig2001non}] \label{costume}
The function
$$U(x) =
\prod_{i=1}^d \prod_{j=i+1}^d
(x_i - x_j)$$
is harmonic for the random walk $S$.
\end{theorem}

% !TEX root = master sketch.tex

One of our main tools will be the results from \cite{denisov2015random} about random walks in cones.  The random walk we are interested in does not satisfy the hypotheses of \cite{denisov2015random} but can be transformed into one that does by an appropriate linear transformation.  In particular, our random walk takes place on a $(d-1)$-dimensional subspace of $\R^d$ and its covariance matrix is not the identity.  We now explain how to fix this for our random walk. 

Let
\[\H = \left\{ x=(x_1,\dots,x_d)\in \R^d\ \middle| \ \sum_{i=1}^d x_i =0\right\},\]
\[\C_{>} = \left\{ x=(x_1,\dots,x_d)\in \R^d\ \middle| \ x_1> x_2> \cdots > x_d \right\},\]
and $\C = \C_{>} \cap \H$.

Let $\bu\in \R^d$ be the unit vector defined by $u_i = (2d-2\sqrt{d})^{-1/2}$ for $1\leq i\leq d-1$ and 
\[ u_d = \frac{1 - \sqrt{d}}{\sqrt{2d-2\sqrt{d}}}. \]
Using $u$, we define the linear transformation $H_{\bu} : \R^d \to \R^d$ by
\[ H_{\bu}(x) = x - 2\ip{x}{u}u,\]
where $\ip{\cdot}{\cdot}$ is the usual inner product on $\R^d$.
\begin{proposition} $H_{\bu}$ has the following properties:
\begin{enumerate}
\item $H_{\bu}$ is an orthogonal involution.
\item If $\mathbf{1}$ is the all ones vector then $H_{\bu}(\mathbf{1}) = \sqrt{d} e_d$, where $e_d$ is the $d$'th standard basis vector in $\R^d$.
\item If $x\in \R^d$, then $\ip{x}{\mathbf{1}}=0$ if and only if $H_{\bu}(x)_d=0$.
\item $H_{\bu}$ is a bijection between $\C$ and the cone 
\[ \tilde \C = \left\{x\in \R^d : x_1> x_2 > \cdots > x_{d-1} > \frac{1}{\sqrt{d}-1}\sum_{i=1}^{d-1} x_i, x_d=0 \right\}.\]
\end{enumerate}
\end{proposition}

\begin{proof}
$H_{\bu}$ is a Householder transformation, and the fact that it is an orthogonal transformation is both classical and easy to show.  The second property is an easy computation and the third and fourth properties follow from orthogonality.
\end{proof}

Let $\pi_{d-1}: \R^d\to \R^{d-1}$ be the natural projection onto the first $d-1$ coordinates, so that $\pi_{d-1}(x_1,x_2,\dots,x_{d-1},x_d) = (x_1,x_2,\dots,x_{d-1})$.  Define 
\[\bar X_i = (d/2)^{1/2}\pi_{d-1} H_{\bu}(X_i) \quad \textrm{and} \quad \bar S_n = (d/2)^{1/2}\pi_{d-1} H_{\bu}(S_n).\]  

%The distribution of $\bar X_i$ is then given by $\P(\bar X_i= 0) = 1/3$,  
%\[\P\left(\bar X_i= \left(\frac{\sqrt{3}}{\sqrt{2}},-\frac{\sqrt{3}}{\sqrt{2}}\right)\right) = \P\left(\bar X_i= \left(-\frac{\sqrt{3}}{\sqrt{2}},\frac{\sqrt{3}}{\sqrt{2}}\right)\right) =1/9 \]
%\[ \P\left(\bar X_i= \left(-\frac{3\sqrt{2} -\sqrt{6}}{4} ,-\frac{3\sqrt{2} + \sqrt{6}}{4}\right)\right) = \P\left(\bar X_i= \left(\frac{3\sqrt{2} -\sqrt{6}}{4},\frac{3\sqrt{2} + \sqrt{6}}{4}\right)\right) = 1/9\]
%\[\P\left(\bar X_i= \left(-\frac{3\sqrt{2} + \sqrt{6}}{4},-\frac{3\sqrt{2} -\sqrt{6}}{4}\right)\right) = \P\left(\bar X_i= \left(\frac{3\sqrt{2} + \sqrt{6}}{4},\frac{3\sqrt{2} -\sqrt{6}}{4}\right)\right) = 1/9 .\]
%Let
%\[ A = \begin{bmatrix}  \frac{3\sqrt{2} + \sqrt{6}}{4} & \frac{3\sqrt{2} - \sqrt{6}}{4} \\ \frac{3\sqrt{2} - \sqrt{6}}{4} & \frac{3\sqrt{2} + \sqrt{6}}{4} \end{bmatrix} ,\]

Let us introduce two lattices, $\cL = \Z^d \cap \C$ and $\bar \cL = (d/2)^{1/2} \pi_{d-1} H_{\bu}(\cL)$.  Note that $(S_n)_{n\geq 0}$ is a random walk on the lattice $\cL$ and $\pi_{d-1}\circ H_\bu$ is an isometry from $\H$ to $\R^{d-1}$.  Using this, and letting $v_{\cL}$ ($v_{\bar \cL}$) be the volume of a fundamental cell of $\cL$ ($\bar \cL)$ with respect to the appropriate dimensional Hausdorff measure, we see that the volume $v_{\cL}=(d/2)^{-(d-1)/2} v_{\bar \cL}$.

\begin{proposition} \label{grievance two}
The random walk $\bar S_n = \bar X_1+ \cdots + \bar X_n$ and the cone $\tilde \C$ satisfy the hypotheses of \cite{denisov2015random}.  
In particular, 
\begin{enumerate}
\item Letting $\bar X_i = (\bar X_{i,1}, \dots \bar X_{i,d-1})$, we have that $\E(\bar X_{i,j})=0$, $\E(\bar X_{i,j}^2) = 1$, and $\E(\bar X_{i,j}\bar X_{i,k})=0$ for $j\neq k$.  Additionally, each $\bar X_{i,k}$ is a bounded random variable.
\item The random variable $\bar X_i$ is supported on the $\bar \cL$, which is a non-degenerate linear transformation of $\Z^{d-1}$, and $\bar \cL$ is a communicating class for $(\bar S_n)_{n\geq 0}$.
\item $\tilde \C$ is convex and $\tilde U(y) = U(\H((y,0)))$ a harmonic function on $\R^{d-1}$, positive on $\tilde \C$ and vanishing on its boundary, such that for all $y \in \tilde \C$
\[ \P(\bar \tau^B_y >t) \sim \aleph \frac{\tilde U(y)}{t^{d(d-1)/4}},\quad t\to\infty,\]
where $\bar \tau^B_y$ is the first time a standard Brownian motion started from $y$ exits $\tilde \C$ and 
\[ \aleph =  \frac{1}{\prod_{i<j} (j-i)} \frac{2^{3d/2}}{(2\pi)^{d/2}(d!)} \prod_{k=1}^d[\Gamma((k/2)+1)].\]
\end{enumerate}
\end{proposition}

\begin{proof}
The claims in part (1) about the means, (co)variances, and boundedness of the coordinates of $\bar X_i$ are a straightforward computation.  The assertions about $\tilde \C$ and $\bar\cL$ are in parimmediate. 

%Based on this result, it is now straightforward to translate the results of \cite{denisov2015random} to our current context.

We now turn to part (3).  Let $B$ be a standard Brownian motion in $\R^d$ and let $v_1,v_2,\dots, v_d$ be an orthonormal basis for $\R^d$ such that $v_1= d^{-1/2}(1,1,\dots, 1)$.  Then we can express $B$ as
\[ B= \sum_{i=1}^d B_iv_i\]
where $B_1,\dots, B_d$ are independent, standard, one dimensional Brownian motions.  Since 
\[\ip{ \sum_{i=2}^d B_iv_i}{(1,1,\dots,1)} = \sqrt{d}\ip{ \sum_{i=2}^d B_iv_i}{v_1}=0,\]
letting $ B^0 = B_2u_2+\cdots+B_du_d$, we see that $\P( B^0_t\in \H \textrm{ for all } t) =1$.  Furthermore, it is easy to see that if we let $\tau^B_x = \inf\{t : x+B_t\notin \C\}$ and $\tau^0_x = \inf\{t : x+ B^0_t\notin \tilde\C\}$, then $\tau^0_x = \tau^B_x$ since adding or subtracting $B_1v_1$ preserves the relative order of the coordinates.  Furthermore, since $H_{\bu}$ is an isometry, we see that $\bar B = \pi_{d-1} H_{\bu}(B) =  \pi_{d-1} H_{\bu}( B^0)$ is a standard Brownian motion on $\R^{d-1}$.  Thus, if we let $\bar \tau^B_y = \inf\{ t : y+ \bar B \notin \pi_{d-1}H_{\bu}(\C)\}$ and $x = H_{\bu}( (y,0))$ (where $(y,0) = (y_1,\dots, y_{d-1},0) $) then we have that $\bar \tau_y^B = \tau^B_x$.  It follows from \cite[Equation (20)]{GRABINER99}, that 
\[ \P(\bar \tau^B_y >t) = \P(\tau^B_x >t) \sim \aleph \frac{U(x)}{t^{d(d-1)/4}} =\aleph \frac{U(H_{\bu}((y,0)))}{t^{d(d-1)/4}}  . \qedhere\]
\end{proof}

Letting $\tau_x = \inf\{n : x+ S_n\notin \C\}$ and $\bar\tau_x = \inf\{n : x+ \bar S_n\notin \pi_{d-1}H_\bu(\C)\}$, we translate \cite[Theorem 1]{denisov2015random} to our present context, to get the following result.

\begin{theorem}[\cite{denisov2015random}, Theorem 1]
There exists a strictly positive function $\bar V$ such that for all $x$ in the interior of $\pi_{d-1}H_\bu(\C)$ we have
\[ \P(\bar\tau_x >n) \sim \aleph \bar V(x) n^{-d(d-1)/4}\]
Consequently, for $x \in \tilde\C$, we have 
\[ \P(\tau_x >n) \sim \aleph \bar V\left( (d/2)^{1/2} \pi_{d-1}H_\bu(x)\right) n^{-d(d-1)/4}.\]
\end{theorem} 
Motivated by this theorem, we define $V(x) = \bar V\left( (d/2)^{1/2} \pi_{d-1}H_\bu(x)\right)$.
Translating \cite[Theorem 5]{denisov2015random} into the present context thus gives the following result.

\begin{theorem}\label{theorem local limit}
For $x\in \C$,
\begin{equation}
\sup_{y\in \C} \left| v_{\tilde \cL}^{-1}n^{(d-1)(d+2)/4} \P(x+S_n=y, \tau_x>n) - \aleph V(x) h_0 U\left(\frac{\sqrt{d}}{\sqrt{2n}} y \right) e^{-d|y|^2/4n}\right| \rightarrow 0,
\end{equation}
where $h_0$ is chosen so that $p(y) = h_0U(H_\bu((y,0))) e^{-|y|^2/2}$ is a probability density function with respect to Lebesgue measure on $\pi_{d-1}H_\bu(\C)$.
\end{theorem}
%
%Let us introduce two lattices, $\cL = \Z^d \cap \tilde \C$ and $\tilde \cL = \pi_{d-1} H_{\bu}(\cL)$.  Note that $(S_n)_{n\geq 0}$ is a random walk on the lattice $\cL$ and $\pi_{d-1}\circ H_\bu$ is an isometry from $\H$ to $\R^{d-1}$.  Using this, we see that the volume $v_{\cL}$ of a fundamental cell of $\cL$ with respect to $(d-1)$-dimensional Hausdorff measure equals the volume with respect to Lesbesgue measure of a fundamental cell of $\tilde \cL$, %%%%%%%%%%%
%
%
%which is $(3/2) \det(A)$.  Combining all of this gives us the following local central limit theorem.

\begin{theorem}
For $x\in \C$,
\begin{equation}
\sup_{y\in  \C} \left| \frac{n^{(d-1)/2}}{v_{\tilde \cL}} \P(x+S_n=y | \tau_x>n) - h_0U\left(\frac{\sqrt{d}}{\sqrt{2n}} y\right) e^{-d|y|^2/4n}\right| \rightarrow 0.
\end{equation}
\end{theorem}

We remark that this local central limit theorem should be interpreted has taking place on $\H$ relative to the $(d-1)$-dimensional Hausdorff measure.  In particular, 
\[\nu(x)=(d/2)^{(d-1)/2} h_0U\left(\frac{\sqrt{d}}{\sqrt{2}} y\right) e^{-d|y|^2/4},\]
is a probability density function with respect to the $(d-1)$-dimensional Hausdorff measure and we obtain the following corollary.

\begin{corollary}\label{cor end convergence}
Let $A\subseteq \H$ be an open set.  Then
\[ \lim_{n\to\infty} \P\left(\frac{x+S_n}{\sqrt{n}} \in A \middle| \tau_x>n\right) = \int_A\nu(u) du,\]
where $dx$ is the $(d-1)$-dimensional Hausdorff measure on $\H$.
\end{corollary}

\begin{theorem}[\cite{denisov2015random}, Theorem 6%proof on page 50
]
For $x,y\in \C$,
\begin{equation}
\P(x+S_n=y, \tau_x>n) \sim \frac{v_{\cL}^2\aleph^2V(x)V(y)}{n^{(d-1)(d+1)/4}}\int_{\tilde \C} \left(\frac{d}{2}\right)^{d-1} h_0^2 U\left(\frac{\sqrt{d}}{\sqrt{2}} u\right)^2 e^{-d|u|^2/4} du
\end{equation}
If $t\in (0,1)$ and $D\subseteq  \C$ then, letting $[t]$ be the integer part of $t$,  
\[ \P\left(\frac{x+S_{[tn]}}{\sqrt{n}} \in D \middle| x+S_n=y, \tau_x>n \right) \rightarrow  \frac{ \int_{D}  U\left(\frac{\sqrt{d}}{\sqrt{2}} u\right)^2 e^{-d|u|^2/4t(1-t)} du}{ \int_{\tilde \C} U\left(\frac{\sqrt{d}}{\sqrt{2}} u\right)^2 e^{-d|u|^2/4t(1-t)} du}.\]
\end{theorem}

For the next result, we let $S^{x,y,k,n} = (S^{x,y,k,n}_j)_{j=0}^{n-2k}$ be a random variable whose distribution is given by
\[ \P(  S^{x,y,k,n}  \in A) = \P\left( (S_{k+j})_{j=0}^{n-2k}\in A | S_0=x,S_n=y, \tau_x>n\right).\]

\begin{theorem} \label{ice}
If $x \in \tilde \C$ and $y,x',y' \in (x+\Z^m)\cap \tilde \C$ are then
\[\lim_{k\to\infty} \limsup_{n\to\infty} D_{TV}(S^{x,y,k,n},S^{x',y',k,n})=0.\]
\end{theorem}

\begin{proof}
If $s=(s_j)_{j=0}^{n-2k}$ is a path in $(x+\Z^m)\cap \tilde \C$, then the Markov property of $(S_n)_{n\geq 0}$ and a time reversal argument imply that
\[ \P(  S^{x,y,k,n}=\zwalk) = \frac{ \P_{\zwalk(0)}( (S_j)_{j=0}^{n-2k}=\zwalk)\P_x(S_k=\zwalk_0, \tau_x>k)\P_y(S_k=\zwalk_{n-2k},\tau_y>k)}{\P_x(S_n=y, \tau_x>n)}.\]
Consequently,
\[ \frac{\P(  S^{x,y,k,n}=\zwalk)}{ \P(  S^{x',y',k,n}=\zwalk) } =  \frac{ \frac{\P_x(S_k=\zwalk_0, \tau_x>k)}{\P_{x'}(S_k=\zwalk_0, \tau_{x'}>k)}\frac{\P_y(S_k=\zwalk_{n-2k},\tau_y>k)}{\P_{y'}(S_k=\zwalk_{n-2k},\tau_{y'}>k)}}{\frac{\P_x(S_n=y, \tau_x>n)}{\P_{x'}(S_n=y', \tau_{x'}>n)}}\]
Given $\epsilon>0$, using Corollary \ref{cor end convergence} we find a bounded, open set $A$ such that $\dis(A,\partial \tilde C) >0$,  
\[  \int_A \nu(x) dx > 1-\epsilon\]
and
\[ \lim_{n\to\infty} \max_{z\in \{x,x',y,y'\}} \left|\P\left(\frac{z+S_n}{\sqrt{n}} \in A \middle| \tau_z>n\right) - \int_A\nu(u)du\right|=0.\]
It follows from Theorem \ref{theorem local limit} that
\begin{equation}
\max_{z\in \{x,x',y,y'\}}\sup_{y\in  (n^{-1/2}(x+\cL))\cap A} \left| \frac{\frac{n^{(d-1)(d+2)/4}}{v_{\tilde \cL}} \P(z+S_n=\sqrt{n}y, \tau_z>n)}{ \aleph V(z) h_0 U\left(\frac{\sqrt{d}}{\sqrt{2}} y\right) e^{-d|y|^2/4}}-1\right| \rightarrow 0.
\end{equation}
Consequently, we see that
\[ \lim_{k\to\infty} \max_{z,w\in \{x,x',y,y'\}} \max_{u \in (x+\cL) \cap (\sqrt{k} A)} \left| \frac{V(w) \P_z(S_k=u, \tau_v>k)}{V(z)\P_{w}(S_k=u, \tau_{w}>k)} - 1\right| =0.\]
Similarly, 
\[\frac{\P_x(S_n=y, \tau_x>n)}{\P_{x'}(S_n=y', \tau_{x'}>n)} \sim \frac{V(x)V(y)}{V(x')V(y')}. \]
Therefore, 
\begin{multline*} \lim_{k\to\infty} \lim_{n\to\infty} \max_{\zwalk^{n,k}: \zwalk^{n,k}_0,\zwalk^{n,k}_{n-2k}\in (x+\cL) \cap (\sqrt{k} A)} \left| \frac{\P(  S^{x,y,k,n}=\zwalk^{n,k}) }{\P(  S^{x',y',k,n}=\zwalk^{n,k})} -1 \right|\\
=\lim_{k\to\infty} \lim_{n\to\infty} \max_{\zwalk^{n,k}: \zwalk^{n,k}_0,\zwalk^{n,k}_{n-2k}\in (x+\cL) \cap (\sqrt{k} A)}  \left|\frac{ \frac{\P_x(S_k=\zwalk_0, \tau_x>k)}{\P_{x'}(S_k=\zwalk_0, \tau_{x'}>k)}\frac{\P_y(S_k=\zwalk_{n-2k},\tau_y>k)}{\P_{y'}(S_k=\zwalk_{n-2k},\tau_{y'}>k)}}{\frac{\P_x(S_n=y, \tau_x>n)}{\P_{x'}(S_n=y', \tau_{x'}>n)}}-1\right| \\
=\lim_{k\to\infty}  \max_{z, w \in (x+\cL) \cap (\sqrt{k} A)}  \left|\frac{ \frac{\P_x(S_k=z, \tau_x>k)}{\P_{x'}(S_k=z, \tau_{x'}>k)}\frac{\P_y(S_k=w,\tau_y>k)}{\P_{y'}(S_k=w,\tau_{y'}>k)}}{\frac{V(x)V(y)}{V(x')V(y')}}-1\right|\\
= \left| \frac{\frac{V(x)}{V(x')}\frac{V(y)}{V(y')}}{\frac{V(x)V(y)}{V(x')V(y')}} -1 \right| =0.
\end{multline*}
Hence, for $k$ and $n$ sufficiently large,
\begin{multline*} 
\sum_{\zwalk^{n,k}} \left|\P(  S^{x,y,k,n}=\zwalk^{n,k}) - \P(  S^{x',y',k,n}=\zwalk^{n,k})\right|\\
 \leq 4\epsilon +  \sum_{\zwalk^{n,k}: \zwalk^{n,k}_0,\zwalk^{n,k}_{n-2k}\in (x+\cL) \cap (\sqrt{k} A)} \left|\P(  S^{x,y,k,n}=\zwalk^{n,k}) - \P(  S^{x',y',k,n}=\zwalk^{n,k})\right|\\
\leq 4\epsilon + \max_{\zwalk^{n,k}: \zwalk^{n,k}_0,\zwalk^{n,k}_{n-2k}\in (x+\cL) \cap (\sqrt{k} A)} \left| \frac{\P(  S^{x,y,k,n}=\zwalk^{n,k}) }{\P(  S^{x',y',k,n}=\zwalk^{n,k})} -1 \right|
\leq 5\epsilon,
\end{multline*}
which completes the proof.
\end{proof}

\begin{theorem}\label{cone scaling limit}
Suppose that $x\in \C$.  Then for all bounded, continuous functions $f: D([0,1], \R^{d}) \to \R$ we have
\[\E\left( f\left( \frac{x+S_{[n\cdot]}}{\sqrt{2n/d}}\right) \middle| \tau_x >n, S(n)=0 \right)  \rightarrow \E[ f(\Lambda(Z))].\]
\end{theorem}

\begin{proof} 
Suppose that $y\in \tilde \C$.  For $0\leq t\leq 1$ define
\[  \bar S^{(n)}(t) =  \frac{y+\bar S_{[nt]}}{\sqrt{n}}.\]
By \cite[Theorem 4]{duraj2015invariance} there is a process $\tilde B^0 = (\tilde B^0(t), 0\leq t\leq 1)$, called Brownian excursion in $\tilde \C$, such that for bounded and continuous $f: D([0,1], \R^{d-1}) \to \R$, 
\[ \E \left( f\left(\bar S^{(n)}\right) \middle| \bar\tau_y >n, \bar S(n)=0 \right) \rightarrow  \E f(\tilde B^0).\]
Consequently, if for $x\in \C$ we define
\[ S^{(n)}(t) =  \frac{x+S_{[nt]}}{\sqrt{2n/d}},\]
then for bounded and continuous $f: D([0,1], \R^{d}) \to \R$ we have
\[ \E \left( f\left( S^{(n)}\right) \middle| \tau_x >n, S(n)=0 \right) \rightarrow  \E f(H_{\bu}(\tilde B^0,0)).\]

It remains to identify the law of $H_\bu(\tilde B^0,0)$ as the law of the eigenvalues of a traceless Hermitian Brownian bridge.  The law of $\tilde B^0$ is specified in terms of Brownian motion conditioned to remain in $\tilde C$ for all time, denoted $\tilde B^0_>$,  which in turn is defined as an $h$-transform of Brownian motion killed on exiting $\tilde \C$.  Since $H_\bu$ is an isometry, we have that $\tilde B^0_> =\pi_{d-1} H_\bu(B^0_>)$.

Suppose that $0<t<1$ and $f: D([0,1], \R^{d}) \to \R$ is a bounded continuous function such that $f(g)$ depends only on the restriction of $g$ to $[0,t]$.  From \cite[Equations (18), (26)]{duraj2015invariance} (accounting for a time change suppressed in (35)) and the Brownian scaling invariance of $B^0_>$ we have that there is a constant $C_t$ such that 
\[ \E[f(H_{\bu}(\tilde B^0,0))]  = C_t \E[ f(B^{0}_>(\cdot\wedge t)) e^{-|B^{0}_>(t)|^2/(2(1-t))} ],\]
so the result follows from Equation \eqref{Imhof}.
\end{proof}

%Fix $K'$ and $n>2K'$. Recall our definition of a pseudo-metric on paths $s$ and $s'$ in $\Z^d$. 
%We say $D_{K'}(s,s')=0$ if $s(t)=s'(t)$ for all $i \in [K',n-K']$. Otherwise we say $D_{K'}(s,s')=1$.

\begin{lemma} \label{red hen}
Fix $\epsilon>0$ and $K$. Let $M$ be a measure on quadruples 
$(s,x)$ and $(t,y)$ such that with probability one have $0 \leq s,t \leq K$, $x,y \in \cone$ and $|x|,|y|\leq K$. 
For $n>2K$ define $\hat M_n$ to be the measure generated by picking 
$(s,x)$ and $(t,y)$ according to $M$ and then sampling $\walk$ from
$\cw((s,x),(n-t,y))$.
There is a $K'$ such that for any $M$ 
%there is a coupling of $\hat M$ with $\cw ((0,0),(n,0))$ such that
$$D_{K'}\bigg(\hat M_n,(\sloc : \omega \in \cw ((0,0),(n,0)))\bigg)<\epsilon$$
where $D_{K'}$ is defined prior to Corollary \ref{Elsa}.
\end{lemma}
 
 \begin{proof}
 The measure $\hat M_n$ generates a measure $M^*$ on quadruples $(K,x)$ and $(K,y)$. The measure on 
 $\cw ((0,0),(n,0))$ also generates $N^*$, a measure on quadruples of the same form. Both of these measures are linear combinations of a finite number of point masses on quadruples $(K,x)$ and $(K,y)$. By Theorem \ref{ice} we can find a $K''$ that works for any two choices of $(K,x)$ and $(K,y)$ and $(K,x')$ and $(K,y')$. The  lemma follows from taking any coupling of $M^*$ and $N^*$, applying Theorem \ref{ice} and taking linear combinations. 
 %The second claim follows from the first and Theorem \ref{cone scaling limit}.
 \end{proof}

\section{Proofs of results about random walks close to the Weyl chamber}
\label{birs}
In this section we lay out our main theorems about the paths that are the image of pattern avoiding permutations. We show, in some very strong sense, a relationship between the distribution of the path of a uniformly chose pattern avoiding permutation and the distribution of a random walk in $\cone$. 

%For any $k \in N$ we define the pseudo metric $D_{k}$ on paths of length $n$ by
%$$D_{k}(s,s')=\#\{i \in (k,n-k):\ s(i) \neq s'(i)\}.$$

Remember that for $x \in \Z^d$ we defined 
$$U(x)=\prod_{i<j}(x_i-x_j), $$
which is harmonic for the random walk $\sloc(t)$.

We start with the following argument that proves (after a minor alteration) that for any 
$x \in \cone, T \in \Z$ and $l $
$$\prob_{(T,x)} \big(\spath \in \cw((T,x),(T_l,\cdot))  \big) \leq U(x) (2^l)^{-d(d-1)/2}.$$

Let $T^*_l$ be the minimum time $t$ greater than or equal to $T$ such that  $\sloc(t) \not \in \cone$ or $\sloc(t) \in \cone_{2^l}$. This is a stopping time.
By the optional stopping time theorem and the fact that 
$$\text{$U(x)=0$ for all $x\in \partial \cone$}$$
we have that 
\begin{eqnarray*}
U(x)&=&U(T^*_l)\\
&=&\sum_{x'}U(x')\prob_{(T,x)}(\sloc(T^*_l)=x') \\
&\geq& \prob_{(T,x)}(T^*_l  \in \cone_{2^l}) \min_{x' \in  \cone_{2^l}}U(x')
\end{eqnarray*}
 
As $$U(x') \geq (2^l)^{d(d-1)/2}$$ for all $x' \in  \cone_{2^l}$
we have
$$\P_{(T,x)}(T^*_l  \in \cone_{2^l}) \leq U(x)(2^l)^{-d(d-1)/2}.$$ Proposition \ref{grievance} and results in \cite{denisov2015random}  imply that the lower bound is within a constant factor of this upper bound.

This is not rigorous because $\sloc$ changes in two coordinates every time that it changes. 
Thus we can have that $$\text{$\sloc(T^*_l) \not \in \partial \cone \cup \cone_{2^l}$}$$
and $U(\sloc(T^*_l))<0$. 
To account for this we need to bound
$$\sum_{x':U(x')<0}U(x') \prob_{(T,x)}(\sloc(T^*_l)=x')$$ from below. We give a bound that is (in absolute value) much
smaller than $U(x)$.

All of the bounds in this section are some variant of this argument. We define a stopping time $T^*$ such that with very high probability either
\begin{enumerate}
\item $\sloc(T^*) \in \cone_{2^l}$ or 
\item $|\sloc(T^*)|$ is small. 
\end{enumerate}
Then we bound $\prob(\sloc(T^*) \in \cone_{2^l})$ with the optional stopping time theorem.

Now we make the preceding argument rigorous and strengthen it. 
%Let $T_{l_0}=T$.
%For $l>l_0$ w
We define the sequence $$T_l=\min\{\inf \{t: \sloc(t) \in \cone_{2^l}\},\lfloor 4.1^l\rfloor\}.$$

%We will show that if $x \in \partial \cone$ is large then
%$$\prob \bigg(\spath \in \cwplus((T,x),(T_l,\cdot)) \ | \ \sloc(T)=x \bigg) \leq U(x) (2^l)^{-d(d-1)/2}.$$
%{\bf What is the purpose of this?}

Let $B(t)$ be Brownian motion on the submanifold of $\R^d$ such that the sum of all the coordinates is zero. We choose $\gamma$  such that 
\begin{equation} \prob(B(1) \not \in \cone_{2}) < \gamma <1.\label{club} \end{equation} 

\begin{lemma} \label{gerrymander}
For any $l$ sufficiently large and any $t\leq 4.1^{l}$ and $x$ with $\dis(x,\cone)\leq t^{.4}$
$$\prob_{(t,x)}(\spath \in \cw((t,x),(t+4^l,\cdot)) \cap T_l>t+4^l )<\gamma$$ and
$$\prob_{(t,x)}(\spath \in \cwplus((t,x),(t+4^l,\cdot)) \cap T_l>t+4^l)<\gamma.$$
\end{lemma}

\begin{proof} For sufficiently large $l$, if
$\dis(\sloc(t),\cone)\leq t^{.4}$ and $\sloc(t+4^l)-\sloc(t) \in \cone_{2^{l+1}}$ then $\sloc(t+4^l) \in \cone_{2^{l}}$ and
$T_l \leq t+4^l.$ %Therefore if $F_t$ denotes the event $d(\sloc(t),\cone)<t^{.4}$,
$$\prob_{(t,x)}(T_l > t+4^l) \leq \prob(\sloc(t+4^l)-\sloc(t) \notin \cone_{2^{l+1}}).$$  By the convergence of random walk to Brownian motion
$$ \prob_{(t,x)}(\sloc(t+4^l)-\sloc(t) \notin \cone_{2^{l+1}})\leq \prob(B(1) \not \in \cone_2)+\epsilon<\gamma$$
for $\epsilon$ sufficiently small and for all large $l$.
\end{proof}

\begin{lemma}\label{seemsuseful}
For $l$ sufficiently large, for all $j$, all $T< 4.1^l$, and all $x \in \cone$,
\[\prob_{(T,x)}\left( \spath \in \cw((T,x),(T_l,\cdot)), T_l -T \geq j4^l \big | \sloc(T) = x \right) 
 \leq \gamma^{j},\]
 and
 \[ \prob_{(T,x)}\left( \spath \in \cwplus((T,x),(T_l,\cdot)) ,  T_l - T \geq j4^l\big | \sloc(T) = x \right ) 
 \leq \gamma^{j}.\]
\end{lemma}

\begin{proof}
Note that $T_l \leq 4.1^l$ so we only need to check this for $j \leq (4.1/4)^l $.
Thus this lemma follows from repeated applications of Lemma \ref{gerrymander}.  
\end{proof}

The following lemma is a consequence of standard moderate deviations bounds, see e.g.\ \cite[Lemmas 5.1-2]{hoffman2017pattern}.

\begin{lemma}\label{moderate1}
There exist constants $\Theta,\theta>0$ such that for all $j, l$,
\[ \P_{(0,0)}\left(\max_{0\leq i\leq j4^l}|\sloc(i)| \geq j2^l \right) \leq \Theta e^{-\theta j}.\]
\end{lemma} 

Consequently,

\begin{lemma} \label{inauguration}
There exists $\beta \in (0,1)$ such that for all $l$ sufficiently large, $T \leq 4.1^l$, 
$x \in \cone \setminus  \text{INT}(\cone_{2^l})$, 
%$$\prob\left(T_l -T> j4^l \ | \  \spath \in \cw((T,x),(T_l,\cdot))\right) \leq \beta^{j}.$$
%Also 
$$\prob_{(T,x)}\left( |\sloc(T_l)-x| \geq j2^l \ \cap \  \spath \in \cw((T,x),(T_l,\cdot) ) \right) \leq C\beta^{j}.$$
and
$$\prob_{(T,x)}\left( |\sloc(T_l)-x| \geq j2^l \ \cap \  \spath \in \cwplus((T,x),(T_l,\cdot) )\right) \leq C\beta^{j}.$$
%Both statements are also true if we replace $\cw$ with $\cwplus$. 
\end{lemma}

\begin{proof}
If $|\sloc(T_l)-x|>j2^l$ then 
either 
\begin{enumerate}
\item $T_l - T> j4^{l}$ or 
\item $T_l - T\leq {j4^{ l}}$ and $|\sloc(T_l)-x|>j2^l$.
\end{enumerate}
From the first claim in Lemma \ref{seemsuseful} we have that for $l$ sufficiently large
and for all $j$ and all $x \in \cone$
$$\prob_{(T,x)}\left( \spath \in \cw((T,x),(T_l,\cdot)) , T_l - T \geq j4^l  \right) 
 \leq \gamma^{j}.$$
Thus the first event happens with probability bounded by $\gamma^{j}.$  By Lemma \ref{moderate1}, the probability of the second event is bounded, uniformly in $x$, $j$, and $l$ by $Ce^{-cj}$ for some appropriate constants.  Taking $\beta > \max(e^{-c}, \gamma )$ completes the proof of the first claim.

The calculation for $\cwplus$ is done in the same manner. We just use the second part of Lemma \ref{seemsuseful} instead of the first.        
\end{proof}

%By Theorem \ref{kun} for any $x\in \cone$ we have that 
%$$\prob(\spath \in \cw((0,x),(T_l,\cdot)) \geq 
%\prob(\spath \in \cw((0,0),(T_l,\cdot)) \geq 
%C''(4.1^l)^{-\alpha/2}.$$
%Choose $K$ such that 
%\begin{equation}  K \ln(\gamma) < \log(C'')- \alpha \ln(4.1). \label{faculty} \end{equation} 
%Then 
%%for all $j>Kl$ we have  $\gamma^{j/2}<C''(4.1^l)^{-\alpha/2}.$
%\begin{eqnarray*}
%\prob\left(T_l \geq Cj4^l \ \cap \  \spath \in \cw((0,x),(T_l,\cdot))\right) %& \leq & 
%%\frac{
%%\prob\left(\spath \in \cw((0,x),(T_l,\cdot)) \text{ and } T_l \geq Cj4^l\right)\\
%%}
%%{\prob(\spath \in \cw((0,x),(T_l,\cdot)))} \\
%%& \leq & \frac{\gamma^{j}}{C''(4.1^l)^{-\alpha/2}}\\
%& \leq & \gamma^{j/2}.
%\end{eqnarray*}

% Let $\nu$ be a measure on $\Z \times \Z^d.$ Let $R$ be a stopping time 
%Define
%$$\nu_{R}(t',x')=\sum_{(t,x)}\nu(t,x)\cdot\P\bigg( \left\{\sloc(R)
%=x' \right\}\cap \left\{R= t' \right\}\ | \ \sloc(T)=x \bigg).$$

Let $x \in \Z^2$ and $T <4.1^{l_0}$.
%\partial \cone_{2^{l_0}}% \textbf{We use this for general x}
Fix some large integer $L$.  Remember that we have defined
$$R_L=\inf \{t:\ \petrov(t) \cap\{ \spath(t) \in \cone_{2^L}\}\}$$ and 
$$R^*_L= \inf \{t:\ \petrov^*(t) \cap \{ \spath^*(t) \in \cone_{2^L}\}\}.$$
Note that $R_L$ is a stopping time and $R^*_L$ is a stopping time for the reverse walk. 
%We define a sequence of stopping times, $R_l$. Let $R_l=T_l$
%if $$\spath \in \cwplus((T,x),(T_l,\cdot)).$$
%Otherwise we set $R_l$ to be the smallest $r$ such that 
%$$\spath \not \in \cwplus((T,x),(r,\cdot)).$$
Choose $L$ such that 
\begin{equation} \label{cranksgiving}
\sum_{\lambda \in \N}.9^{L}(2\lambda+1)^{-1+d(d-1)/2}
C \beta^{\lambda-1} \leq .01\cdot (.95)^{L},
\end{equation}
and
\begin{equation} \label{cranksgiving2}
 \sum_{m=L}^{\infty} 13^{m(-1+d(d-1)/2)} \gamma^{(1.025)^{m}} \leq 1.
\end{equation}
 
\begin{lemma} \label{charlie}
For any $l_0>L$, $x \in \partial \cone_{2^{l_0}}$ and $T <4.1^{l_0}$
$$\sum_y|U(y)|\P_{(T,x)}\big( \sloc(R_{l_0+1})=y, \ R_{l_0+1} \neq T_{l_0+1}  \big) \leq .01\cdot (.95)^{l_0}U(x).$$
\end{lemma}

\begin{proof}
%We prove the result by induction. First we prove it for $a=1$. 

If  $|y-x| \leq \lambda 2^{l_0}$ then 
for each $i<i'$
$$|y_i -y_{i'}| \leq x_i-x_{i'} +\lambda 2^{l_0} \leq (2 \lambda +1)(x_i-x_{i'}).$$
If %$R_{l_0+1} \neq T_{l_0+1}$ 
in addition $\dis(y,\cone) \leq (4.1^{l_0+1})^{.4}$
then for one $i,i'$ we have $$|y_i-y_{i'}| \leq \dis(y,\cone) \leq (4.1^{l_0+1})^{.4}$$
so 
\begin{eqnarray*}
|U(y)|&=&\prod_{i<i'}|y_i -y_{i'}| \\
&\leq& \frac{(4.1^{l_0+1})^{.4}}{2^{l_0}}(2\lambda+1)^{-1+d(d-1)/2}U(x)\\
&\leq& .9^{l_0}(2\lambda+1)^{-1+d(d-1)/2}U(x)
\end{eqnarray*}
Also for any $\lambda \in \N$ by the portion of Lemma \ref{inauguration} for $\cwplus$
\begin{equation} \label{nameless puncs}
\sum_{y:|y-x| \in ((\lambda -1)2^{l_0},\lambda 2^{l_0}]}
\P_{(T,x)}\big( \sloc(R_{l_0+1})=y , R_{l_0+1}\neq T_{l_0+1} \big) \leq C \beta^{\lambda-1}.
\end{equation}

\begin{eqnarray*}
\lefteqn{\sum_{y}  |U(y)|\P_{(T,x)}\big( \sloc(R_{l_0+1})=y, \ R_{l_0+1} \neq T_{l_0+1}  \big)  } \hspace{.25in} \text{}&&\\
&=&\sum_{\lambda \in \N}\sum_{y:|y-x| \in ((\lambda -1)2^{l_0},\lambda 2^{l_0}]}|U(y)|
\P_{(T,x)}\big( \sloc(R_{l_0+1})=y, \ R_{l_0+1} \neq T_{l_0+1}  \big) \\
&\leq &\sum_{\lambda \in \N}.9^{l_0}(2\lambda+1)^{-1+d(d-1)/2}U(x)\\
 & & \qquad \times\left( \sum_{y:|y-x| \in ((\lambda -1)2^{l_0},\lambda 2^{l_0}]}
\P_{(T,x)}\big( \sloc(R_{l_0+1})=y, \ R_{l_0+1} \neq T_{l_0+1}  \big)\right)\\
&\leq &\sum_{\lambda \in \N}.9^{l_0}(2\lambda+1)^{-1+d(d-1)/2}U(x)
%\sum_{x':|x'-x| \in ((\lambda -1)2^l,\lambda 2^l]}
C \beta^{\lambda-1}\\
& \leq & .01 \cdot (.95)^{l_0}U(x).
\end{eqnarray*}
The last two inequalities are by  \eqref{cranksgiving} and \eqref{nameless puncs}.
%{\bf Why is $U(x')>0$ in the second line?}
\end{proof}

Now we prove two slight variants of Lemma \ref{charlie}.

\begin{lemma} \label{charlie brown}
For any $l_0>L$, $T <\lfloor 4.1^{l_0}\rfloor $ and $x \not \in \cone_{2^{l_0}}$ 
$$\sum_y|U(y)|\P_{(T,x)}\big( \sloc(R_{l_0+1})=y, \ R_{l_0+1} \neq T_{l_0+1}  \big) \leq 2^{l_0}\max(|x|,2^{l_0})^{-1+d(d-1)/2}.$$
\end{lemma}

\begin{proof}
The proof of this lemma is very similar to the proof of Lemma \ref{charlie}. Instead of bounding $U(y)$ with the estimate $|y_{i'}-y_i|\leq |x_{i'}-x_i|+ \lambda 2^{l_0+1}$ we use the estimate $|y_{i'}-y_i| \leq |x|+ \lambda 2^{l_0+1}$ to find that 
\begin{eqnarray*}
 \lefteqn{\sum_{y}  |U(y)|\P_{(T,x)}\big( \sloc(R_{l_0+1})=y, \ R_{l_0+1} \neq T_{l_0+1}  \big)  } \hspace{.25in} \text{}&&\\
 & \leq &\sum_{\lambda \in \N}(4.1^{l_0+1})^{.4} (|x|+\lambda 2^{l_0+1})^{-1+d(d-1)/2} C \beta^{\lambda-1}\\
& \leq&  \max(|x|,2^{l_0})^{-1+d(d-1)/2)}(4.1^{l_0+1})^{.4}  \sum_{\lambda \in \N}(1+\lambda 2^{l_0+1})^{-1+d(d-1)/2} C \beta^{\lambda-1}\\
& \leq  & 2^{l_0}\max(|x|,2^{l_0})^{-1+d(d-1)/2)}.
\end{eqnarray*}
\end{proof}

Let $x \in \Z^d$ and $T <4.1^{l_0}$. %\textbf{We use this for general $x$}
%Let $x \in \partial \cone_{2^{l_0}}$ and $T <4.1^{l_0}$. \textbf{We use this for general $x$}
We define a stopping time, $\hat R_{l_0}$ by setting $\hat R_{l_0}$ to be the smallest $r>T$ such that 
$$\dis(\sloc(r),\partial \cone) \leq r^{.4}$$
if that is less than $T_{l_0}$ and otherwise we set  $\hat R_{l_0}=T_{l_0}$.

\begin{lemma} \label{cassablanca}
For any $l_0>L$, $x \in \partial \cone_{2^{l_0}}$ and $T <4.1^{l_0}$
$$\sum_y|U(y)|\P_{(T,x)}\big( \sloc(\hat R_{l_0+1})=y, \ \hat R_{l_0+1} \neq T_{l_0+1}  \big) \leq .01\cdot (.95)^{l_0}U(x).$$
\end{lemma}

\begin{proof}
The proof is identical to Lemma \ref{charlie}. 
\end{proof}

%Thus the lemma is true for all $l_0$ and $a=1$.
%{\bf We need a version of this with the walk in $\cwplus \setminus \cwminus$.}
%Then we define
%$$ H_{l_0,l}(T,x)
%=.$$

Now we repeatedly apply these three lemmas.
\begin{lemma} \label{chaplain}
%Let $$f(a)=\prod_{j=0}^{a-1}(1+.01\cdot .9^j).$$
For any $l>l_0>L$, $ x \in \partial \cone_{2^{l_0}}$, $T <4.1^{l_0}$, and $|x|\leq 2T$,

\begin{equation} \label{city lights}
\sum_{x'}U(x')
\P_{(T,x)}\big( \spath \in \cwplus((T,x),(T_l,x'))  \big) \leq K U(x)
\end{equation}
where $K= 1+\prod_{j=0}^{\infty}(1+.01\cdot (.95)^{l_0+j})$
\end{lemma}

\begin{proof}
%\textbf{Make sure all uses of the markov property respect that $CW^+$ depends on the starting time}
We prove the lemma by induction.
  Set
$$ H_{l_0,l}(T,x)=\sum_{x'}U(x')
\P_{(T,x)}\big( \spath \in \cwplus((T,x),(T_l,x'))  \big) .$$

We will inductively show that for all $k>l_0$
\begin{equation} \label{little tramp}
%\sum_{x' \in \partial \cone_{2^{l_0+a}}}U(x') \P\big( \spath \in \cwplus((T,x),(T_{l_0+a},x')) \ | \ \sloc(T)=x \big) \leq ??
 H_{l_0,k}(T,x) \leq  U(x) \prod_{j=0}^{k-l_0-1}(1+.02\cdot (.95)^{l_0+j})
 \end{equation}
which implies the lemma.
%The case $a=1$ {\bf should} follows easily from Lemma \ref{charlie}.

Observe that, since $R_{k}$ is a bounded stopping time and $U(\sloc(t))$ is a martingale, it follows from the optional stopping theorem that for each $k>l_0$,
\begin{equation}\label{eq mart0} 
U(x)= H_{l_0,k}(T,x)+\sum_{y}  U(y)\P_{(T,x)}\big( \sloc(R_{k})=y, R_{k} \neq T_{k}  \big) .
\end{equation}

Taking $k=l_0+1$ and applying Lemma \ref{charlie}, the sum in Equation \eqref{eq mart0} is at most $.01(.95)^{l_0}U(x)$ and
$$H_{l_0,l_0+1}(T,x) \leq U(x) (1+.01(.95)^{l_0}) \leq U(x)(1+.02(.95)^{l_0}).$$
This establishes Equation \eqref{little tramp} in the case $k=l_0+1$.

To extend this to all $k>l_0$, we take the difference of Equation \eqref{eq mart0} for consecutive values of $k$ %to compute $\E U(\sloc(R_k)) = \E U(\sloc(R_{k+1}))$ 
to find that  
\begin{multline*}
{H_{l_0,k+1}(T,x)-H_{l_0,k}(T,x)} = \\
 \sum_{y}U(y) \left[ \P_{(T,x)}\big( \sloc(R_k)=y, R_k \neq T_k   \big) - \P_{(T,x)}\big( \sloc(R_{k+1})=y, R_{k+1} \neq T_{k+1}   \big) \right]
\end{multline*}
Decomposing $\big(\sloc(R_{k+1})=y, R_{k+1} \neq T_{k+1} \big)$ based on whether the path exits $CW^+$ before or after $T_k$ we see that, almost surely under $\P(\ \cdot \ |  \sloc(T)=x)$,
\begin{multline*}\big(\sloc(R_{k+1})=y, R_{k+1} \neq T_{k+1} \big) = \\
\big( \sloc(R_k)=y, R_k \neq T_k\big) \cup \big( \sloc(R_{k+1})=y, R_{k+1} \neq T_{k+1}, R_k =T_k\big),
\end{multline*}
and the union is disjoint.  Consequently,
\begin{eqnarray*}
\lefteqn{H_{l_0,k+1}(T,x)-H_{l_0,k}(T,x)}&&\\
& = & - \sum_{y}U(y)  \P_{(T,x)}\big( \sloc(R_{k+1})=y, R_{k+1} \neq T_{k+1}, R_k =T_k   \big) \\
& = & - \sum_{y'}\sum_{y}U(y)  \P_{(T,x)}\big( \sloc(R_{k+1})=y, \sloc(R_k)=y', R_{k+1} \neq T_{k+1}, R_k =T_k   \big)   .
\end{eqnarray*}
%{\bf The first line is by definition of $H_{}$. Why is the second equality true?? The third is just breaking up the sum into smaller parts.}

We break the outer sum up into two parts depending on whether $y' \in \partial \cone_{2^{k}}$ or not.
If not then $T_{k}= \lfloor 4.1^{k}\rfloor$, so that $|y'| \leq |x|+2(4.1)^k$ and $T_k - T \geq \lfloor 4.1 ^{k-1} \rfloor$. 
% and $|U(y')| \leq 4.1^{(l_0+a)d(d-1)/2}$. 
By Lemma \ref{seemsuseful}
\begin{equation} \label{gold rush}
\P_{(T,x)}\big( \spath \in \cwplus((T,x),(T_{k},\cdot)),T_{k} - T \geq \lfloor 4.1^{k-1} \rfloor  \big) \leq \gamma^{(1.025)^{k-1}},
\end{equation}
so by Lemma \ref{charlie brown} the sum is at most 
%{\bf Check where the $2\cdot 4.1$ is coming from}
$$2^{l_0}(|x|+2(4.1))^{k(-1+d(d-1)/2)} \gamma^{(1.025)^{k-1}}.$$

For each $y'$ for which the contribution is not zero and $y' \in \partial \cone_{2^{l_0+a}}$
we get %calculate the inner sum exactly as in the case $a=1$ to get 

\begin{eqnarray*}
\lefteqn{
\sum_{y}  |U(y)|\P_{(T,x)}\big( \sloc(R_{k})=y', \sloc(R_{k+1})=y, R_{k}=T_{k},R_{k+1} \neq T_{k+1}  \big) } &&\\
&\leq &\sum_{y}  |U(y)|\P_{(R_k,y')}\big( \sloc(R_{k+1})=y, R_{k+1} \neq T_{k+1}   \big) 
      \prob_{(T,x)}(\sloc(R_{k})=y')\\
& \leq &
(.01)(.95)^{k} U(y') \prob_{(T,x)}(\sloc(R_{k})=y')
\end{eqnarray*}
When we sum over all $y'$ we get at most
$$(.01)(.95)^{k} H_{l_0,k}(T,x)$$
Combining these two estimates we get
\begin{eqnarray*}
\lefteqn{H_{l_0,k+1}(T,x)-H_{l_0,k}(T,x)} \hspace{1in} \text{} &&\\
& \leq & 
(.01)(.95)^{k} H_{l_0,k}(T,x) +
2^{l_0}(|x|+2(4.1))^{k(-1+d(d-1)/2)} \gamma^{(1.025)^{k-1}}\\
\end{eqnarray*}
Solving this first order linear recurrence with variable coefficients, and using that $U(x)\geq 2^{l_0}$ and $|x|\leq 4.1^{l_0}$, gives the bound 
\begin{eqnarray*}
\lefteqn{H_{l_0,n}(T,x)} &&\\ & \leq& \left(\prod_{k=l_0}^{n-1} (1+ (.01)(.95)^{k})\right)\left( U(x) + \sum_{m=l_0}^{n-1} 2^{l_0}(|x|+2(4.1))^{m(-1+d(d-1)/2)} \gamma^{(1.025)^{m-1}} \right)\\
& \leq & \left(1+ \prod_{k=0}^{\infty} (1+ (.01)(.95)^{l_0+k})\right)U(x)
,\end{eqnarray*}
as desired.
%
%\textbf{How do we get the second inequality? Where does the lower bound on $H_{l_0,k}$ come from?}
%
%\textbf{The sum has to be less than twice the first term (the case we do) or twice the second term. The second term is super exponential in $k$ while the first term is only exponential in $k$. So if the second term is smaller then $H_{l_0,k}$ must be tiny and much less than the bound from induction. But some argument is necessary.}
%
%and 
%\begin{eqnarray*}
%H_{l_0,k+1}(T,x) 
%&\leq & (1+(.02)(.95)^{k}) H_{l_0,k}(T,x)\\
%&\leq & U(x) \prod_{j=0}^{k-l_0}(1+.02\cdot (.95)^{l_0+j})
%\end{eqnarray*}
%which proves the induction hypothesis.
\end{proof}

 For any $l\geq l'>L$, $T<4.1^{L}$, $|x| \leq T$ and  $x \in \partial \cone_{2^{L}}$ 
let $\hat E_1(l',T,x)$ be the event that 
\begin{enumerate}
\item $\spath \in \cwplus((T,x),(T_l,\cdot)$
\item there exists $t \in (T_{l'-1},T_{l'})$ such that $\dis(\sloc(t), \partial \cone) \leq t^{.4}$ and
\item $\dis(\sloc(t), \partial \cone) > t^{.4}$
for all $t \in(T_{l'},T_{l})$. 
\end{enumerate}

%\begin{corollary}
%\end{corollary}

\begin{lemma} \label{magnus}
For any $l>l_0>L$, $T\geq \lfloor 4.1^{l_0}\rfloor$ and $x$ with $|x|\leq 2T$ and $x \notin  \cone_{2^{l_0}}$ 

\begin{equation} \label{carlsen}
\sum_{x'}U(x')
\P_{(T,x)}\big( \spath \in \cwplus((T,x),(T_l,x'))  \big) \leq K( 2T)^{d(d-1)/2}
\end{equation}
where $K= \prod_{j=0}^{\infty}(1+.02\cdot (.95)^{l_0+j})$.
\end{lemma}

\begin{proof}
The proof is virtually identical to Lemma \ref{chaplain}.
\end{proof}

\begin{lemma} \label{we rise together}
 For any $T<4.1^L$, $|x| \leq 2T$, $x \in \partial \cone_{2^L}$ and  $l\geq l'>L$,
 we have that
 $$\sum_{y} U(y) \prob_{(T,x)}\big(\spath \in \hat E_1(l',T,x), \sloc(T_l)=y \big)\leq .04 K^2 \cdot (.95)^{l'-1}U(x).$$
%where $K=\left(1+ \prod_{j=0}^{\infty}(1+.01\cdot .9^j)\right).$
\end{lemma}

\begin{proof}
First we apply Lemma \ref{chaplain} to get
 \begin{equation} \label{prairie lights}
\sum_{x'}U(x')
\P_{(T,x)}\big( \spath \in \cwplus((T,x),(T_{l-1},x'))  \big) \leq KU(x).
\end{equation}

We define two stopping times. First let  $\hat T_1$ be the minimum of $t > T_{l'-1}$
%$4.1^{l'}$ and the minimum $t > T_{l'-1}$ 
such that 
$$\dis(\sloc(t),\partial \cone) \leq t^{.4} \ \ \ \text{or } \ \ \  \sloc(t) \in \cone_{2^{l'}} \ \ \ \text{or } \ \ \  4.1^{l'}.$$
Second let $\hat T_2$ be %the minimum of $4.1^{l'}$ and 
the minimum $t > T_{l'-1} $ 
such that 
$$\dis(\sloc(t), \cone) \geq t^{.4} \ \ \ \text{or } \ \ \  \sloc(t) \in \cone_{2^{l'}}\ \ \ \text{or } \ \ \  4.1^{l'}.$$
If $\hat E_1(l',T,x)$ occurs then the  first stopping time is achieved by the first condition. 
Let $F_1$ be the event that the first stopping time is achieved by the first condition.
%By Lemma \ref{cassablanca} and \eqref{prairie lights} 

\begin{eqnarray*}
\lefteqn{ \sum_{z} U(z) \prob_{(T,x)}\big(\spath \in F_1, \sloc(\hat T_1)=z \big) } &&\\
%& \leq &  \sum_{x'} \sum_{z} U(z) \prob\big(\spath \in F_1, \sloc(\hat T_1)=z, \sloc(T_{l_0-1})=x'\ | \ \sloc(T)=x\big) \\
& = & \sum_{x'} \sum_z  \prob\big(\spath \in F_1, \sloc(\hat T_1)=z \ | \ \spath \in \cwplus((T,x),(T_{l'-1},x')  \big)  \\
&&\hspace{1in}    \prob_{(T,x)}(\spath \in \cwplus((T,x),(T_{l'-1},x') )U(z)\\
& = & \sum_{x' \not \in \partial \cone_{2^{l'-1}}} \prob_{(T,x)}(\spath \in \cwplus((T,x),(T_{l'-1},x') ) \\
&&\hspace{1in}\sum_z  \prob\big(\spath \in F_1, \sloc(\hat T_1)=z \ | \ \spath \in \cwplus((T,x),(T_{l'-1},x')  \big)  U(z) 
   \\
&  & + \sum_{x'  \in \partial \cone_{2^{l'-1}}} \prob_{(T,x)}(\spath \in \cwplus((T,x),(T_{l'-1},x') ) \\
&&\hspace{1in}\sum_z  \prob\big(\spath \in F_1, \sloc(\hat T_1)=z \ | \ \spath \in \cwplus((T,x),(T_{l'-1},x')  \big)  U(z) 
   \\
& \leq &\P(T_{l'-1}=\lfloor 4.1^{l'-1} \rfloor) \cdot \sup |U(\sloc(\hat T_1))|  \\ 
&& \hspace{.25in} + \sum_{x' \in \partial \cone_{2^{l'-1}}} \prob_{(T,x)}(\spath \in \cwplus((T,x),(T_{l'-1},x')) ) \big(.01(.95)^{l'-1} U(x')\big)\\
& \leq & \gamma^{1.025^{l'-1}}(2 \cdot 4.1^{l'})^{d(d-1)/2}  + .01 K \cdot (.95)^{l'-1}U(x)\\
& \leq &  .02 K \cdot (.95)^{l'-1}U(x).
\end{eqnarray*}
The first equality is the decomposition of the event based on  the value $\sloc(T_{l'-1})$. The second half of the first inequality comes from Lemma \ref{cassablanca}. The second inequality comes from \eqref{prairie lights}.

%Let $F_2$ be the event that the second stopping time is achieved by the first condition.
%A similar calculation shows that 
%\begin{eqnarray*}
%\lefteqn{ \sum_{z} U(z) \prob\big(\spath \in F_2, \sloc(\hat T_1)=z \ | \ \sloc(T)=x\big) } &&\\
%%& \leq &  \sum_{x'} \sum_{z} U(z) \prob\big(\spath \in F_1, \sloc(\hat T_1)=z, \sloc(T_{l_0-1})=x'\ | \ \sloc(T)=x\big) \\
%& = & \sum_{x'} \sum_z  \prob\big(\spath \in F_2, \sloc(\hat T_1)=z \ | \ \spath \in \cwplus((T,x),(T_{l'-1},x')  \big)  \\
%&&\hspace{1in}    \prob(\spath \in \cwplus((T,x),(T_{l'-1},x') \ | \ \sloc(T)=x)U(z)\\
%& = & \sum_{x'} \prob(\spath \in \cwplus((T,x),(T_{l'-1},x') \ | \ \sloc(T)=x) \\
%&&\hspace{1in}\sum_z  \prob\big(\spath \in F_2, \sloc(\hat T_1)=z \ | \ \spath \in \cwplus((T,x),(T_{l'-1},x')  \big)  U(z) 
%   \\
%& \leq & \sum_{x'} \prob(\spath \in \cwplus((T,x),(T_{l'-1},x')) \ | \ \sloc(T)=x) \big(.01(.95)^{l'-1} U(x')\big)\\
%& \leq & .01 K \cdot (.95)^{l'-1}U(x).
%\end{eqnarray*}
%The first equality is the decomposition of the event based on  the value $\sloc(T_{l'-1})$. The first inequality comes from Lemma \ref{charlie}. The second inequality comes from \eqref{prairie lights}.

Let $F_2$ be the event that the second stopping time is achieved by the first condition.
A similar calculation (using Lemma \ref{charlie} instead of Lemma \ref{chaplain}) shows that 
\begin{equation*}
\sum_{z} U(z) \prob_{(T,x)}\big(\spath \in F_2, \sloc(\hat T_1)=z \big) 
 \leq  .02 K \cdot (.95)^{l'-1}U(x).
\end{equation*}

Now let $l'=l$. If $\spath \in \hat E_1(l',T,x)$ then $\spath \in F_1 \setminus F_2$. As $U$ is harmonic
 then $U(\sloc(\hat t))$ is a martingale. As $\hat T_2$ is a stopping time
\begin{eqnarray*}
\lefteqn{\big| \sum_{x'}U(x')
\P_{(T,x)}\big( \spath \in \cwplus((T,x),(\hat T_2,x'))  \big)\big|}\hspace{1in}&&
\\
&=&\big|\sum_{x''}U(x'')
\P_{(T,x)}\big( \spath \in \cwplus((T,x),(\hat T_1,x''))  \big)\big|\\
%\end{multline*}
%
%\begin{align*}
%\sum_{x'}&U(x') \P\big( \spath \in \cwplus((T,x),(T_l,x')) \ | \ \sloc(T)=x \big) \\
&\leq&  \big |  \sum_{z} U(z) \prob_{(T,x)}\big(\spath \in F_1, \sloc(\hat T_1)=z \big)  \big | \\
&&+ \big | \sum_{z} U(z) \prob_{(T,x)}\big(\spath \in F_2, \sloc(\hat T_1)=z \big)  \big | \\
 &\leq&  .04 K \cdot (.95)^{l'-1}U(x).
\end{eqnarray*}
Thus the lemma is proven for $l=l'$. 

For $l>l'$ we need to another decomposition of the path based on
where it is at $T_{l'}$ and apply Lemma \ref{chaplain}. 
%{\bf Do we need to include this computation as well?}
\end{proof}

\begin{lemma} \label{benalla}
There exists a function $H(L)=o(1)$ such that  for any $L, l>L$, $T<4.1^L$, $|X| \leq 2T$ and  $x \in \partial \cone_{2^L}$ 
\begin{equation} \label{les enfants}
\sum_y U(y)\prob_{(T,x)}(\spath \in \cwpplus((T,x),(T_l,y))\setminus \cwminus((T,x),(T_l,y))) \leq H(L)U(x).
\end{equation}
\end{lemma}

\begin{proof}
Define the sum on the left hand side of \eqref{les enfants} to be $\msum$.
Let $J$ be the largest $j$ such that $T_j=\lfloor 4.1^j \rfloor$ and
let $M$ be the largest $t$ such that $\petrov(t)^C$ occurs.
If there is no such $j$ (resp.  $t$) then we say $J=\infty$ ($M=\infty$).
If $$\spath \in \cwpplus((T,x),(T_l,y))\setminus \cwminus((T,x),(T_l,y))$$ then 
either 
\begin{enumerate}
\item $J=M =\infty$,
\item $T_{J+1}>M$ or 
\item $M \geq T_{J+1}.$
\end{enumerate}

We define the events $E_1$, $E_2$ and $E_3$ to be events that
$$\spath \in \cwpplus((T,x),(T_l,y))\setminus \cwminus((T,x),(T_l,y))$$
and (respectively) the first, second or third of those possibilities occurs.
For $i=1,2,3$ we write
\begin{equation} \label{joyeux}
\msum_i= \sum_y U(y)\prob_{(T,x)}(\1_{E_i}, \spath \in \cwpplus((T,x),(T_l,y))\setminus \cwminus((T,x),(T_l,y)) ).
\end{equation}

Note that $$E_1 \subset \bigcup_{l'\in (L,l]} \hat E_1(l',T,x) \subset \cwplus((T,x),(T_l,\cdot).$$
As $x \in \partial \cone_{2^L}$ by the above and Lemma \ref{we rise together} we can bound
\begin{eqnarray*}
\msum_1 & \leq& \sum_{l' \in (L,l]} \sum_{y} U(y) \prob_{(T,x)}\big( \hat E_1(l',T,x), \sloc(T_l)=y \big)\\
&\leq& \sum_{l' \in (L,l]} .04 K^2 \cdot (.95)^{l'-1}U(x) \\
&\leq& K^2(.95)^LU(x).
\end{eqnarray*}

If $E_2 \cap \{J=j\}$ occurs then 
\begin{enumerate}
\item $\spath \in \cwplus((T_{j+1},\cdot),(T_l,\cdot))$ \label{itemone}
\item $\sloc(T_{j+1}) \in \partial \cone$ \label{itemtwo} and
\item $T_{j+1} <\lfloor 4.1^{j+1}\rfloor$. \label{itemthree}
\end{enumerate}
Thus
\begin{equation} U(\sloc(T_{j+1})) \leq (2 \cdot 4.1)^{(j+1)d(d-1)/2} \label{itemfour} \end{equation}
and  by Lemma \ref{seemsuseful} for  $j>L$   we get that 
\begin{eqnarray}
 \sum_{y'} \P_{(T,x)}\big(E_2,  J=j, \sloc(T_{j+1})=y' \big)
&\leq& \P_{(T,x)}(E_2, J=j ) \nonumber \\ 
& \leq & \P_{(T,x)}(T_j=\lfloor 4.1^j \rfloor )  \nonumber\\
&\leq& \gamma^{1.025^{j-1}}. \label{no confidence}
 \end{eqnarray}

We then have, \begin{eqnarray*}
\lefteqn{\sum_{y''} U(y'') \prob_{(T,x)}\big(E_2,  J=j, \sloc(T_{l})=y'' \big) } && \\
&\leq & 
K\sum_{y} U(y) \prob_{(T,x)}\big(E_2,  J=j, \sloc(T_{j+1})=y \big) \\
&\leq & K (2 \cdot 4.1)^{(j+1)d(d-1)/2} \sum_{y'} \P_{(T,x)}\big(E_2,  J=j, \sloc(T_{j+1})=y' \big)\\
&\leq & K (2 \cdot 4.1)^{(j+1)d(d-1)/2}  \gamma^{1.025^{j-1}}\\
& \leq & (.95)^j
\end{eqnarray*}
The first inequality comes from Lemma \ref{chaplain}. The second comes from \eqref{itemfour} and the third from
 \eqref{no confidence}. The fourth holds because $j \geq L$ and $L$ is large. Thus

\begin{eqnarray*}
\msum_2 
& = & \sum_{j \in (L,l)} \sum_{y''} U(y'') \prob_{(T,x)}\big(E_2,  J=j, \sloc(T_{l})=y'' \big)\\
&\leq& \sum_{l' \in (L,l]} (.95)^{j} \\
&\leq& 20(.95)^L
\end{eqnarray*}

The bound for $\msum_3$ is similar to the bound for $\msum_2$. 
If $E_3 \cap \{M=m\} $ let $J^*$ be such that $T_{J^*+1} < M \leq T_{J^*+2}$ occurs. Then 
\begin{enumerate}
\item $\spath \in \cwplus((T_{J^*+2},\cdot),(T_l,\cdot))$, \label{mad}
\item $T_{J^*} <\lfloor 4.1^{J^*+2}\rfloor$ and 
\item $\sloc(T_{J^*+2}) \in  \partial \cone_{2^{J^*+2}}$.
\end{enumerate}
Also $2^{J^*+1}\leq m \leq 4.1^{J^*+2}$. So for any fixed $m$ there are at most $\log(m)$ possibilities for $J^*$.
From this we get that $T_{J^*+2} \leq m^3$.
This implies that $U(T_{J^*+2}) \leq (2m^3)^{d(d-1)/2}$ and there are at most $(2m^3)^{d(d-1)/2}$ possibilities for $T_{J^*+2}$.

For any $m$
\begin{eqnarray*}
\lefteqn{\sum_j \sum_{y''} U(y'') \prob_{(T,x)}\big(E_3,  J^*=j, M=m, \sloc(T_{l})=y'' \big) } && \\
&= &  \sum_j \sum_{z} \sum_{y''} U(y'') \prob_{(T,x)}\big(E_3,  J^*=j, M=m, \sloc(T_{l})=y'' , \sloc(T_{j+2})=z \big) \\
&= &  \sum_j \sum_{z} \sum_{y''} U(y'') \prob_{(T,x)}\big(E_3, J^*=j,  \sloc(T_{l})=y''  \ |  \ \sloc(T_{j+2})=z, M=m\big) \\
&& \hspace{3in}\cdot\P(\sloc(T_{j+2})=z, M=m ) \\
&\leq &  \sum_j \sum_{z} \sum_{y''} U(y'') \prob_{(T,x)}\big(\cwplus((T_{j+2},z),(T_l,y'')) \ | \  \sloc(T_{j+2})=z \big) \P(\petrov(m)^C )\\
&\leq & \sum_j  \sum_z KU(z) e^{-m^{\delta^*}}  \\
&\leq & \log(m) (2m^3)^{d(d-1)/2}  K (2m^3)^{d(d-1)/2}  e^{-m^{\delta^*}} \\
& \leq & e^{-m^{\delta^*}/2}.
\end{eqnarray*}
The first inequality comes from \eqref{mad} and the Markov property of $S$. The second comes from Lemma \ref{chaplain} and Lemma \ref{pathpetrov}. 
The third inequality comes from the estimates in the above paragraph.

So we get 
\begin{eqnarray*}
\msum_3 & \leq&  \sum_{m \in [L,l)} \sum_j \sum_{y''} U(y'') \prob_{(T,x)}\big(E_3,  J^*=j, M=m, \sloc(T_{l})=y'' \big)\\
& \leq&  \sum_{m \in [L,l)}e^{-m^{\delta^*}/2}\\
& \leq & 20(.95)^L.
\end{eqnarray*}

Combining these three bounds we get
\begin{eqnarray*}
\msum & \leq& \msum_1+\msum_2+\msum_3\\
&\leq& K^2(.95)^LU(x) +20(.95)^L+20(.95)^L\\
&\leq& 50K^2(.95)^LU(x) 
\end{eqnarray*}
Setting $H(L)=50K^2(.95)^L$ completes the proof.
\end{proof}

%\newpage
\subsection{Symmetric versions of these sets of paths.}

%Fix some small $\epsilon>0$ according to $4.1^L<.01n$??
%{\bf This notation is bad. I think it should this be $L_f$?}
Choose $\Lf$ to be the smallest integer such that 
\begin{equation} \label{back flip}
4^\Lf>n^{1-.08/d(d-1)}.
\end{equation} 
Also let $\delta = .02/d(d-1) < .01.$  Then $2^\Lf > n^{.5-2\delta} > n^{.49}.$
Thus

For any $j,k<n/2$ and $x,y\in \Z^d$ define 
 $\scwminus((j,x),(n-k,y))$ to be any path such that
\begin{enumerate}
\item $\sloc (i) \in \cwminus((j,x),(\lfloor n/2 \rfloor,\cdot)$ and
\item $\sloc (n-i) \in \cwminus((k,y),(\lfloor n/2 \rfloor,\cdot)$
\end{enumerate}
$\scwplus((j,x),(n-k,y))$ and $\scwpplus((j,x),(n-k,y))$ are defined in an analogous way.

\begin{lemma} \label{kiera}

There exist $C$ such that for all $n$, all $T,T^*>n^{.5-2\delta}$, $T+T^*<n/2$ and 
for all $x,y \in \Kstar$ such that $|x|,|y| \leq n^{.5-\delta}$
\begin{equation} \label{raz}
\prob_{(T,x)}( \spath \in \scwpplus((T,x),(n-T^*,y)) \setminus \scwminus((T,x),(n-T^*,y))  )%\ | \ \sloc (T)=x,\sloc (n-T^*)=y)$$
\end{equation}
$$\leq CU(x)U(y)n^{-d(d-1)/2}  \cdot n^{-(d-1)/2}\cdot n^{-\delta}.$$
\end{lemma}

\begin{proof}
The conditions on $x$ and $y$ imply that $U(x),U(y)\geq n^{(.5-2\delta)d(d-1)/2}$ and
%and our choice of $\delta$ imply that {\bf By what result about $U$?} %we can choose $\delta$ such that  
$$U(x)U(y)n^{-d(d-1)/2}  \geq n^{-2(d-1)d\delta} \geq C n^{-.04}.$$
So it sufficient to show that  \eqref{raz} is less than 
$Cn^{-.05} \cdot n^{-(d-1)/2}.$
%$n^{-(2d(d-1)+1)\delta} \cdot n^{-(d-1)/2}.$

There are two ways a path could be in 
$$ \scwpplus((T,x),(n-T^*,y)) \setminus \scwminus((T,x),(n-T^*,y)).$$
First it could fail a Petrov condition. The values of $T$ and $T^*$ are such that probability of this 
is at most $e^{-n^c}$ for some $c>0$.
The other possibility is that the path gets close to $\partial \cone$ but never goes too far away from $\cone$. To bound the probability of this we break this event up into parts. First we note that either there exists $t \leq n/2$ such that
$\dis(\sloc (t), \partial \cone) \leq d^2 t^{.4}$ or 
there exists $t \geq n/2$ such that
$\dis(\sloc (t), \partial \cone) \leq d^2(n-t)^{.4}$. We consider the first possibility. The latter case is identical.

Let $T_1$ be the first time after $T$ where $\dis(\sloc (T_1),\partial \cone) \leq d^2t^{.4}$. By assumption $T_1 \leq n/2$. Also let $T_2$ be the first time
after $T_1$ such that  $\dis(\sloc (T_2),\partial \cone) \geq n^{.45}$. We will split the case up into two cases. The first is that $T_2<n-T^*-.01n$.
Start a path at $\sloc (T_1)$. By a martingale argument (restricting the path to the dimension where it is initially closest to $\partial \cone$) the probability that it hits
$\cone_{n^{.45}}$ before the distance to $\cone$ is at least $n^{.4}$ is at most $2d^2n^{-.05}$. Now for any value of $\sloc (T_2)$ the probability that  the path hits $y$ at time $n-T^*$ is comparable to 
$\prob(\sloc (n-T^*-T_2)=x-y)$. Thus the probability that this happens is at most
$$C2d^2n^{-.05}n^{-(d-1)/2}.$$
%If $\hat T_2-T>.01n$ then we get a similar bound as before.

Otherwise the path spent an interval of time of at least $.4n$ without venturing more than $n^{.45}$ from $\partial \cone$. The probability of this is 
at most $e^{-n^c}$ for some $c>0$. We can see this by breaking up the time interval into pieces of 
length $n^{.9}$. In each interval the path has a bounded from below positive probability of straying more than 
$n^{.45}$ from $\partial \cone$. The bound on the probability of each event is independent of whether 
the preceding events occurred.

Combining these estimates proves the lemma.
\end{proof}

\begin{lemma} \label{current}
There exist $C''$ such that for all $n$, $R \in [n/2,n]$, $T \leq n/4$ and 
for all $x,y \in \Kstar$ such that $|x|,|y| \leq n^{.5-\delta}$,

$$
\prob_{(T,x)}(\spath \in \cw((T,x),(R,y)) )
\geq C''U(x)U(y)n^{-d(d-1)/2}  \cdot n^{-(d-1)/2}. $$
There also exists $C'''$ such that for all $n$, all $R \in [n/2,n]$ and 
for all $x,y \in \Kstar$ such that $|x|,|y| \leq n^{.5-\delta}$
$$
\prob_{(T,x)}(\spath \in \cw((T,x),(R,y)) )
\leq C'''U(x)U(y)n^{-d(d-1)/2}  \cdot n^{-(d-1)/2}. $$
\end{lemma}

\begin{proof}
%{\bf We need to  either define the $\tau_x$ notation or get rid of it. 
Define $\tau_x$ to be the minimum time $t$ such that
$x+\sloc (t) \not \in \cone$.

Find $b>0$ and $V \subset \R^d$ satisfying the following. Let $V_n=\Z^d \cap \sqrt{n} V$.
For all $r,s \in V_n$ we have
$$ \prob(r+\sloc (.1n)=s)>bn^{-(d-1)/2}.$$

Get $a,C$ from Proposition \ref{grievance} and  \cite{denisov2015random} Lemma 29 . Find $u$ such that 
$Ce^{-au^2}<b/2$. Let $V'$ be a translate of $V$ such that 
$$D_{}(V',\partial \cone) >u.$$
Let $V'_n=\Z^d \cap \sqrt{n} V'$.
Then by  Proposition \ref{grievance} and Lemma 29 of \cite{denisov2015random}   for all $n$ sufficiently large and all $r,s \in V_n'$ we have
$$ \prob(r+\sloc (.1n)=s, \tau_r<.1n)<Ce^{-au^2}n^{-(d-1)/2}<(b/2)n^{-(d-1)/2}.$$

Then for all $n$ sufficiently large and all $r,s \in V_n'$ we have
$$ \prob(r+\sloc (.1n)=s, \tau_r \geq .1n) >(b/2)n^{-(d-1)/2}.$$

Let $x,y \in \Kstarred$.
By Proposition \ref{grievance}  and Lemma 20 of \cite{denisov2015random}  we have that there exists $C'>0$ such that for all $R \in [n/2,n]$
$$ \prob(\tau_x \geq .45 n-(n-R)/2, x +\sloc (.45n-(n-R)/2) \in V_n')  \geq C'U(x)n^{-d(d-1)/4}$$
and 
$$ \prob(\tau_y \geq .45 n-(n-R)/2, y +\sloc (.45n-(n-R)/2) \in V_n')  \geq C'U(y)n^{-d(d-1)/4}.$$

Putting this together we have that for all $x,y \in \Kstarred$
\begin{eqnarray*}
\prob(\tau_x>R, x+\sloc (R)=y)
& \geq & \prob(\tau_x \geq .45 n-(n-R)/2, x +\sloc (.45n-(n-R)/2) \in V_n') \\
&& \cdot \prob(\tau_y \geq .45 n-(n-R)/2, y +\sloc (.45n-(n-R)/2) \in V_n')\\
&&\cdot \min_{r,s \in V'_n} \prob(\tau_r \geq .1 n, r +\sloc (.1n)=s)\\
& \geq & C'U(x)n^{-d(d-1)/4} \cdot C'U(y) n^{-d(d-1)/4}\cdot (b/2)n^{-(d-1)/2}\\ 
& \geq & C''U(x)U(y)n^{-d(d-1)/2}  \cdot n^{-d/2}.\\ 
\end{eqnarray*}
As this holds for all $n$ sufficiently large we can find a constant for which it holds for all $n$.
The upper bound follows from Proposition \ref{grievance} and Lemma 28 of \cite{denisov2015random}. %\textbf{We should justify this more}.
\end{proof}

\begin{lemma} \label{spontini2}
For any $\epsilon>0$ there exists $l$ such that if $v,v' \in \cone_{2^l}$ and
$T,T^* \leq (4.1)^l$ then for any $n$ sufficiently large
$$\prob_{(T,v)}\bigg( \scwminus((T,v),(n-T^*,v')) \ \bigg| \ \scwpplus((T,v),(n-T^*,v'))  \bigg)>1-\epsilon.$$
This implies 
$$\prob_{(T,v)}\bigg( \scwminus((T,v),(n-T^*,v')) \ \bigg| \ \cw((T,v),(n-T^*,v'))  \bigg)>1-\epsilon.$$
\end{lemma}

\begin{proof}
Recall $\Lf$ defined in \eqref{back flip}.
If $$\spath \in  \scwpplus((T,v),(n-T^*,v')) \setminus \ \scwminus((T,v),(n-T^*,v')) $$
then either 

\begin{enumerate}
\item $T_\Lf>.01n$ \label{apt}
\item $T_\Lf^*>.01n$ \label{pupil}
\item $|\sloc (T_\Lf)|>n^{.5-\delta}$ or \label{heat}
\item $|\sloc^*(T^*_\Lf)|>n^{.5-\delta}$ \label{wave}
%\end{enumerate}

\noindent
or neither of those happen but at least one of the following occurs.
%\begin{enumerate}
%\item[(5)] 
\item $\spath \in  \scwpplus((T,v),(n-T^*,v')) \setminus \cwminus((T,v),(T_\Lf,\cdot))$ \label{not cooling}
%\item[(6)] 
\item $\spath^* \in  \scwpplus((T^*,v'),(n-T,v)) \setminus \cwminus((T^*,v'),(T^*_\Lf,\cdot))$ \label{off at}
%\item[(7)] 
\item \label{night} $\spath \in  \scwpplus((T,v),(n-T^*,v')) \setminus \scwminus((T_\Lf,\cdot),(n-T^*_\Lf,\cdot))$ and 
\end{enumerate}
$$\spath \in  \cwminus((T,v),(T_\Lf,\cdot))$$
and $$\spath^* \in   \cwminus((T^*,v'),(T^*_\Lf,\cdot)).$$

Thus it is sufficient to show that each of these seven sets have probability that is small in comparison with the probability of 
$\cw((T,v),(n-T^*,v')).$ By the lower bound in Lemma \ref{current}  this is of order $U(v)U(v')n^{-d(d-1)/2}n^{-(d-1)/2}$.

First we show that the events in \eqref{apt} and \eqref{pupil} and \eqref{heat} and \eqref{wave} 
%$\sloc (T_l)>n^{.5-\delta}$ and $\sloc^*(T^*_l)>n^{.5-\delta}$ 
have probability at most $Ce^{-n^{\eta}}$  for some $C$ and $\eta>0$.
The probability of the event in \eqref{apt} is at most $Ce^{-n^{\eta}}$ by Lemma \ref{seemsuseful}.
% by  taking $x=0$, $T=0$ and an appropriate choice of $j$. 
The probability of the event in \eqref{heat} is at most $Ce^{-n^{\eta}}$ by Lemma \ref{inauguration}.
% by  taking $x=0$, $T=0$ and an appropriate choice of $j$. 
The argument for the events in \eqref{pupil} and \eqref{wave} are the same by symmetry. 

%$$\delta_l U(v)(2^{l})^{-d(d-1)/2}\cdot  U(w)(2^l)^{-d(d-1)/2}\cdot $$
Next we bound the probability of the event in \eqref{not cooling}.
We break $\partial \cone_{2^\Lf}$ into disjoint sets 
$$D_i=\{ x \in \partial \cone_{2^{\Lf}}:\  U(x)2^{-\Lf d(d-1)/2} \in [i,i+1)\}$$
for $i \in \N$.

For each $i$ and $j$ and $x \in D_i$ and $y \in D_j$
 by the upper bound in Lemma \ref{current}
\begin{multline} \label{roland}
\prob_{(T_{L_f},x)}(\spath \in \cw((T_\Lf,x),(n-T^*_\Lf,y)) ) \\
\leq C(i+1)2^{\Lf d(d-1)/2}(j+1)2^{\Lf d(d-1)/2}n^{-(d-1)/2} n^{-d(d-1)/2}.
\end{multline}
By Lemma \ref{kiera} and the second half of Lemma \ref{current} we get 
\begin{multline} \label{garros}
\prob_{(T_{L_f},x)}(\spath \in \scwpplus((T_\Lf,x),(n-T^*_\Lf,y)) ) \\
\leq C' i2^{\Lf d(d-1)/2}j2^{\Lf d(d-1)/2}n^{-(d-1)/2} n^{-d(d-1)/2}.
\end{multline}
This bound is uniform over all $x \in D_i$ and $y \in D_j$ and value of $T_{L_f}$ and $T^*_{L_f}$.
%{\bf should we change the conclusion of the second half of Lemma \ref{current} to be the above.}

By Lemma \ref{benalla} we get 
\begin{multline*}\sum_i \prob_{(T,v)}(\spath \in \cwpplus((T,v),(T_\Lf,\cdot)) \setminus \cwminus((T,v),(T_\Lf,\cdot))), \sloc (T_\Lf) \in D_i )
\\\cdot i2^{\Lf d(d-1)/2} 
\leq \epsilon U(v)
\end{multline*} 
and by  Lemma \ref{benalla} and Lemma \ref{chaplain}
$$\sum_j \prob(\spath^* \in \cwpplus((T^*,w),(T^*_\Lf,\cdot)), \sloc^*(T^*_\Lf) \in D_j \ | \ \sloc^*(T^*)=v')
j2^{\Lf d(d-1)/2} \leq  \mathbf{C} U(w).$$

We can sample the paths with $\sloc (T)=v$ and $\sloc (n-T^*)=v'$ as follows. First we sample $\spath^*$ with $\sloc^*(T^*)=v'$ and find $T^*_\Lf$. 
Then we can (independently) sample $\spath$ from time $T$ to $n-T^*_\Lf$. If $\sloc^*(T^*_\Lf)=\sloc (n-T^*_\Lf)$ then we concatenate the paths.

$$\P_{(T,v)}(\spath \in  \scwpplus((T,v),(n-T^*,v')) \setminus \ \scwminus((T,v),(n-T^*,v')) %, \sloc^*(n-T^*)=w
	) \hspace{2in} \text{}$$
$$=\sum_{i,j}\sum_{x \in D_i,y \in D_j,t,t^*} \prob_{(T,v)} \bigg(T_\Lf=t,T^*_\Lf=t^*,\spath \in \cwpplus((T,v),(t,x)) \setminus \cwminus((T,v),(t,x))), \hspace{2in} \text{}$$
$$\spath^* \in \cwpplus((T^*,v'),(t^*,y)),  \spath \in \scwpplus((t,x),(n-t^*,y)),
$$
$$ \big| \ \sloc (T)=v, \sloc^*(T^*)=v' \bigg) {} \hspace{.1in}
$$
$$\leq \sum_{i,j} \prob_{(T,v)} \bigg(\spath \in \cwpplus((T,v),(T_\Lf,\cdot)) \setminus \cwminus((T,v),(T_\Lf,\cdot))), \sloc (T_\Lf) \in D_i ,\hspace{2in} \text{}$$
$$\spath^* \in \cwpplus((T^*,v'),(T^*_\Lf,\cdot)), \sloc^*(T^*_\Lf) \in D_j \big| \ \sloc (T)=v, \sloc^*(T^*)=v' \bigg) \hspace{-.6in} \text{}$$
$$\sup_{x \in D_i, y \in D_j,t,t^*}\P_{(t,x)}\bigg( \spath \in \scwpplus((t,x),(t^*,y)),
  \bigg) {} \hspace{0.15in}
$$
$$\leq \sum_{i,j} \prob_{(T,v)} \bigg(\spath \in \cwpplus((T,v),(T_\Lf,\cdot)) \setminus \cwminus((T,v),(T_\Lf,\cdot))), \sloc (T_\Lf) \in D_i ,\hspace{2in} \text{}$$
$$\spath^* \in \cwpplus((T^*,v'),(T^*_\Lf,\cdot)), \sloc^*(T^*_\Lf) \in D_j \ \big| \ \sloc (T)=v, \sloc^*(T^*)=v' \bigg) \hspace{1in} \text{}$$
$$  C' i2^{\Lf d(d-1)/2}j2^{\Lf d(d-1)/2}n^{-d(d-1)/2}n^{-(d-1)/2} \hspace{1in} \text{} $$
The last line comes from \eqref{garros}. Let 
\[ \Pi =  \cwpplus((T,v),(T_\Lf,\cdot)) \setminus \cwminus((T,v),(T_\Lf,\cdot))) \]
and
\[ \Pi^* =  \cwpplus((T^*,v'),(T^*_\Lf,\cdot)) \setminus \cwminus((T^*,v'),(T^*_\Lf,\cdot))) .\]
Using the independence of $\spath$ and $\spath^*$ we find that
\begin{multline*}
\P_{(T,v)}(\spath \in  \scwpplus((T,v),(n-T^*,v')) \setminus \ \scwminus((T,v),(n-T^*,v')) ) \\
\leq C' n^{-d(d-1)/2}n^{-(d-1)/2} 
 \sum_{i} i2^{\Lf d(d-1)/2} \prob_{(T,v)} \bigg(\spath \in\Pi , \sloc (T_\Lf) \in D_i \bigg) \\
\times \sum_{j}j2^{\Lf d(d-1)/2} \prob \bigg(\spath^* \in \Pi^*, \sloc^*(T^*_\Lf) \in D_j | \sloc^*(T^*)=v' \bigg)\\
\leq  C'' \epsilon U(v)U(v')n^{-d(d-1)/2}n^{-(d-1)/2}
\end{multline*}

%
%
%$$ \P(\spath \in  \scwpplus((T,v),(n-T^*,w)) \setminus \ \scwminus((T,v),(n-T^*,w)) |\sloc (T)=v%, \sloc^*(n-T^*)=w
%$$
%$$ \leq %\textbf{Don't have the independence needed for this } 
% C' n^{-d(d-1)/2}n^{-(d-1)/2} \hspace{4in} \text{}$$
%$$\cdot \sum_{i} i2^{\Lf d(d-1)/2} \prob \bigg(\spath \in \cwpplus((T,v),(T_\Lf,\cdot)) \setminus \cwminus((T,v),(T_\Lf,\cdot))), \sloc (T_\Lf) \in D_i | \sloc (T)=v\bigg)$$
%$$\cdot \sum_{j}j2^{\Lf d(d-1)/2} \prob \bigg(\spath^* \in \cwpplus((T^*,w),(T^*_\Lf,\cdot)) \setminus \cwminus((T^*,w),(T^*_\Lf,\cdot))), \sloc^*(T^*_\Lf) \in D_j | \sloc^*(T^*)=w \bigg)$$
%$$ \leq  C'' \epsilon U(v)U(w)n^{-d(d-1)/2}n^{-(d-1)/2}\hspace{3in} \text{}$$
%%\textbf{Oy, I guess we need to know what. $C$ depends on} 
%which is small in comparison with 
%$$\P(\spath \in  \scwpplus((T,v),(n-T^*,w))  |\sloc (T)=v)\geq C'''  U(v)U(w)n^{-d(d-1)/2}n^{-(d-1)/2}$$
%%\textbf{Which I guess we have a lower bound on from Lemma \ref{current}? I suppose we need to see that none of these C, C'' constants depend on anything they are not supposed to.}

%Combining the inequalities above and Lemma \ref{current}
%the probability of the first event is at most
%$$\epsilon U(v)U(w)n^{-d(d-1)/2}n^{-(d-1)/2}.$$

The bound for the event in \eqref{off at} is identical by symmetry.
For the event in \eqref{night} we calculate the  probability of  
$$\spath \in  \cwpplus((T,v),(T_\Lf,\cdot)) \cap \sloc (T_\Lf) \in D_i$$
and $$\spath^* \in   \cwpplus((T^*,v'),(T^*_\Lf,\cdot))\cap \sloc^*(T^*_\Lf) \in D_j$$
and the maximum of $x \in D_i$ and $y \in D_j$ of the probability of 
$$ \scwpplus((T_\Lf,x),(n-T^*_\Lf,y)) \setminus \scwminus((T_l,\cdot),(n-T^*_l,\cdot))$$
and then summing up over $i$ and $j$ as before.

By Lemma \ref{benalla} we have
$$  \sum_{i} (i+1)2^{\Lf d(d-1)/2} \prob_{(T,v)} \bigg(\sloc (T_\Lf) \in D_i, \spath \in \cwpplus((T,v),(T_\Lf,\cdot)\bigg) \leq 2  U(v)
2^{-\Lf d(d-1)/2}$$
$$  \sum_{j} (j+1)2^{\Lf d(d-1)/2} \prob \bigg(\sloc^*(T^*_\Lf) \in D_j, \spath \in \cwpplus((T^*,v'),(T^*_\Lf,\cdot)\bigg) \leq 2  U(v')
2^{-\Lf d(d-1)/2}$$
and the maximum of $x \in D_i$ and $y \in D_j$ of the probability of 
$$ \scwpplus((T_\Lf,x),(n-T^*_\Lf,y)) \setminus \scwminus((T_l,\cdot),(n-T^*_l,\cdot))$$
is at most
$$C(i+1)2^{\Lf d(d-1)/2}(j+1)2^{\Lf d(d-1)}n^{-d(d-1)/2}  \cdot n^{-(d-1)/2}\cdot n^{-\delta}.$$
We sum up over $i$ and $j$ and use independence to get
that the probability of the event in (5) is at most 
$$  C U(v)U(v')n^{-d(d-1)/2}n^{-(d-1)/2}n^{-\delta}$$
which is small in comparison with the probability of $\cw((T,v),(n-T^*,v'))$
and thus in comparison with $\scwpplus((T,v),(n-T^*,v'))$.

The final inequality follows because 
$$  \cw((T,v),(n-T^*,v')) \subset  \scwpplus((T,v),(n-T^*,v'))  $$
\end{proof}

\begin{cor} \label{Anna}
For any $\epsilon>0$ there exists $K$ and $l$ such that if $v,v' \in \cone_{2^l}$ and
$T,T^* \leq (4.1)^l$ then there exists $K$ such that for any $n$ sufficiently large
$$D_K(\scwminus((T,v),(n-T^*,v')), \cw((T,v),(n-T^*,v')))<\epsilon.$$
\end{cor}

\begin{proof}
This follows from Lemma \ref{red hen} and Lemma \ref{spontini2} as follows.
By Lemma \ref{red hen} for any $\epsilon>0$ we can find $M$ such that 

$$D_M(\cw((T,v),(n-T^*,v')), \cw((0,0),(n,0)))<\epsilon.$$
By Lemma \ref{spontini2} for any $\epsilon>0$ we can find  $M'$  such that 
$$D_{M'}(\scwminus((T,v),(n-T^*,v')), \cw((0,0),(n,0)))<\epsilon.$$
Putting $K=\max\{M,M'\}$ and using the triangle inequality  proves the lemma.
\end{proof}

\section{Appendix: Traceless Hermitian Brownian Motion}
\label{dyson}

In this section we recount some elementary facts about traceless Hermitian Brownian motion an its connection to non-intersecting paths.  These results are well-known without the traceless condition and the transfer of the traceless case is straightforward.

Let $\{H_{ii}\}_{i=1}^{d}$ be standard Brownian motions and let $\{H_{ij}\}_{1\leq i <j \leq d}$ be independent standard complex Brownian motions.  For $j<i$, let $H_{ij} = \bar H_{ji}$.  Hermitian Brownian motion is the matrix valued process $H = (H_{ij})_{i,j=1}^d$.  Traceless Hermitian Brownian motion can then be constructed by projecting $H$ onto the space of matrices with trace equal to $0$.  In particular, if we define
\[ H^0 = H - \frac{Tr(H)}{d} I,\]
where $I$ is the identity matrix, then $H^0$ is a traceless Hermitian Brownian motion (independent of $Tr(H)$) whose distribution is the same as the distribution of $H$ conditioned to have trace equal to $0$ for all time.  Consequently, on the level of eigenvalues the projection onto traceless matrices results in each eigenvalue being shifted by the same amount.

The equivalent relation holds for $d$-dimensional Brownian motion conditioned on its coordinates summing to $0$.  Indeed, if we let $v_1,v_2,\dots, v_d$ be an orthonormal basis for $\R^d$ such that $v_1= d^{-1/2}(1,1,\dots, 1)$.  Then we can define the standard Brownian motion on $\R^d$ as
\[ B= \sum_{i=1}^d B_iv_i\]
where $B_1,\dots, B_d$ are independent, standard, one dimensional Brownian motions.  Conditioning the coordinates of $B$ to sum to $0$ is equivalent to projecting onto the subspace orthogonal to $v_1$, so that
\[ B^0=  B - B_1v_1= \sum_{i=2}^d B_iv_i\]
is distributed like $B$ conditioned on its coordinates summing to $0$.

Dyson \cite{Dyson62} found that, if we let $B_>$ be distributed like $B$ conditioned to remain in the cone 
\[\C_{>} = \left\{ x=(x_1,\dots,x_d)\in \R^d\ \middle| \ x_1> x_2> \cdots > x_d \right\},\]
for all time, which can be defined using an $h$-transform, then $B_>=_d \Lambda(H)$.  Let $B^0_>$ be distributed like $B^0$ conditioned to remain in $\C_{>}$ for all time.  Since subtracting $B_1(t)v_1$ does not change whether or not $B(t)\in \C_{>}$ and subtracting $\frac{Tr(H)}{d} I$ induces the same shift (in distribution) on the the eigenvalues of $H$ that subtracting $B_1v_1$ induces on the coordinates of $B$, it is easy to that we also have $B^0_>=_d \Lambda(H^0)$, see e.g.\ \cite[Proposition 2]{Biane09}.

To obtain our limiting object $Z$, it remains to condition $H^0$ to be $0$ at time $1$.  The easiest way to do this is by conditioning each entry to be $0$ at time $1$, which results in the definition of $Z$ we gave in terms of Brownian bridge.  Since all norms on finite dimensional spaces are equivalent, this is the same as conditioning the spectral norm of $H^0$ to be $0$ at time $1$.  In particular, for $G$ bounded and continuous we have
\[ \E G(Z) = \lim_{\epsilon\downarrow 0} \E\left[ G(H^0) \middle| |H^0(1)|<\epsilon\right].\]
Since $M\mapsto \Lambda(M)$ is continuous, using the identity $B^0_>=_d \Lambda(H^0)$, we see that for $F$ bounded and continuous we have that
\[\E F(\Lambda(Z))= \lim_{\epsilon\downarrow 0} \E\left[ F(\Lambda(H^0)) \middle| |H^0(1)|<\epsilon\right] =  \lim_{\epsilon\downarrow 0} \E\left[ F(B^0_>) \middle| |B_>^0(1)|<\epsilon\right].\]
This shows that $\Lambda(Z)$ has the same distribution as a bridge of $B^0_>$ from $0$ to $0$.  For later use, we remark that a straightforward computation using the Markov property and the transition densities for $B^0_>$ (see \cite{Biane09}) shows that if $F : D([0,1],\R)$ is bounded and continuous and $F(g)$ depends only on the restriction of $g$ to $[0,t]$ for some $0<t<1$, then there is a constant $C_t$ such that
\begin{equation}\label{Imhof} \begin{split} \E[ F(\Lambda(Z))] & =  \lim_{\epsilon\downarrow 0} \E\left[ F(B^0_>) \middle| |B_>^0(1)|<\epsilon\right] \\
& = C_t\E\left[ F(B^0_{>}(\cdot\wedge t)) e^{-|B^0_{>}(t)|^2/(2(1-t))} \right]
.\end{split}\end{equation}
This shows that for $0<t<1$, the law of $(\Lambda(Z_s), 0\leq s\leq t)$ is absolutely continuous with respect to the law of $(B^0_>(s), 0\leq s\leq t)$.

\begin{acks}[Acknowledgments]
We thank Jacopo Borga and Richard Kenyon for helpful conversations.
\end{acks}

\begin{funding}
This work was supported by NSF grants DMS-1712701, DMS-1855568, and DMS-1954059 as well as ERC Starting Grant 680275 MALIG
\end{funding}

\bibliographystyle{imsart-number} % Style BST file (imsart-number.bst or imsart-nameyear.bst)
\bibliography{pattern}       % Bibliography file (usually '*.bib')

\end{document}